\documentclass[10pt]{article}

\usepackage[margin=1in]{geometry}
\usepackage[latin1]{inputenc}
\usepackage[english]{babel}
\usepackage[babel]{csquotes}
\usepackage[noadjust]{cite}
\usepackage{amssymb, amsmath, amsthm, amsfonts, mathtools}
\usepackage[toc,page]{appendix}
\usepackage{enumerate}
\usepackage[shortlabels]{enumitem}
\usepackage[colorlinks=true, pdfstartview=FitV, linkcolor=colorLink, citecolor=colorCite, urlcolor=colorLink, linktocpage=true]{hyperref}
\usepackage{tikz, graphicx, pgfplots, floatrow}
\usetikzlibrary{backgrounds}
\usepackage{tocloft}
\usepackage{upgreek}
\usepackage{bbm}
\usepackage{longtable}
\usepackage{booktabs}
\usepackage{float}
\usepackage{caption}
\usepackage{listings}
\usepackage{mathrsfs}
\usepackage{xcolor}

\makeatletter
\renewcommand{\paragraph}{%
\@startsection{paragraph}{4}%
{\z@}{1.5ex \@plus 1.5ex \@minus .2ex}{-0.7em}%
{\normalfont\normalsize\bfseries}%
}
\makeatother

\setlist[itemize]{leftmargin=5mm}

\theoremstyle{plain}
\newtheorem{theorem}{Theorem}[section]

\newtheorem{lemma}[theorem]{Lemma}
\newtheorem{proposition}[theorem]{Proposition}

\newtheorem{problem}[theorem]{Problem}
\newtheorem{theoremA}{Theorem}[section]

\theoremstyle{definition}   
\newtheorem{definition}[theorem]{Definition}

\newtheorem{remark}[theorem]{Remark}
\newtheorem*{acknowledgements}{Acknowledgements}

\numberwithin{equation}{section}

\makeatletter
\newcommand\RedeclareMathOperator{%
\@ifstar{\def\rmo@s{m}\rmo@redeclare}{\def\rmo@s{o}\rmo@redeclare}%
}
\newcommand\rmo@redeclare[2]{%
\begingroup \escapechar\m@ne\xdef\@gtempa{{\string#1}}\endgroup
\expandafter\@ifundefined\@gtempa
{\@latex@error{\noexpand#1undefined}\@ehc}%
\relax
\expandafter\rmo@declmathop\rmo@s{#1}{#2}}
\newcommand\rmo@declmathop[3]{%
\DeclareRobustCommand{#2}{\qopname\newmcodes@#1{#3}}%
}
\@onlypreamble\RedeclareMathOperator
\makeatother

\def\R{\mathbb{R}}
\def\C{\mathbb{C}}
\def\S{\mathbb{S}}
\def\P{\mathbb{P}}
\def\Z{\mathbb{Z}}
\def\E{\mathbb{E}}
\def\N{\mathbb{N}}

\def\H{\mathcal{H}}

\def\CC{\mathcal{C}}
\def\OO{\mathcal{O}}
\def\FF{\mathcal{F}}

\def\CB{\mathcal{B}}

\def\GG{\mathcal{G}}

\def\TT{\mathcal{T}}
\def\SS{\mathcal{S}}

\def\CK{\mathcal{K}}

\def\BE{\mathbf{E}}
\def\BP{\mathbf{P}}
\def\BO{\mathbf{\Omega}}
\def\BF{\mathcal{\mathbf{F}}}

\def\rmR{\mathrm{R}}
\def\rmP{\mathrm{P}}
\def\rmF{\mathrm{F}}
\def\rmC{\mathrm{C}}
\def\rmB{\mathrm{B}}
\def\rmU{\mathrm{U}}
\def\rmA{\mathrm{A}}
\def\rmS{\mathrm{S}}
\def\rmX{\mathrm{X}}
\def\rmZ{\mathrm{Z}}

\def\rmG{\mathrm{G}}
\def\rmN{\mathrm{N}}
\def\rmI{\mathrm{I}}

\def\cd{\mathfrak{c}}
\def\frkm{\mathfrak{m}}

\def\frkb{\mathfrak{b}}

\def\frkp{\mathfrak{p}}

\def\frkd{\mathfrak{d}}
\def\frks{\mathfrak{s}}

\def\bh{\mathbf{h}}

\def\rmD{\mathrm{D}}

\def\rmE{\mathrm{E}}
\def\rmT{\mathrm{T}}
\def\rmm{\mathrm{m}}

\def\eps{\upvarepsilon}
\def\tau{\uptau}
\def\eta{\upeta}
\def\sigma{\upsigma}
\def\delta{\updelta}
\def\xi{\upxi}

\def\d{\mathrm{d}}
\def\theta{\upvartheta}
\def\zeta{\upzeta}
\def\rho{\uprho}
\def\omega{\upomega}
\def\gamma{\upgamma}
\def\mu{\upmu}
\def\nu{\upnu}
\def\alpha{\upalpha}
\def\beta{\upbeta}
\def\chi{\upchi}
\def\pi{\uppi}
\def\phi{\upphi}
\def\psi{\uppsi}

\renewcommand{\tilde}[1]{\widetilde{#1}}

\renewcommand{\bar}[1]{\overline{#1}}

\DeclareMathOperator\loc{loc}

\DeclareMathOperator\dTV{\d_{\mathrm{TV}}}
\DeclareMathOperator\dKL{\d_{\mathrm{KL}}}

\RedeclareMathOperator\Im{Im}
\RedeclareMathOperator\Re{Re}

\DeclareMathOperator\midpoint{mid}

\DeclareMathOperator\Cov{\mathbb{C}ov}
\RedeclareMathOperator\deg{deg}

\DeclareMathOperator\rad{rad}
\DeclareMathOperator\sph{sph}
\DeclareMathOperator\Vol{Vol}
\DeclareMathOperator\law{law}
\DeclarePairedDelimiter\abs\lvert\rvert

\newcommand{\eqdef}{\stackrel{\mbox{\tiny\rm def}}{=}}
\newcommand{\eqlaw}{\stackrel{\mbox{\tiny\rm law}}{=}}

\allowdisplaybreaks
\definecolor{colorLink}{RGB}{0,100,162}
\definecolor{colorCite}{RGB}{8,124,100}

\title{Liouville Brownian motion and quantum cones in dimension $d > 2$}
\author{
	\begin{tabular}{c} Federico Bertacco\\ [-5pt] \small Imperial College London \end{tabular}
	\begin{tabular}{c} \\[-5pt]\small  \end{tabular}
	\begin{tabular}{c} \\[-5pt]\small  \end{tabular}
	\begin{tabular}{c} Ewain Gwynne\\ [-5pt] \small University of Chicago \end{tabular}
	\begin{tabular}{c} \\[-5pt]\small  \end{tabular}
	\begin{tabular}{c} \\[-5pt]\small  \end{tabular}
}
\date{ }

\begin{document}	
\maketitle

\begin{abstract}
\noindent
For $d > 2$ and $\gamma \in (0, \sqrt{2d})$, we study the Liouville Brownian motion associated with the whole-space log-correlated Gaussian field in $\mathbb{R}^d$. We compute its spectral dimension, i.e., the short-time asymptotics of the heat kernel along the diagonal, which, in contrast to the two-dimensional case, depends on both $\gamma$ and on the thickness of the starting point. Furthermore, for even dimensions $d > 2$, we show that the spherical average process of the whole-space log-correlated Gaussian field in $\mathbb{R}^d$ can be identified with the integral of a stationary Gaussian Markov process of order $(d-2)/2$. Exploiting this representation, we construct the higher-dimensional analogue of the $\beta$-quantum cone for $\beta \in (-\infty, Q)$, with $Q = d/\gamma + \gamma/2$. Lastly, for $\alpha = Q - \sqrt{Q^2-4}$, we prove that the law of the $d$-dimensional $\alpha$-quantum cone is invariant under shifts along the trajectories of the associated Liouville Brownian motion.
\end{abstract}

\tableofcontents

\section{Introduction}
In the past couple of decades, substantial research has focused on random geometry constructed by exponentiating the two-dimensional \emph{Gaussian free field} (GFF), commonly known as \emph{Liouville quantum gravity} (LQG). For a comprehensive overview, see, for instance, the recent surveys \cite{BP24, Sheffield2023, EwainRev}. However, most results in this field remain specific to two dimensions and have not yet been extended to higher dimensions. This restriction largely arises from the critical role of conformal invariance in two-dimensional results, as non-trivial conformal maps do not exist in higher dimensions. Nonetheless, recent papers \cite{Cercle, DGZExp, Sturm1, Sturm2} have begun to investigate higher-dimensional analogs of LQG for dimensions $d > 2$.

Roughly speaking, in this paper, we are interested in studying a conformally flat Riemannian metric tensor on $\mathbb{R}^d$, for $d \geq 2$, of the following form
\begin{equation}
\label{eq:metricTensor}
e^{\gamma h}(dx_1^2 + \cdots + dx_d^2) \;,
\end{equation}
where $\gamma \in (0, \sqrt{\smash[b]{2d}})$ is a coupling parameter,  $dx_1^2 + \cdots + dx_d^2$ denotes the Euclidean metric tensor on $\mathbb{R}^d$, and $h$ is the whole-space \emph{log-correlated Gaussian field} (LGF) on $\mathbb{R}^d$, i.e., it is the centred Gaussian random generalised function defined up to a global additive constant and with covariance of the form 
\begin{equation*}
\E\bigl[(h, \phi) (h, \psi)\bigr] = - \int_{\R^d} \int_{\R^d} \log\abs{x-y} \phi(x) \psi(y) dx dy \;,
\end{equation*}
where $\phi$ and $\psi$ are smooth, compactly supported functions with zero average\footnote{For $d = 2$, the whole-space LGF is the same as the whole-plane GFF.}. 
In what follows, we denote by $\bh$ the whole-space LGF, with the additive constant chosen so that its spherical average over the unit sphere centred at the origin is zero (see Section~\ref{sub:spherical} for more details on the spherical average of the LGF).

As in the two-dimensional case, the Riemannian metric tensor in \eqref{eq:metricTensor} is not directly interpretable pointwise due to $h$ being a random generalised function. The simplest object to construct is the area measure $\mu_{\bh, \gamma}$, which arises as a limit as $\eps \to 0$ of $\smash{\eps^{\gamma^2/2} e^{\gamma \bh_{\eps}(x)} dx}$, where $dx$ denotes the Lebesgue measure and $(\bh_{\eps})_{\eps > 0}$ is a collection of regularising fields given by the spherical average approximation. This construction falls under the theory of \emph{Gaussian multiplicative chaos} (GMC) (see e.g.\ \cite{Kahane, RVReview, BP24, AruRev}).
We highlight that the measure $\mu_{\bh, \gamma}$ satisfies a certain coordinate change formula. Suppose that $\phi$ is a conformal automorphism of $\R^d$, then for any Borel subset $A \subseteq \R^d$, it holds almost surely that
\begin{equation}
\label{eq:defQd}
\mu_{\bh, \gamma}\bigl(\phi(A)\bigr) = \mu_{\bh \circ \phi + Q \log |\phi'|, \gamma}\bigl(A\bigr) \;,  \qquad  \qquad Q \eqdef \frac{d}{\gamma} + \frac{\gamma}{2} \;.
\end{equation}
Essentially, \eqref{eq:defQd} implies that we can regard $(\R^d, \bh)$ and $(\R^d, \bh \circ \phi + Q \log |\phi'|)$ as two different parameterisations of the same $d$-dimensional manifold, analogously to two-dimensional setting \cite[Proposition~2.1]{DS11}. 

We also emphasise that recently, in \cite{DGZExp}, the authors established the tightness of a natural approximation scheme for the distance function associated to the metric tensor \eqref{eq:metricTensor}. We also refer to \cite{DGZPerc} for some recent results on thick points of the LGF in higher dimensions. 

In this paper, we continue the extension of LQG to higher dimensions. In particular, we study the canonical diffusion associated with the metric tensor \eqref{eq:metricTensor} which is compatible with the coordinate transformation rule \eqref{eq:defQd}. This diffusion represents the higher-dimensional analog of the \emph{Liouville Brownian motion} (LBM) constructed in \cite{Ber_LBM, GRV_LBM}. We emphasise that the higher-dimensional analogue of the LBM has already been studied in \cite{Sturm1}. For a discussion of the main differences with our setting, we refer to Section~\ref{sub:outline}.

For $\gamma  \in (0, \sqrt{2d})$, the $d$-dimensional LBM is a continuous strong Markov process. As such, we can consider the associated semigroup, referred to as the \emph{Liouville semigroup}. In Proposition~\ref{pr:existenceLHK}, we establish that the Liouville semigroup has a density, called the \emph{Liouville heat kernel} (LHK), with respect to the underlying GMC measure. In Theorem~\ref{th_specDim} below, we compute the spectral dimension of the LBM, i.e., the short-time asymptotic of the LHK along the diagonal, which, unlike the two-dimensional case studied in \cite{SpecDim}, depends on both $\gamma$ and the thickness of the starting point.

Next, for even\footnote{We refer to Remark~\ref{rem:RestrictionEven} for an explanation of the restriction to even dimensions.} dimensions $d > 2$, $\gamma \in (0, \sqrt{2d})$, and $\beta \in (-\infty, Q)$, we construct the $d$-dimensional analogue of the $\beta$-\emph{quantum cone}, which is the field obtained by ``zooming in'' in an appropriate way near the origin of $\bh - \beta \log\abs{\cdot}$ (see Theorem~\ref{th:convQuantum}). To construct the $d$-dimensional quantum cone, we characterise the spherical average process of the LGF in even dimensions $d > 2$ as the integral of a stationary Gaussian Markov process of order $(d-2)/2$. More precisely, in Theorem~\ref{th:identSphere}, we show that the vector of derivatives up to order $(d-2)/2$ of the spherical average process is a solution to a $(d-2)/2$-dimensional Langevin equation. With this characterisation in hand, the construction of the $d$-dimensional quantum cone, differently from the two-dimensional case, is non-trivial and forms the content of Theorem~\ref{th:convRecentred} below. Having constructed the $d$-dimensional quantum cone, we then proceed to show in Theorem~\ref{th:invarianceShift} that for $\alpha = Q - \sqrt{Q^2-4}$, the law of the $\alpha$-quantum cone is invariant under shifting along the trajectories of the corresponding LBM.

\section{Main results}
We now present our main results. First, we identify the spherical average process in any arbitrary even dimension. Next, we provide the result regarding the spectral dimension of the LBM. Finally, we construct the quantum cone in an arbitrary even dimension and examine its interaction with the corresponding LBM.

Throughout the rest of the paper, we let $h$ be a whole-space LGF modulo an additive constant defined on a probability space $(\Omega, \FF, \P)$. We denote by $\bh$ the whole-space LGF $h$ with the additive constant chosen such that its spherical average over the unit sphere centred at the origin is zero. Additionally, we consider a $d$-dimensional Brownian motion $(\BO, \BF, (B_t)_{t \geq 0}, (\BF_{t})_{t \geq 0}, (\BP_x)_{x \in \R^d})$, which is independent of $h$.

\subsection{Identification of the spherical average in arbitrary even dimension}
As shown in Section~\ref{sub:spherical} below, the spherical average process of the LGF $h$ around a point $x \in \R^d$, denoted by $(h_r(x))_{r > 0}$, is well-defined modulo a global additive constant. 
Moreover, we will see that the process $\smash{(h_{r}(x))_{r \geq 0}}$ is differentiable $\cd_d$ times, where
\begin{equation}
\label{eq:defcdd}
\cd_d \eqdef \frac{d-2}{2} \;.
\end{equation}
We define the following stochastic processes 
\begin{equation}
\label{eq:defVecDer}
\rmS_t \eqdef h_{e^{-t}}(x) - h_1(x)\;, \qquad \mathbf{S}_t \eqdef \bigl(\rmS_t^{(1)}, \ldots, \rmS_t^{(\cd_d)}\bigr)^{\top} \;, \qquad \forall \, t \in \R \;,
\end{equation}
where, for $k \in \{1, \ldots, \cd_d\}$, the process $(\rmS^{(k)}_t)_{t \in \R}$ denotes the $k$-th derivative of $\rmS$. We note that $(\rmS_t)_{t \in \R}$ is well-defined not just modulo an additive constant. It is well-known that in dimension $d = 2$, the process $(\rmS_t)_{t \in \R}$ has the same law of a standard two-sided Brownian motion (see e.g.\ \cite[Theorem~1.59]{BP24}). In our first result, for any arbitrary even dimension $d > 2$, we identify the process $(\rmS_t)_{t \in \R}$ as the integral of a stationary Gaussian Markov process of order $\cd_d$. Before proceeding, we recall the following definition. 

\begin{definition}
\label{def:companion}
For $p \geq 1$ and $(a_i)_{i = 0}^{p-1} \subset \R$, a $p \times p$ matrix $\mathbf{A}$ is the Frobenius companion matrix of the polynomial $\smash{p(\lambda) = \lambda^p + a_{p-1} \lambda^{p-1} + \cdots + a_1 \lambda + a_0}$, if it is of the following form
\begin{equation*}
\mathbf{A} = \begin{pmatrix}
0 & 1 & 0 & 0 & \cdots & 0 \\
0 & 0 & 1 & 0 & \cdots & 0 \\
0 & 0 & 0 & 1 & \cdots & 0 \\
\vdots & \vdots & \vdots & \vdots & \ddots & \vdots \\
0 & 0 & 0 & 0 & \cdots & 1 \\
-a_0 & -a_1 & -a_2 & \cdots & -a_{p-2} & - a_{p-1}
\end{pmatrix} \;.
\end{equation*} 
\end{definition}
We observe that, by construction, the characteristic polynomial of the Frobenius companion matrix is exactly equal to $p(\cdot)$.

\begin{theoremA}[Identification of the spherical average process]
\label{th:identSphere}
For $d > 2$ even and any $x \in \R^d$, the spherical average process $(h_r(x))_{r > 0}$ is $\cd_d$-times differentiable, where we recall \eqref{eq:defcdd}.
Consider the processes $(\rmS_t)_{t \in \R}$ and $(\mathbf{S}_t)_{t \in \R}$ as defined in \eqref{eq:defVecDer}. Then the multivariate process $(\mathbf{S}_{t})_{t \in \R}$ is a stationary Gaussian Markov process and it is the unique stationary solution of the following $\cd_d$-dimensional Langevin equation,
\begin{equation}
\label{eq:SDEpOrderMat}
\mathrm{d} \mathbf{S}_t = \mathbf{A} \mathbf{S}_t \mathrm{d}t + \mathbf{b} \mathrm{d} \mathbf{B}_t \;,
\end{equation}
where $(\mathbf{B}_t)_{t \in R}$ denotes a $\cd_d$-dimensional two-sided Brownian motion, $\mathbf{b}$ is the $\cd_d$-dimensional vector given by $\smash{\mathbf{b} = (0, \ldots, 0, 2^{d/2-1}\Gamma(d/2))^{\top}}$, and $\mathbf{A}$ is the $\cd_d \times \cd_d$ Frobenius companion matrix of the polynomial
\begin{equation*}
p(\lambda) = \prod_{k = 1}^{\cd_d} (\lambda + (d-2k)) \;.
\end{equation*}
Furthermore, the process $(\rmS_t)_{t \in \R}$ admits the following representation
\begin{equation}
\label{eq:defIntOU}
\rmS_t \eqlaw (d-2) \int_{-\infty}^{0} e^{2s}(1-e^{2s})^{\frac{d-4}{2}} \bigl(B_{s+t} - B_{s}\bigr) ds \;, \qquad \forall \, t \in \R \;, 
\end{equation}
where $(B_s)_{s \in \R}$ is a standard two-sided Brownian motion.
\end{theoremA}

We observe that for $d > 2$ even, the representation \eqref{eq:defIntOU} of the spherical average process indicates that $(\rmS_{t})_{t \in \R}$ acts as a ``smoothed-out'' version of a two-sided Brownian motion $(B_t)_{t \in \R}$. Moreover, as expected, increasing the dimension $d$ enhances the smoothness of $(\rmS_t)_{t \in \R}$ (see Figure~\ref{fig:sph}).

\begin{figure}[h]
\centering
\includegraphics[scale=0.35]{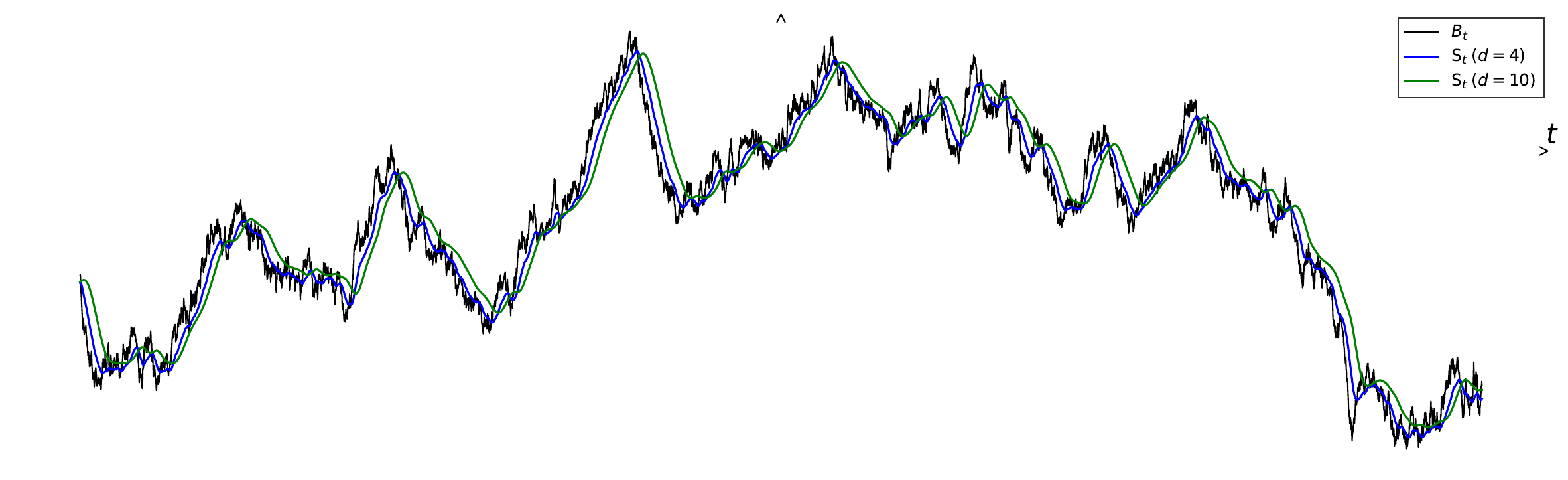}
\caption{\small The blue (resp.\ green) curve represents a simulation of the process $\smash{(\rmS_t)_{t \in \R}}$ in dimension $d = 4$ (resp.\ $d=10$), generated according to the representation given in \eqref{eq:defIntOU}. The black curve corresponds to the driving two-sided Brownian motion $(B_t)_{t \in \R}$.}
\label{fig:sph}
\end{figure}

\begin{remark}
We emphasise that if $d = 4$, then \eqref{eq:SDEpOrderMat} reduces to a one-dimensional SDE whose solution is the stationary \emph{Ornstein--Uhlenbeck} (OU) process introduced in \cite{OU}. For this reason, in the literature, solutions to the SDE of the form \eqref{eq:SDEpOrderMat} are usually called \emph{multivariate OU} processes or also \emph{continuous time autoregressive} processes \cite{GenOU}. 
\end{remark}

\begin{remark}
\label{rem:RestrictionEven}
The identification of the spherical average process is restricted to even dimensions since, in odd dimensions $d > 3$, the covariance kernel of the spherical average process does not admit a ``simple'' representation. At a high level, this is due to the fact that the whole-space LGF is the standard Gaussian on the homogeneous Sobolev space of order $d/2$, where the inner product is defined using the non-local operator $(-\Delta)^{d/2}$. We emphasise that, similarly, in \cite{Cercle, Sturm1}, the authors also restricted their analysis to even dimensions $d > 2$. In contrast, the construction of LBM and the tightness of approximations of the Riemannian distance function for \eqref{eq:metricTensor} in \cite{DGZExp} are not restricted to even dimensions, since these constructions are not reliant on any special exact solvability.
\end{remark}

\subsection{Spectral dimension of the LBM in arbitrary dimension}
The construction of LBM in dimensions $d > 2$ closely follows the approach used for its construction in dimension $d = 2$ \cite{Ber_LBM, GRV_LBM}, without introducing any significant conceptual difficulties. We also refer the reader to \cite{Sturm1} for a similar construction of the LBM in higher dimensions. 
For $\alpha \in (0, 2)$, we define the \emph{clock process} $(\rmF_{\bh, \alpha}(t))_{t \geq 0}$ by
\begin{equation}
\label{eq:clockLBM}
\rmF_{\bh, \alpha}(t) \eqdef \lim_{\eps \to 0} \int_{0}^t \eps^{\alpha^2/2} e^{\alpha \bh_{\eps}(B_s)} ds \;, \qquad \forall \, t \geq 0 \;. 
\end{equation} 
The convergence of the limit in the right-hand side of the above display falls within the scope of the theory of subcritical GMC measures, which guarantees that, for all $x \in \R^d$, the clock process is $\BP_x \otimes \P $-almost surely well-defined and non-trivial for all $\alpha \in (0, 2)$ (see Section~\ref{sub:LBMFixed} for more details). 

We then aim to define the LBM as 
\begin{equation*}
	\rmB_{\bh, \alpha, t} \eqdef B_{\rmF_{\bh, \alpha}^{-1}(t)} \;, \qquad \forall \, t \geq 0 \;.
\end{equation*}
Recall that we want the LBM to be compatible with the coordinate transformation rule \eqref{eq:defQd}. Specifically, as in the two-dimensional setting \cite[Theorem~1.4]{Ber_LBM}, if $\phi$ is a conformal automorphism of $\R^d$ that can be expressed as a composition of translations, dilations, and rotations, then it is desirable that 
\begin{equation}
\label{eq:conformalInvLBM1}
\bigl(\phi^{-1}(\rmB_{\bh, \alpha, t})\bigr)_{t \geq 0} \eqlaw \bigl(\bar{\rmB}_{\bh \circ \phi + Q \log \abs{\phi'}, \alpha, t})_{t \geq 0}\;,  \qquad \bar{\rmB}_{\bh \circ \phi + Q \log \abs{\phi'}, \alpha, t}	\eqdef \bar{B}_{\bar{\rmF}^{-1}_{\bh \circ \phi + Q \log \abs{\phi'}, \alpha}(t)} \;,
\end{equation}
with $(\bar{B}_t)_{t \geq 0}$ a Brownian motion on $\R^d$ and $\bar{\rmF}_{\bh \circ \phi + Q \log \abs{\phi}, \alpha}$ the associated clock process with respect to the field $\bh \circ \phi + Q \log \abs{\phi'}$. Thanks to the fact that in every dimension $d \geq 2$, Brownian motion exhibits two-dimensional scaling, we will show in Proposition~\ref{pr:conformalLBM} that for \eqref{eq:conformalInvLBM1} to hold, we need to impose that
\begin{equation}
\label{eq:relGammaAlpha}
\frac{2}{\alpha} + \frac{\alpha}{2} = Q \qquad \iff \qquad \alpha = \alpha(Q) = \alpha(d, \gamma) \eqdef Q - \sqrt{Q^2-4} \;,
\end{equation}
where $Q$ is defined in \eqref{eq:defQd}.
In particular, in dimension $d=2$, we have $\alpha = \gamma$, whereas in higher dimensions, $\alpha$ becomes a non-linear function of $\gamma$. We also introduce the critical threshold $\alpha_c = \alpha(\sqrt{2d})$ given by
\begin{equation}
\label{eq:alphaCrit}
	\alpha_c \eqdef \sqrt{\smash[b]{2d}} - \sqrt{2d-4} \;.
\end{equation}
We can now define the $d$-dimensional LBM associated to the whole-space LGF.
\begin{definition}
\label{def:LBM}
For $d \geq 2$ and $Q > 2$, the \emph{Liouville Brownian motion} associated to the $d$-dimensional LGF $\bh$ is the process defined as follows
\begin{equation*}
\rmB_{\bh, \alpha, t} \eqdef B_{\rmF_{\bh, \alpha}^{-1}(t)}\; \qquad \forall \, t \geq 0 \;,
\end{equation*}	
where $\alpha$ is related to $Q$ as specified in \eqref{eq:relGammaAlpha}, and $\rmF_{\bh, \alpha}$ is the clock process introduced in \eqref{eq:clockLBM}.
\end{definition}

We observe that, for all $Q > 2$, it holds that $\alpha \in (0, 2)$, ensuring that the clock process is both well-defined and non-trivial. In particular, the LBM with parameter $Q > \sqrt{\smash[b]{2d}}$ corresponds to the higher-dimensional analog of a \emph{subcritical} phase of LQG. When $Q = \sqrt{\smash[b]{2d}}$, it represents the analog of a \emph{critical} phase. Meanwhile, for $Q \in (2, \sqrt{\smash[b]{2d}})$, it corresponds to a \emph{supercritical} phase\footnote{In this phase, the coupling constant $\gamma$ is complex. It is expected that there is no area measure, and the corresponding metric exhibits infinite spikes, analogous to the supercritical LQG metric in two dimensions. For further details, see \cite[Problem~7.13]{DGZExp} and \cite{Super1, Super2, BGS_Super}.}.

To define the LBM as a Markov process, we need to establish that, $\P$-almost surely, the clock process $\rmF_{\bh, \alpha}$ can be defined for all starting points $x \in \R^d$ of the underlying Brownian motion. The placement of ``almost sure'' is critical here, as it presents challenges that differ fundamentally from constructing the process at a single fixed point. This fact will be explored further in Sections~\ref{sub:outline} and~\ref{sec:LMB}.
As we will discuss in more detail in Section~\ref{sub:outline}, in contrast to the two-dimensional case, the LBM can be constructed as a Markov process for all $Q > 2$, i.e., also in the critical and supercritical phases. However, for $Q \in (2, \sqrt{\smash[b]{2d}}]$, the LBM can only be defined for starting points outside a polar set of the $d$-dimensional Brownian motion.

In the subcritcal phase, i.e., $Q > \sqrt{\smash[b]{2d}}$, the transition semigroup $(\rmP_{\bh,\alpha, t})_{t \geq 0}$ associated with the LBM is absolutely continuous with respect to the GMC measure $\mu_{\bh, \alpha}$. More precisely, we have the following result which is proved in Section~\ref{sub:LBMMarkov}. 
\begin{proposition}
\label{pr:existenceLHK}
Let $d \geq 2$ and $Q > \sqrt{\smash[b]{2d}}$. Then the Liouville semigroup $(\rmP_{\bh, \alpha, t})_{t \geq 0}$ is $\P$-almost surely absolutely continuous with respect to $\mu_{\bh, \alpha}$, i.e., there exists a family of jointly measurable functions $(\frkp_{\bh, \alpha,t}(\cdot, \cdot))_{t \geq 0}$, called the Liouville heat kernel (LHK), such that 
\begin{equation*}
\rmP_{\bh,\alpha, t}[f](x) = \int_{\R^d} f(y) \frkp_{\bh,\alpha, t}(x, y) \mu_{\bh, \alpha}(dy) \;, \qquad \forall \, f \in \CB_b(\R^d, \R), \; \forall \, x \in \R^d, \;  \forall \, t \, \geq 0  \;,
\end{equation*}
where $\CB_b(\R^d, \R)$ denotes the collection of bounded Borel functions in $\R^d$.
\end{proposition}

\begin{remark}
We emphasise that, a priori, for any $x \in \R^d$ and $t \geq 0$, we only know that $\frkp_{\bh,\alpha, t}(x, y)$ is well-defined for $\mu_{\bh, \alpha}$-almost every $y \in \R^d$. This is due to the fact that $\frkp_{\bh,\alpha, t}(x, \cdot)$ is defined as the density with respect to $\mu_{\bh, \alpha}$ of the semigroup associated with $\rmB_{\bh, \alpha, t}$ started from $x$.
\end{remark}

For $x \in \R^d$, one would like to define the spectral dimension of the LBM at $x \in \R^d$ as follows,
\begin{equation}
\label{eq:defSpecDimDef}
\frkd_{\bh,\alpha}(x) = \lim_{t \to 0} \frac{2 \log \frkp_{\bh,\alpha, t}(x,x)}{\abs{\log t}} \;.
\end{equation}
However, at present, we lack any regularity results for the LHK in arbitrary dimension $d > 2$, which means the limit on the right-hand side of \eqref{eq:defSpecDimDef} may not be well-defined. 
To address this issue, following \cite{SpecDim}, instead of treating the LHK on the diagonal as a function, we consider it as a measure on $\R^{+}$,
and we focus on the mass of this measure in the neighbourhood of $t = 0$.
The key insight is that the mass at $0$ of this measure can be determined via the behaviour of its Laplace transform. In turn, the Laplace transforms of the LHK are well-defined along the diagonal for all $x \in \R^d$ and can be represented as functionals of the GMC along the paths of Brownian bridges.

In order to make this rigorous, we need to introduce some notation. We use the convention that, for $x$, $y \in \R^d$ and $t > 0$, under $\BP_{x, y, t}$, the process $(B_s)_{s \in [0, t]}$ is a $d$-dimensional Brownian bridge of length $t$ from $x$ to $y$\footnote{We emphasise that, $\P$-almost surely, the clock process $\rmF_{\bh, \alpha}$ can be defined for all starting and end points $x, y \in \R^d$ and for any lifespan $t > 0$ of the underlying Brownian bridge (see Theorem~\ref{th:conLBMBB} below).}. We also denote by $p_t(\cdot, \cdot)$ the heat kernel of the $d$-dimensional Brownian motion.

\begin{definition}
Let $d \geq 2$ and $Q > \sqrt{\smash[b]{2d}}$. For $x \in \R^d$, for $\mu_{\bh, \alpha}$-almost every $y \in \R^d$, and for all $\chi \geq 0$, we define the following quantity:
\begin{equation}
\label{eq:defMellin}
\tilde{\frkm}_{\bh, \alpha, \chi}(x, y) \eqdef \int_0^{\infty} t^{\chi} e^{-t} \frkp_{\bh, \alpha, t}(x, y) dt \;. 
\end{equation}
Moreover, for $x$, $y \in \R^d$ and for all $\chi \geq 0$, we introduce the following quantity: 
\begin{equation}
\label{eq:defMellinTrue}
\frkm_{\bh, \alpha, \chi}(x, y) \eqdef \int_0^{\infty} \BE_{x, y, t}\bigl[(\rmF_{\bh,\alpha}(t))^{\chi}e^{-\rmF_{\bh, \alpha}(t)}\bigr] p_t(x, y) dt \;,
\end{equation}
with the notational convention that $\frkm_{\bh, \alpha, \chi}(x, x) = \frkm_{\bh, \alpha, \chi}(x)$. 
\end{definition}

A key result for the computation of the spectral dimension is that, thanks to Lemma~\ref{lm:keyBB} below, we have that, $\P$-almost surely,
\begin{equation}
\label{eq:keyRelSpecMain}
\tilde{\frkm}_{\bh, \alpha, \chi}(x, y) = \frkm_{\bh, \alpha, \chi}(x, y) \;, \qquad \forall \, x \in \R^d, \; \mu_{\bh, \alpha}\text{-a.e.\ } y \in \R^d \;,
\end{equation}
where we emphasise that the quantity on the right-hand side of the above expression is well-defined for all $x$, $y \in \R^d$.
We recall that our focus is on the behaviour of the quantity $\tilde{\frkm}_{\bh, \alpha, \chi}$ along the diagonal. However, as noted above, $\tilde{\frkm}_{\bh, \alpha, \chi}(x, y)$ is only well-defined for all $x \in \R^d$ and for $\mu_{\bh, \alpha}$-almost every $y \in \R^d$. In \cite{SpecDim}, this issue is addressed by showing that the mapping $\R^d \times \R^d \ni (x, y) \mapsto \frkm_{\bh, \alpha, \chi}(x, y)$ is continuous along the diagonal for all $\chi \geq 0$. Consequently, thanks to \eqref{eq:keyRelSpecMain}, $\tilde{\frkm}_{\bh, \alpha, \chi}(x, y)$ also admits a limit as $y \to x$ for all $x \in \R^d$ and for all $\chi \geq 0$.

In contrast to the two-dimensional case, in dimension $d > 2$, we expect that the quantity $\frkm_{\bh, \alpha, \chi}(x, y)$ does not admit a limit as $y \to x$ for some $\chi \geq 0$ and $x \in \R^d$. Roughly speaking, this is due to the fact that as will see below, the quantity $\frkm_{\bh, \alpha, \chi}(x)$ is infinite (resp.\ finite) for $\chi \in [0, \bar{\chi})$ (resp. $\chi > \bar{\chi}$), for some $\bar\chi$ depending on the thickness of $x$ (see Definition~\ref{def:thickPoints} below). Hence, since the set of thick points is dense, this implies that in any neighbourhood of $x$, there is a dense set of points where $\frkm_{\bh, \alpha, \chi}(\cdot)$ is infinite for some $\chi > \bar{\chi}$. 

The above consideration implies that we need an alternative approach to define $\tilde \frkm_{\bh, \alpha, \chi}(x, y)$ for all $x, y \in \R^d$. Given \eqref{eq:keyRelSpecMain}, a natural choice is to directly extend the definition of $\tilde \frkm_{\bh, \alpha, \chi}(x, y)$ to every $x, y \in \R^d$ by setting it equal to $\frkm_{\bh, \alpha, \chi}(x, y)$ as defined in \eqref{eq:defMellinTrue}. Using this convention, we define the spectral dimension of the LBM as follows.
\begin{definition}
\label{def:specDim}
Let $d \geq 2$, $Q > \sqrt{\smash[b]{2d}}$, and $x \in \R^d$. If there exists $\bar \chi > 0$ such that 
\begin{equation}
\label{eq:keyRelSpec} 
\frkm_{\bh, \alpha, \chi}(x) = \infty, \; \text{ for } \chi \in [0,  \bar{\chi}) \qquad \text{ and } \qquad 
\frkm_{\bh, \alpha, \chi}(x)  < \infty, \;  \text{ for } \chi > \bar{\chi} \;,
\end{equation}  
then the spectral dimension of $\rmB_{\bh, \alpha}$ at $x$ is $\frkd_{\bh,\alpha}(x) = 2(\bar{\chi} + 1)$.
Alternatively, if it holds that 
\begin{equation}
\label{eq:keyRelSpecTwoDim} 
\frkm_{\bh, \alpha, \chi}(x) = \infty, \; \text{ for } \chi = 0 \qquad \text{ and } \qquad 
\frkm_{\bh, \alpha, \chi}(x)  < \infty, \;  \text{ for } \chi > 0 \;,
\end{equation}
then the spectral dimension of $\rmB_{\bh, \alpha}$ at $x$ is $\frkd_{\bh,\alpha}(x) = 2$.
\end{definition}

Instead of computing the spectral dimension of the LBM associated to $\bh$, we compute the spectral dimension at the origin of the LBM associated to the field $\bh - \beta \log\abs{\cdot}$\footnote{We refer to the beginning of Section~\ref{sub:spectral} for an explanation of why the LBM associated with the field $\bh - \beta \log\abs{\cdot}$ can be constructed.}, for $\beta \in (-\infty, Q)$. More precisely, we have the following result. 
\begin{theoremA}[Spectral dimension]
\label{th_specDim}
Let $d \geq 2$, $Q > \sqrt{\smash[b]{2d}}$, and $\beta \in (-\infty, Q)$. It holds $\P$-almost surely that 
\begin{equation}
\label{eq:expressionSpecDim}
\frkd_{\bh - \beta \log\abs{\cdot}, \alpha}(0) = 2 + \frac{2(d-2)}{2 + \alpha^2/2 - \alpha\beta}\;,
\end{equation}
where we recall that $\alpha$ is related to $Q$ as specified in \eqref{eq:relGammaAlpha}.
\end{theoremA}

We observe that for $d = 2$, the spectral dimension is equal to $2$, independently of the choice of $\alpha$ and $\beta$. Furthermore, in dimension $d > 2$, the spectral dimension $\smash{\frkd_{\bh - \beta \log\abs{\cdot}, \alpha}(0)}$ coincides with the underlying dimension $d$ when $\beta = \alpha/2$. In the limit $\alpha \to 0$ (which corresponds to the limit $\gamma \to 0$), it holds that $\smash{\frkd_{\bh - \beta \log\abs{\cdot}, \alpha}(0) = d}$ for all $\beta \in (-\infty, Q)$. Additionally, for a fixed dimension $d$ and a fixed value $Q > \sqrt{\smash[b]{2d}}$, the spectral dimension is an increasing function of the parameter $\beta$.

We now discuss the implications of Theorem~\ref{th_specDim} for the spectral dimension of the LBM around the thick points of $\bh$. To this end, we first recall the definition of thick points. For $x \in \R^d$, let $(\bh_{\eps}(x))_{\eps > 0}$ denote the spherical average approximation of the whole-space LGF around $x$, as introduced in Section~\ref{sub:spherical} below.
\begin{definition}
\label{def:thickPoints}
For $\beta \geq 0$, a point $x \in \R^d$ is said to be $\beta$-thick if the following limit holds
\begin{equation*}
\lim_{\eps \to 0} \frac{\bh_{\eps}(x)}{\abs{\log \eps}} = \beta\;.
\end{equation*}
We denote the set of $\beta$-thick points of $\bh$ by $\TT_{\beta}$.
\end{definition}

Thanks to \cite[Theorem~4.2]{RVReview}, we know that the Hausdorff dimension of $\TT_{\beta}$ is almost surely equal to $d - \beta^2/2$, for $\beta \in [0, \sqrt{\smash[b]{2d}}]$. Moreover, when $\beta > \sqrt{\smash[b]{2d}}$, it holds $\P$-almost surely that $\TT_{\beta} = \emptyset$. Furthermore, for $\beta \in [0, \sqrt{2d}]$, the GMC measure $\mu_{\bh, \beta}$\footnote{Here, we are tacitly assuming that for $\beta = \sqrt{2d}$, the measure $\mu_{\bh, \beta}$ corresponds to the critical GMC measure introduced in \cite{GMC_critical}.} can be interpreted as the ``uniform measure on $\beta$-thick points'' (see \cite[Theorem~4.1]{RVReview}). Specifically, if $x$ is a point sampled from this measure (restricted to some bounded open set), then near $x$, the field $\bh$ locally looks like $\bh - \beta \log\abs{\cdot - x}$. Consequently, the behaviour of the field $\bh - \beta \log\abs{\cdot}$ near the origin is similar to the behaviour of $\bh$ near a ``typical'' $\beta$-thick point. Therefore, thanks to Theorem~\ref{th_specDim}, we expect that if $x \in \TT_{\beta}$, then $\frkd_{\bh, \alpha}(x) = 2 + 2(d-2)/(2 + \alpha^2/2 - \alpha\beta)$. This fact is formalised in the following remark.

\begin{remark}
A stronger version of Theorem~\ref{th_specDim} would assert that $\P$-almost surely, for all $\beta \in [0, \sqrt{\smash[b]{2d}}]$ and for all $x \in \TT_{\beta}$, it holds that $\frkd_{\bh, \alpha}(x) = 2 + 2(d-2)/(2 + \alpha^2/2 - \alpha\beta)$. While we believe this result to be true, proving such a stronger version would involve overcoming some technical difficulties.
\end{remark}

\subsection{Quantum cones in arbitrary even dimension}
For $d \geq 2$ even, $Q > \sqrt{\smash[b]{2d}}$, and $\beta \in (-\infty , Q)$, the $\beta$-quantum cone is, roughly speaking, the limiting field obtained by ``zooming in'' appropriately near the origin on a whole-space LGF with a $\beta$-log singularity at the origin. 
As for the two dimensional case, we will define the $\beta$-quantum cone field by specifying, separately, its averages on spheres centred at the origin, and what is left when we subtract these. The second of these components will have exactly the same law as the corresponding projection of the whole-space LGF, i.e., it is only the radial part which is different.

Consider the whole-space LGF $\bh$. For $b > 0$ a large constant, applying the coordinate change formula to the field $\bh - \beta \log\abs{\cdot} + b$, with respect to the conformal automorphism $\R^d \ni x \mapsto e^{-t} x$, for some $t \in \R$, we obtain the field $\bh^{t, b}$ given by
\begin{equation}
\label{eq:fieldHtb}
\bh^{t, b} \eqdef \bh(e^{-t}\cdot) -\beta \log\abs{e^{-t}\cdot} - Q t + b \;.
\end{equation}  
Now, recalling Theorem~\ref{th:identSphere}, the spherical average process $(\bh^{t, b}_{e^{-s}}(0))_{s \in \R}$ admits the following representation 
\begin{equation*}
\bh^{t, b}_{e^{-s}}(0) = \rmS_{s+t} + \beta s - (Q - \beta)t  + b \;, \qquad \forall \, s \in \R \;.
\end{equation*} 
In particular, by taking $s = 0$ in the previous expression, we get that $\bh^{t, b}_{1}(0) = \rmS_{t} - (Q-\beta) t + b$. Since $\beta \in (-\infty , Q)$, the drift term is always negative and so in particular, for each $b > 0$, there exists $\P$-almost surely a smallest $\sigma_{b} \in \R$ such that $\bh^{t, b}_{1}(0) = 0$, i.e., 
\begin{equation}
\label{eq:defTb}
\sigma_{b} \eqdef \inf\bigl\{t \in \R \, : \,  \rmS_t - (Q - \beta) t = - b\bigr\} \;.
\end{equation}
For each $b > 0$, we define the recentred process $(\rmS_{b, s})_{s \in \R}$ by letting
\begin{equation}
\label{eq:defRecProc}
\rmS_{b, s} \eqdef \rmS_{s + \sigma_{b}} - \rmS_{\sigma_{b}} \;, \qquad \forall \, s \in \R \;.
\end{equation}
The radial part of the $\beta$-quantum cone field is constructed by combining the limit of the process $(\rmS_{b, s})_{s \in \R}$ as $b \to \infty$ with an additional drift component.
\begin{theoremA}[Construction of the radial part of the quantum cone]
\label{th:convRecentred}
Let $d > 2$ be even, $Q > \sqrt{\smash[b]{2d}}$, and $\beta \in (-\infty , Q)$. For any fixed $\rmT \in \R$, the process $\smash{(\rmS_{b, s})_{s \geq \rmT}}$ as defined in \eqref{eq:defRecProc} converges in total variation distance to a process $\smash{(\rmS_{\infty, s})_{s \geq \rmT}}$ as $b \to \infty$.
\end{theoremA}

In order to define the $d$-dimensional quantum cone, we need to observe that the homogeneous Sobolev space of order $d/2$, denote by $\H_{d}$ and introduced in Definition~\ref{def:HilbertSpace}, can be written as the direct sum of two subspaces. The first subspace is the closure of functions in $\SS(\R^d)$ that are radially symmetric about the origin, viewed modulo constants, and is denoted by $\H_{d, \rad}$. The second subspace consists of the closure of functions in $\SS(\R^d)$ that have mean zero on all spheres centred at the origin, and it is denoted by $\smash{\H_{d, \sph}}$ (see Lemma~\ref{lm:radialDeco} for more details).  
\begin{definition}
\label{def:quantumCone}
Let $d > 2$ be even, $\gamma \in (0, \sqrt{\smash[b]{2d}})$ and $\beta \in (-\infty , Q)$. 
The $d$-dimensional $\beta$-quantum cone field $\bh^{\star}_{\rad}$ is given by the sum $\smash{\bh^{\star} \eqdef \bh^{\star}_{\rad} + \bh^{\star}_{\sph}}$ where $\bh^{\star}_{\rad}$ is the function that is constant equal to $\rmS_{\infty, s} + \beta s$ on each sphere $\partial B(0, e^{-s})$ for $s \in \R$, and $\bh^{\star}_{\sph}$ is a field independent of $\bh^{\star}_{\rad}$ and that has the same law of the projection of the whole-space LGF onto $\smash{\H_{d, \sph}}$. 
\end{definition}

We emphasise that the above definition of the $d$-dimensional quantum cone field generalises the definition of the two-dimensional quantum cone introduced in \cite[Definition~4.10]{DMS21}. In particular, in what follows, for $d = 2$, we denote by $\bh^{\star}$ the quantum cone field introduced in \cite[Definition~4.10]{DMS21}.

We state the next result which identifies the $d$-dimensional $\beta$-quantum cone field as a local limit of a whole-space LGF with an additional $\beta$-log singularity at a fixed point of $\R^d$. This result is a straightforward consequence of our construction of the quantum cone field.

\begin{theoremA}[Quantum cone as a local limit]
\label{th:convQuantum}
Let $d \geq 2$ be even, $Q > \sqrt{\smash[b]{2d}}$, and $\beta \in (-\infty , Q)$. For $t \in \R$, $z \in \R^d$, and $b > 0$, consider the field $\bh^{t, b}$ defined as 
\begin{equation*}
	\bh^{t, b} \eqdef \bh(e^{-t} \cdot + z) - \beta \log\abs{e^{-t} \cdot} - Q t + b \;,
\end{equation*}
which is obtain by applying the coordinate change formula \eqref{eq:defQd} to the field $\bh - \beta \log\abs{\cdot - z}$ with respect to the conformal automorphism $\R^d \ni x \mapsto  e^{-t} x + z$. Furthermore, let $\sigma_{b}$ be such that 
\begin{equation}
\label{eq:defSigmaBTh}
\sigma_{b} \eqdef \inf\bigl\{t \in \R \, : \, \bh_1^{t, b}(0) = 0 \bigr\} \;.
\end{equation}
Then, for any $R > 0$, the field $\smash{\bh^{\sigma_{b}, b}}|_{B(0, R)}$ converges in total variation distance to ${\bh^{\star}}|_{B(0, R)}$ as $b \to \infty$.
\end{theoremA}

Finally, we show that the law of the $d$-dimensional quantum cone $\bh^{\star}$ is invariant in law modulo conformal coordinate change under shifting along the trajectories of the LBM. Before proceeding, we first define the LBM associated with the quantum cone field. Using the same notation as above, we recall that the process $(B_t)_{t \geq 0}$ is a $d$-dimensional Brownian motion independent of the quantum cone field $\bh^{\star}$. 

\begin{definition}
\label{def:LBMQuantum}
For $d \geq 2$ even, $Q > \sqrt{2d}$, and $\beta \in (-\infty , Q)$, the \emph{$d$-dimensional Liouville Brownian motion}  associated to the $d$-dimensional $\beta$-quantum cone field $\bh^{\star}$ is the process defined as follows
\begin{equation*}
\rmB_{\bh^{\star}\! , \alpha, t} \eqdef B_{\rmF_{\bh^{\star}\!, \alpha}^{-1}(t)}\; \qquad \forall \, t \geq 0 \;,
\end{equation*}	
where $\alpha$ is related to $Q$ as specified in \eqref{eq:relGammaAlpha}, and $\rmF_{\bh^{\star}\!, \alpha}$ is the clock process introduced in \eqref{eq:clockLBM} but with $\bh$ replaced by $\bh^{\star}$. 
\end{definition}

We emphasise that the well-posedness and non-triviality of the clock process $\rmF_{\bh^{\star}\!, \alpha}$ follows from the fact that the clock process associated with the whole-space LGF $\bh$ is well-defined, combined with the observation that $\bh$ and $\bh^{\star}$ can be coupled together so that $\bh - \bh^{\star}$ is a continuous function (except at the origin). The only non-trivial aspect is that, due to the log-singularity at the origin of the quantum cone field, there is a possibility that $\rmB_{\bh^{\star}\!, \alpha}$ started from the origin might become ``stuck''. This possibility, however, is ruled out at the beginning of Section~\ref{sub:LBMQuantum}.

With Definition~\ref{def:LBMQuantum} in hand, we can state our last main result.
\begin{theoremA}[Invariance under shifts along the LBM]
\label{th:invarianceShift}
Let $d \geq 2$ be even, $Q > \sqrt{2d}$, and let $\alpha$ be related to $Q$ as specified in \eqref{eq:relGammaAlpha}. Let $\bh^{\star}$ be a $d$-dimensional $\alpha$-quantum cone field and $(\rmB_{\bh^{\star}\!, \alpha, t})_{t \in \R}$ be a two-sided LBM with respect to $\bh^{\star}$ started from the origin at time $0$. Then, for all $t \in \R$, there exists a random $R > 0$ such that 
\begin{equation*}
\bigl(\bh^{\star}(\rmB_{\bh^{\star}\!, \alpha, t} + R \cdot) + Q \log R, R^{-1}(\rmB_{\bh^{\star}\!, \alpha, \cdot} - \rmB_{\bh^{\star}\!, \alpha, t})\bigr) \eqlaw (\bh^{\star}, \rmB_{\bh^{\star}\!, \alpha}) \;.
\end{equation*}
\end{theoremA}

Theorem~\ref{th:invarianceShift} indicates that for LBM on an $\alpha$-quantum cone, the environment as seen from the particle is stationary, modulo  spatial scaling. This property holds only for the specific value $\alpha = Q - \sqrt{Q^2 - 4}$. Consequently, it appears particularly natural to consider LBM starting from a point of thickness $\alpha$, where, according to Theorem~\ref{th_specDim}, the spectral dimension equals $(4d - 2\alpha^2) / (4 - \alpha^2)$.

We emphasise that, for $d=2$, Theorem~\ref{th:invarianceShift} was known to experts but, to the best of our knowledge, had not been written down before.

\subsection{Outline and main differences with the two-dimensional case}
\label{sub:outline}
After introducing the necessary background and preliminary material in Section~\ref{sec:preliminaries}, we proceed in Section~\ref{sec:identification} to prove Theorem~\ref{th:identSphere}, which concerns the identification of the spherical average process. To the best of our knowledge, this characterisation is new. 
The proof relies on the expression for the covariance of the spherical average of the LGF provided in \cite[Section~11.1]{FGFRev}, along with the observation that, in even dimensions $d > 2$, the vector of derivatives $(\mathbf{S}_t)_{t \in \R}$, as defined in \eqref{eq:defVecDer}, is a stationary Gaussian process. 
From this fact, we can infer that the multivariate process $(\mathbf{S}_t)_{t \in \R}$ must satisfy a suitable Langevin equation of the form \eqref{eq:SDEpOrderMat}. The precise identification of the coefficient matrix $\mathbf{A}$ and the diffusion vector $\mathbf{b}$ is achieved by comparing the power spectrum of the first derivative $\smash{(\rmS^{(1)}_t)_{t \in \R}}$ with that of solutions to SDEs of the form \eqref{eq:SDEpOrderMat}.

In Section~\ref{sec:LMB}, we present the construction of the $d$-dimensional LBM. First, we provide additional details on the construction of the LBM starting from a fixed point and prove its conformal invariance. Then, we focus on the construction of the LBM as a Markov process, building on the approach used for the two-dimensional case with minor technical adjustments. We also highlight that in \cite{Sturm1}, the authors construct the higher-dimensional analogue of the LBM on a general $d$-dimensional manifold. One of the main differences in our construction is the canonical identification of the appropriate exponent $\alpha \in (0, 2)$ required for the LBM to be conformally invariant in the sense of \eqref{eq:conformalInvLBM1}. Moreover, differently form \cite[Theorem~5.7]{Sturm1}, in the subcritical regime, we are able to provide an identification of the additive functional $\rmF_{\bh, \alpha}$ simultaneously from all staring points and not only for almost every point of $\R^d$. 

Furthermore, we identify a new phase transition in the higher-dimensional setting that does not arise when $d = 2$. Indeed, as discussed in more detail in Section~\ref{sec:LMB}, to define the LBM as a proper Markov process, we need to establish that, almost surely, the clock process $\rmF_{\bh, \alpha}$ can be defined for all starting points $x \in \R^d$ of the underlying Brownian motion. 
Roughly speaking, in dimension $d > 2$, this simultaneous convergence is possible if the measure $\mu_{\bh, \alpha}$ is regular enough to ensure that the mapping 
\begin{equation*}
B(0, R) \ni x \mapsto \int_{B(0, R)} \abs{x-y}^{2-d} \mu_{\bh, \alpha}(dy) 
\end{equation*}
is a continuous function of $x$ for all $R > 0$\footnote{This result is essentially a consequence of the theory of traces of Dirichlet forms developed in \cite{Fukushima_Symmetric}.}.
To analyse the continuity property of the above mapping, for $Q > 2$ and $\beta \in [0, \sqrt{\smash[b]{2d}}]$, suppose that $x \in B(0, R)$ is a $\beta$-thick point of $\bh$. Then, by using the regularity properties of the measure $\mu_{\bh, \alpha}$ (see Lemma~\ref{lm:scalingThick}), we have that
\begin{equation*}
\int_{B(0, R)} \abs{x-y}^{2-d} \mu_{\bh, \alpha}(dy) \approx 	\sum_{n \in \N_0} e^{-n(2 + \alpha^2/2 - \alpha \beta)} \;.
\end{equation*}
Therefore, in order to guarantee the convergence of the integral on the left-hand side of the above display, we need to require that
\begin{equation}
\label{eq:mustPotSatConv}
2+\alpha^2/2-\alpha\beta > 0 \;.	
\end{equation}

\begin{itemize}
\item In the subcritical case, i.e.\ $Q > \sqrt{\smash[b]{2d}}$, or equivalently $\alpha \in (0, \alpha_c)$, then condition \eqref{eq:mustPotSatConv} is satisfied for all $\beta \in [0, \sqrt{\smash[b]{2d}}]$. Hence, we can define the LBM starting simultaneously from all points of $\R^d$. 
\item In the critical case, i.e.\ $Q = \sqrt{\smash[b]{2d}}$, or equivalently $\alpha = \alpha_c$, the condition \eqref{eq:mustPotSatConv} is satisfied for all $\beta \in [0, \sqrt{\smash[b]{2d}})$. 
\item In the supercritical case, i.e.\ $Q \in (2,\sqrt{\smash[b]{2d}})$, or equivalently $\alpha \in (\alpha_c, 2)$, the condition \eqref{eq:mustPotSatConv} is satisfied only for $\beta \in [0, Q)$.
\end{itemize}
Recalling that the set of $\beta$-thick points $\TT_{\beta}$ has Hausdorff dimension $d - \beta^2 / 2$, we note that for all $\beta > 2$, the set $\TT_{\beta}$ has Hausdorff dimension strictly less than $d - 2$. Consequently, by \cite[Theorem~8.20]{Peres}, $\TT_{\beta}$ is a polar set for $d$-dimensional Brownian motion. Thus, the LBM will never hit points of thickness greater than $2$ unless it is started from one of them. Therefore, for all $Q > 2$, once the LBM is started, it will never get ``stuck''.

In Section~\ref{sub:spectral}, we prove Theorem~\ref{th_specDim}, which establishes the spectral dimension of the LBM for $Q > \sqrt{\smash[b]{2d}}$.   
Unlike the two-dimensional case, where the spectral dimension is everywhere equal to two, the higher-dimensional case requires new techniques. A central component of the proof is Lemma~\ref{lm:keySpec}, which studies how the integer moments of the conditional law given the underlying field $\bh$ of the clock process $\rmF_{\bh - \beta \log\abs{\cdot}, \alpha}(t)$ scale as $t \to 0$. Proving this lemma, detailed in Appendix~\ref{ap:keySpec}, requires careful control of the GMC measure of small balls centred near thick points of the LGF (see Lemma~\ref{lm:mainBoundTechNested}).

Finally, in Section~\ref{sec:quantum}, we prove Proposition~\ref{pr:MainSpec} and construct the $d$-dimensional quantum cone. We recall that in the two-dimensional case the construction of the quantum cone is straightforward and relies on the fact that the spherical average process $(\rmS_t)_{t \in \R}$ is a two-sided Brownian motion. As a result, the fact that the recentered process $\rmS_{b, \cdot}$, defined in \eqref{eq:defRecProc}, has a limit as $b \to \infty$ is essentially a consequence of the strong Markov property of $(\rmS_{t})_{t \in \R}$\footnote{See Remark~\ref{rm:QCd2} below for further details}. 
However, in even dimensions $d >2$, the proof of Proposition~\ref{pr:MainSpec} is non-trivial and is divided into two parts. First, by using the representation \eqref{eq:defIntOU} of the spherical average process, we show that the collection of processes $(\rmS_{b, \cdot})_{b > 0}$ is tight. Second, by leveraging the Markovianity of the multivariate process $(\rmS_t, \mathbf{S}_t)_{t \in \R}$ and a suitable mixing-type result, we prove that the limit as $b \to \infty$ of $(\rmS_{b, \cdot})_{b > 0}$ exists with respect to the total variation distance. The section concludes with the proofs of Theorems~\ref{th:convQuantum}~and~\ref{th:invarianceShift}.

\subsection{Some open problems}
\label{sub:open}
We conclude this section by listing some problems suggested by this work. 

\paragraph{Scaling limit of random walk on Voronoi tessellations}
For $d >2$, we recall that the Voronoi tessellations of a locally finite set of points $S \subset \R^d$, is the partition of $\R^d$ into cells $(\CC_s)_{s \in S}$ such that every point in the tile $\CC_s$ is closer to $s$ than any other point $s' \in S$, i.e., 
\begin{equation*}
\CC_s \eqdef \bigl\{x \in \R^d \, : \,  \abs{x-s}\leq \abs{x-s'} \;\; \forall \, s' \in S \setminus\{s\}\bigr\}.
\end{equation*}
For a bounded domain $D \subseteq \R^d$, $\gamma \in (0, \sqrt{\smash[b]{2d}})$, and each $n \in \N$, sample $n$ points from the measure $\mu_{\bh, \gamma}$ restricted to $D$ and normalised to be a probability measure. These points are then taken as the centres of the Voronoi cells. Two Voronoi cells are considered adjacent if they share a $(d-1)$-dimensional facet. We define a random walk on the adjacency graph of the Voronoi cells, with conductances given by $a(x, y) = \Vol_{d-1}(\CC_x \cap \CC_y) / \abs{x - y}$, where $\Vol_{d-1}$ represents the $(d-1)$-dimensional Lebesgue measure. In \cite[Proposition~1.1 and Remark~1.2]{AhmedEwainRW}, it is shown that this random walk converges to a time-changed Brownian motion.

Inspired by \cite{GB_LBM}, where the authors showed that the random walk on mated-CRT maps converges in law to the LBM with respect to the uniform topology, it would be interesting to establish a similar result in the setting described above.

\begin{problem}
Prove that the scaling limit of random walk on adjacency graph of the Voronoi cells with centres sampled from $d$-dimensional $\gamma$-GMC and with conductances as specified above is given by the LBM constructed with clock process $\rmF_{\bh, \alpha}$, with $\alpha$ related to $\gamma$ as specified in \eqref{eq:relGammaAlpha}.
\end{problem}

We emphasise that if the points in Voronoi tessellations are sampled according to the measure $\mu_{\bh, \gamma}$, then we expect the random walk to converge to the LBM constructed with clock process $\rmF_{\bh,\alpha}$.  One reason for this is that we expect that the scaling limit of the random walk should satisfy the same coordinate change rule as the measure $\mu_{\bh, \gamma}$ when we scale space, since if we preserve the measure $\mu_{\bh, \gamma}$ then we also preserve the Voronoi tessellations. 

\paragraph{Scaling exponent inside the exponential of the LHK}
In \cite{LHKScaling}, the authors studied the short-time asymptotic of the LHK in dimension $d = 2$ associated with a Dirichlet GFF on $[0, 1]^2$. They proved that, for all $x, y \in [0, 1]^2$, it holds that
\begin{equation*}
\lim_{t \to 0} \frac{\log\abs{\log \frkp_{\bh, \alpha, t}(x, y)}}{\abs{\log t}} = \frac{1}{d_{\gamma}- 1 } \;,	
\end{equation*}
where $d_{\gamma}$ denotes the fractal dimension of $\gamma$-LQG. This leads us to the following problem.
\begin{problem}
Obtain the correct scaling exponent inside the exponential of the LHK in arbitrary dimensions $d >2$. More precisely, for all $x$, $y \in \R^d$, compute the following limit,
\begin{equation*}
	\lim_{t \to 0} \frac{\log\abs{\log \frkp_{\bh, \alpha, t}(x, y)}}{\abs{\log t}}\;.
\end{equation*}
\end{problem}

\paragraph{Identification of the the spherical average process for the quantum cone}
In Theorem~\ref{th:convRecentred}, we established that the family of processes $(\rmS_{b, \cdot})_{b > 0}$ converges in total variation distance to a unique limiting process $(\rmS_{\infty, s})_{s \in \R}$. This limiting process is then used to define the radial part of the $d$-dimensional quantum cone. However, this approach is not entirely satisfactory, as it relies on a limiting procedure. 
\begin{problem}
Provide an explicit representation of the process $(\rmS_{\infty, s})_{s \in \R}$.
\end{problem}

\paragraph{Odd dimensions}
In this paper, we provide a characterisation of the spherical average process of the whole-space LGF and construct the $d$-dimensional quantum cone for even dimensions $d > 2$. In even dimensions, the fractional Laplacian $(-\Delta)^{d/2}$ is a local operator. However, in odd dimensions, it ceases to be local, leading to a series of complications. It would be interesting to extend our results to odd dimensions. 
\begin{problem}
Provide a characterisation of the spherical average process of the whole-space LGF and/or construct the $d$-dimensional quantum cone for odd dimensions $d > 2$.
\end{problem}

\begin{acknowledgements}
Part of this work was carried out during F.B.'s visit to the University of Chicago in March 2024, supported by the Doris Chen Mobility Award, and during the workshop ``Two-Dimensional Random Geometry'' at the Institute for Mathematical and Statistical Innovation (IMSI) in Chicago in July 2024. We thank M.\ Hairer and A. Bou-Rabee for helpful discussions. F.B. gratefully acknowledges financial support from the Royal Society through Prof.\ M.\ Hairer's Research Professorship Grant RP$\backslash$R1$\backslash$191065. E.G. was partially supported by NSF grant DMS-2245832.
\end{acknowledgements}

\section{Background and preliminaries}	
\label{sec:preliminaries}
In this section, we gather some preliminary results and definitions.
We begin by presenting the basic and recurrent notation that will be used throughout the paper. In Section~\ref{sub:LGF}, we rigorously define the whole-space LGF, and in Section~\ref{sub:spherical}, we introduce its spherical average approximation. In Section~\ref{sub:GMCChange}, we introduce the GMC measure and establish some of its main properties. Finally, in Section~\ref{sub:SDEs}, we record some results on multidimensional SDEs.

\subsection*{Basic and recurrent notation}
\label{sub:basic}
We write $\N = \{1,2, \ldots\}$ and $\N_0 = \{0, 1, 2, \ldots\}$. We let $\R^{+} = (0, \infty)$, $\R^{-} = (-\infty, 0)$, $\R^{+}_0 = [0, \infty)$, and $\R^{-}_0 = (-\infty, 0]$. Given $n \in \N$, we write $[n] = \{1, \ldots, n\}$ and $[n]_0 = \{0, 1, \ldots, n\}$. 
For $d \geq 2$, we denote by $\S^{d-1}$ the unit sphere centred at the origin in $\R^d$, and by $\abs{\S^{d-1}}$ its $(d-1)$-dimensional volume. 
We denote by $\CB(\R^d)$ the collection of Borel subsets of $\R^d$. We let $\CB(\R^d, \R)$ be the space of $\R$-valued Borel functions in $\R^d$.
For $a$, $b \in \R$, we write $a \lesssim b$ (resp.\ $a \gtrsim b$), if there exists a constant $c > 0$, independent of the quantities of interest, such that $a \leq c b$ (resp.\ $a \geq c b$). 
For $k \in \N_0$, we write $\smash{\CC^{k}_{\loc}(\R)}$ for the space of functions on $\R$ that are locally $k$-times continuously differentiable, equipped with the uniform local metric. We write $\dTV$ for the total variation distance between two measures.

\subsection{The whole-space log-correlated Gaussian field}
\label{sub:LGF}
For $d \geq 2$, we denote by $\SS(\R^d)$ the Schwartz space equipped with the usual topology. The space of tempered distribution $\SS'(\R^d)$ is then defined to be the space of continuous linear functionals from $\SS(\R^d)$ to $\C$. We define the Fourier transform $\FF[\phi]$ and inverse Fourier transform $\FF^{-1}[\phi]$ of a Schwartz function $\phi$ on $\R^d$ by
\begin{equation*}
	\FF[\phi](\xi) = \frac{1}{(2\pi)^{d/2}} \int_{\R^d} \phi(x) e^{-i \xi \cdot x} dx \;, \qquad \FF^{-1}[\phi](x) = \frac{1}{(2\pi)^{d/2}} \int_{\R^d} \phi(\xi) e^{i \xi \cdot x} d\xi \;,
\end{equation*}
where the prefactor $\smash{(2 \pi)^{-d/2}}$ is chosen so that that it makes the Fourier transform a unitary transformation on the complex function space $L^2(\R^d)$.
As a shorthand, we will also write $\widehat \phi$ in place of $\FF[\phi]$.  We define $\SS_0(\R^d) \subset \SS(\R^d)$ to be the set of Schwartz functions $\phi$ such that $\widehat \phi(0) = 0$, or equivalently,
\begin{equation*}
	\SS_0(\R^d) \eqdef \Biggl\{\phi \in \SS(\R^d) \, : \, \int_{\R^d} \phi(x) dx = 0\Biggr\} \;.
\end{equation*}
We equip $\SS_0(\R^d)$ with the topology inherited from $\SS(\R^d)$ and denote by $\SS'_0(\R^d)$ its topological dual. We note that each element of $\SS(\R^d)$ can be interpreted as an element of $\SS'_0(\R^d)$ by considering its equivalence class modulo global additive constant. 

\begin{definition}
\label{def:HilbertSpace}
The homogeneous Sobolev space of order $d/2$, denoted by $\H_d$, is the Hilbert space completion of the space of functions $\smash{f \in \SS(\R^d) \subset \SS'_0(\R^d)}$ such that $\smash{\xi \mapsto \abs{\xi}^{d/2} \widehat{f}(\xi) \in L^2(\R^d)}$ with respect to the following inner product\footnote{Here the choice of normalisation is such that the covariance kernel for the LGF is given by $-\log\abs{x-y}$ rather than a multiple of it.}
\begin{equation}
\label{eq:defInnerProd}
(f, g)_{d} \eqdef \frac{1}{2^{d-1} \pi^{d/2} \Gamma(d/2)} \int_{\R^d} \abs{\xi}^{d} \widehat{f}(\xi) \overline{\widehat{g}(\xi)} d \xi \;.
\end{equation}
\end{definition}

We emphasise that if $d \geq 2$ is even, then the inner product \eqref{eq:defInnerProd} can be conveniently rewritten as follows
\begin{equation*}
(f, g)_{d} = \bigl(f, (-\Delta)^{d/2} g\bigr) =  \bigl((-\Delta)^{d/2} f, g\bigr) \;,
\end{equation*}	
and in particular, if $d = 2$, then the above inner product coincides with the Dirichlet inner product. 

\begin{definition}
\label{def:LGFRd}
For $d \geq 2$, the whole-space LGF is a random element $h \in \SS_0'(\R^d)$, i.e., a random distribution modulo a global additive constant, whose law is centred Gaussian with the following covariance 
\begin{equation*}
\E\bigl[(h, \phi) (h, \psi)\bigr] = - \int_{\R^d} \int_{\R^d} \log\abs{x-y} \phi(x) \psi(y) dx dy \;, \qquad \forall \, \phi, \psi \in \SS_0(\R^d) \;.
\end{equation*}
Equivalently\footnote{We refer to \cite[Proposition~1.1]{LGFRev} for a proof of the equivalence of these two definitions.}, the whole-space LGF is the standard Gaussian on the Hilbert space $\H_d$, i.e., 
\begin{equation}
\label{eq:defLGFSeries}
h = \sum_{i = 1}^{\infty} \alpha_i f_i \;,
\end{equation}
where $\alpha_i$ are i.i.d.\ centred normal random variables and $(f_i)_{i = 1}^{\infty}$ is an orthonormal basis for the Hilbert space $\H_d$. The sum converges $\P$-almost surely in the space $\SS'_0(\R^d)$.
\end{definition}
  
An important feature of the whole-space LGF is that it is conformally invariant in every dimension $d \geq 2$.
\begin{lemma}[{\cite[Proposition~1.2]{LGFRev}}]
\label{lm:LGFconformal}
If $\phi$ is a conformal automorphism of $\R^d$, then the law of $h \circ \phi$ agrees with the law of $h$.
\end{lemma}

\begin{remark}
We recall that in dimension $d >2$, Liouville's theorem asserts that every conformal automorphism of $\R^d$ can be expressed as a composition of translations, dilations, rotations and inversions, i.e., they are M\"obius transformations. This is in contrast with what happens in dimension $d = 2$, where by the Riemann mapping theorem, all simply connected planar domains are conformally equivalent.
\end{remark}

We conclude this section with the following decomposition result, which will play a fundamental role in constructing the $d$-dimensional quantum cone.
\begin{lemma}[Radial decomposition]
\label{lm:radialDeco}
For $d \geq 2$, let $\smash{\H_{d, \rad}(\R^d)}$ be the subspace of $\smash{\H_d}$ obtained as the closure of functions in $\SS(\R^d)$ which are radially symmetric about the origin, viewed modulo constants. Let $\smash{\H_{d, \sph}}$ be the subspace of $\smash{\H}_d$ obtained as the closure of functions in $\SS(\R^d)$ which have mean zero about all spheres centred at the origin. Then, it holds that
\begin{equation*}
\H_d = \H_{d, \rad} \oplus \H_{d, \sph} \;.
\end{equation*}
\end{lemma}
\begin{proof}
We start by checking the orthogonality condition. Let $f_{\rad} \in \SS(\R^d)$ be a function which is radially symmetric about the origin, and $f_{\sph} \in \SS(\R^d)$ be a function which has mean zero about all spheres centred at the origin. We verify that the inner product $(f_{\rad}, f_{\sph})_d$ defined in \eqref{eq:defInnerProd} is equal to zero. To this end, we note that for all $\theta \in \S^{d-1}$ and $r > 0$, it holds that 
\begin{equation}
\label{eq:sphericalDeco2}
\widehat{f}_{\rad}(r \theta) = (2\pi)^{-\frac{d}{2}} \int_{\R^d} f_{\rad}(x)	e^{-i (r \theta) \cdot x} dx = (2\pi)^{-\frac{d}{2}} \int_{\R^d} f_{\rad}(x)	e^{-i (r e_1) \cdot x} dx = \widehat{f}_{\rad}(r e_1) \;,
\end{equation}
where $e_1$ denotes the first unit vector.
Moreover, for any $r > 0$, by letting $\phi_r:\R^d \to \R$ be the radial function defined as
\begin{equation*}
\phi_r(x) = \phi_r(\abs{x}) \eqdef \int_{\S^{d-1}} e^{i (r \theta) \cdot x} d \sigma(\theta) \;,	
\end{equation*}
thanks to Fubini's theorem, it holds that  
\begin{align}
\int_{\S^{d-1}} \overline{\widehat{f}_{\sph}(r \theta)} d \sigma(\theta)	 
& = (2\pi)^{-\frac{d}{2}} \int_{\S^{d-1}} \Biggl(\int_{\R^d} f_{\sph}(x) e^{i (r \theta) \cdot x} dx\Biggr) d \sigma(\theta)	\nonumber \\
& = (2\pi)^{-\frac{d}{2}} \int_{\R^d} f_{\sph}(x) \phi_r(\abs{x}) dx \nonumber \\
& = (2\pi)^{-\frac{d}{2}} \int_0^{\infty} s^{d-1} \phi_r(s)\Biggl( \int_{\S^{d-1}} f_{\sph}(s \theta) d \sigma(\theta) \Biggr) ds  \nonumber \\
& = 0 \;. \label{eq:sphericalDeco3}
\end{align}
Thus, by letting $a(d)$ denote the prefactor in the definition \eqref{eq:defInnerProd} of the inner product $(\cdot, \cdot)_d$, and by applying \eqref{eq:sphericalDeco2} and \eqref{eq:sphericalDeco3}, we obtain that:
\begin{align*}
(f_{\rad}, f_{\sph})_d = a(d) \int_{\R^d} \abs{\xi}^{d} \widehat{f}_{\rad}(\xi) \overline{\widehat{f}_{\sph}(\xi)} d \xi 
& = a(d) \int_{0}^{\infty} \int_{\S^{d-1}} r^{2d-1} \widehat{f}_{\rad}(r \theta) \overline{\widehat{f}_{\sph}(r \theta)} d r d \sigma(\theta) \\
& \overset{\eqref{eq:sphericalDeco2}}{=} a(d) \int_{0}^{\infty} r^{2d-1} \widehat{f}_{\rad}(r e_1) \Biggl(\int_{\S^{d-1}} \overline{\widehat{f}_{\sph}(r \theta)} d \sigma(\theta) \Biggr) d r \\
& \overset{\eqref{eq:sphericalDeco3}}{=} 0 \;. 
\end{align*}
By definition of $\smash{\H_{d, \rad}(\R^d)}$ and $\smash{\H_{d, \sph}(\R^d)}$, these two spaces are therefore orthogonal with respect to $(\cdot, \cdot)_d$.

We now check the spanning condition. Consider a function $\smash{f \in \SS(\R^d) \subset \SS'_0(\R^d)}$ such that $\xi \mapsto \abs{\xi}^{d/2} \widehat{f}(\xi) \in L^2(\R^d)$. Let $f_{\rad}$ be the function which on a given sphere $\partial B(0, r)$ is given by the average of $f$ on $\partial B(0,r)$. Let $f_{\sph} = f - f_{\rad}$. Then $f = f_{\rad} + f_{\sph}$, with $f_{\rad} \in \H_{d, \rad}(\R^d)$, and $f_{\sph} \in \H_{d, \sph}(\R^d)$. Now, if $f \in \H_{d}$, there exists a collection $(f_n)_{n \in \N} \subset \SS(\R^d)$ such that $\smash{\xi \mapsto \abs{\xi}^{d/2} \widehat{f_n}(\xi) \in L^2(\R^d)}$ for each $n \in \N$, and $f = \lim_{n \to \infty} f_n$ with respect to the norm induced by the inner product $(\cdot, \cdot)_d$. Moreover, by orthogonality of $\H_{d, \rad}$ and $\H_{d, \sph}$ the sequences $(f_{n, \rad})_{n \in \N}$ and  $(f_{n, \sph})_{n \in \N}$ are each Cauchy, and $\lim_{n \to \infty} f_{n, \rad} = f_{\rad} \in \H_{d, \rad}$ and $\lim_{n \to \infty} f_{n, \sph} = f_{\sph} \in \H_{d, \sph}$ with respect to the norm induced by the inner product $(\cdot, \cdot)_d$. Therefore, $f = f_{\rad} + f_{\sph}$ thus proving that $\H_d = \H_{d, \rad} \oplus\H_{d, \sph}$.
\end{proof}

\subsection{The spherical average}
\label{sub:spherical}
In this section, we introduce the spherical average process of the whole-space LGF. Most of the material presented here is drawn from \cite[Section~11]{FGFRev}. However, due to the significance of the spherical average in this paper, we provide a summary of the main concepts and results.

For a continuous function $f: \R^d \to \R$, we define its spherical average process around $x \in \R^d$ as the function $(0, \infty) \ni r \mapsto f_{r}(x)$ given by
\begin{equation*}
f_r(x) \eqdef \frac{1}{\abs{\S^{d-1}}} \int_{\S^{d-1}} f(x + r \sigma) d \sigma \;.
\end{equation*} 
A straightforward calculation shows that for all $\phi \in \CC_c^{\infty}(\R^{+})$ it holds that 
\begin{equation*}
\int_0^{\infty} f_r(x) \phi(r) dr = \frac{1}{\abs{\S^{d-1}}} \int_{\R^d} f(x) \frac{\phi(\abs{z-x})}{\abs{z-x}^{d-1}} dz \;.
\end{equation*}
This motivates the following definition.

\begin{definition}
For $h$ the whole-space LGF and $x \in \R^d$, we define the \emph{spherical average} around $x$ modulo a global additive constant $h_{\cdot}(x)$ by letting
\begin{equation*}
(h_{\cdot}(x), \phi) \eqdef \abs{\S^{d-1}}^{-1} \Biggl(h, z \mapsto \frac{\phi(\abs{z-x})}{\abs{z-x}^{d-1}}\Biggr) \;, \qquad \forall \phi \in \CC_c^{\infty}(\R^{+}) \cap \SS_0(\R)\;.
\end{equation*}
\end{definition}
We observe that if $\phi \in \CC_{c}^{\infty}(\R^{+}) \cap \SS_{0}(\R^d)$, then the mapping $z \mapsto \phi(\abs{z-x})/\abs{z-x}^{d-1}$ is in $\SS_0(\R^d)$ and so the previous definition makes sense. 
In order to find the covariance kernel of the spherical average process, we fix $\phi$, $\psi \in \CC_c^{\infty}(\R^{+}) \cap \SS_0(\R)$ and we observe that for $x_1$, $x_2 \in \R^d$, it holds that 
\begin{align*}
\E\bigl[(h_{\cdot}(x_1), \phi) (h_{\cdot}(x_2), \psi)\bigr] 
& = \frac{1}{\abs{\S^{d-1}}^{2}} \E\Biggl[\Biggl(h, z \mapsto \frac{\phi(\abs{z - x_1})}{\abs{z - x_1}^{d-1}}\Biggr)\Biggl(h, z \mapsto \frac{\psi(\abs{z-x_2})}{\abs{z-x_2}^{d-1}}\Biggr) \Biggr] \\
& = - \frac{1}{\abs{\S^{d-1}}^{2}} \int_{\R^d} \int_{\R^d} \log\abs{z_1 - z_2} \frac{\phi(\abs{z_1 - x_1}) \psi(\abs{z_2 - x_2})}{\abs{z_1 - x_1}^{d-1}\abs{z_2 -x_2}^{d-1}} dz_1 dz_2 \\
& = - \frac{1}{\abs{\S^{d-1}}^{2}} \int_{\R} \int_{\R} \Biggl(\int_{\S^{d-1}} \int_{\S^{d-1}} \log\abs{(x_1+r_1 \omega) - (x_2 + r_2 \sigma)} d\sigma d\omega\Biggr) \phi(r_1) \psi(r_2) d r_1 d r_2 \;,
\end{align*}
from which we can deduce that, for all $r_1$, $r_2 \in \R^{+}$,  
\begin{equation}
\label{eq:covarianceSphericalGen}
\CK_{r_1, r_2}(x_1, x_2) \eqdef 	\E\bigl[h_{r_1}(x_1) h_{r_2}(x_2)\bigr] = - \frac{1}{\abs{\S^{d-1}}^{2}}  \int_{\S^{d-1}}  \int_{\S^{d-1}} \log\abs{(x_1 + r_1 \omega) - (x_2 + r_2 \sigma)} d\sigma d\omega \;.
\end{equation}
In particular, if $x = x_1 = x_2 \in \R^d$, by a calculation based on spherical symmetries, we can further simplify the above integrals and obtain that 
\begin{equation}
\label{eq:covSphereDiag}
\CK_{r_1, r_2}(x, x) = -c(d) \int_{0}^{\pi} \log\bigl(r_1^2 + r_2^2 - 2 r_1 r_2 \cos \theta\bigr) \bigl(\sin \theta\bigr)^{d-2} d\theta \;, \qquad \forall \, r_1, r_2 \in \R^{+}\;,
\end{equation}
where, for $d \geq 2$, the constant $c(d) > 0$ is given by 
\begin{equation*}
c(d) \eqdef \frac{\Gamma(d/2)}{2 \sqrt{\pi} \Gamma((d-1)/2)} \;.
\end{equation*}
 
We conclude this subsection with the following result regarding the regularity of the spherical average process.
\begin{proposition}
\label{pr:SmothSpherical}
For $d \geq 2$ even and $x \in \R^d$, there exists a modification of $(h_{r}(x))_{r \geq 0}$ which is differentiable $\cd_d$ times, where we recall \eqref{eq:defcdd}. Moreover, the $\cd_d$-th derivative is $\delta$-H\"older continuous for all $\delta \in (0, 1/2)$. 
\end{proposition}
\begin{proof}
For $d \geq 2$ even and $x \in \R^d$, the fact that $(h_{r}(x))_{r \geq 0}$ is differentiable $\cd_d$ times is a consequence of Theorem~\ref{th:identSphere}. We refer to \cite[Proposition~11.1]{FGFRev} for a proof (based on Kolmogorov--Chentsov continuity theorem) of the fact that that the $\cd_d$-th derivative is $\delta$-H\"older continuous for all $\delta \in (0, 1/2)$.
\end{proof}

\begin{remark}
We remark that for $d \geq 2$ odd and $z \in \R^d$, it is expected that the spherical average process $(h_{r}(z))_{r \geq 0}$ is differentiable $(d-3)/2$ times.
\end{remark}

We record here a straightforward result regarding the variance of the spherical average process, the proof of which involves a simple computation based on \eqref{eq:covSphereDiag}.
\begin{lemma}
\label{lm:varianceSpherical}
For any $x \in \R^d$, it holds that 
\begin{equation*}
\E\bigl[\bigl(h_{e^{-t}}(x) - h_1(x)\bigr)^2\bigr] = \abs{t} + 2 c(d) \int_0^{\pi} \log\Biggl(\frac{1+e^{-2\abs{t}} - 2e^{-\abs{t}} \cos \theta}{2 - 2 \cos(\theta)}\Biggr) \bigl(\sin \theta\bigr)^{d-2} d \theta \;, \qquad \forall \, t \in \R\;.
\end{equation*}	
\end{lemma}

\subsection{Some properties of GMC measures}
\label{sub:GMCChange}
In what follows, for a point $x \in \R^d$, we denote by $(\bh_{\eps}(x))_{\eps > 0}$ the spherical average process of $\bh$ around $x$, i.e., $\bh_{\eps}(x) = h_{\eps}(x) - h_1(0)$.
For $\gamma \in (0, \sqrt{\smash[b]{2d}})$, we define the subcritical GMC measure associated to $\bh$ by letting
\begin{equation}
\label{eq:defLQGSphere}
\mu_{\bh, \gamma}(dx) \eqdef \lim_{\eps \to 0} \mu_{\bh, \gamma, \eps}(dx) \eqdef \lim_{\eps \to 0} \eps^{\frac{\gamma^2}{2}} e^{\gamma \bh_{\eps}(x)} dx \;,
\end{equation}
where $dx$ denotes the Lebesgue measure on $\R^d$. It is well-known (see e.g.\ \cite{Berestycki_Elementary, BP24}) that the limit on the right-hand side of the above display converges in $\P$-probability in the space of non-negative locally finite measures on $\R^d$ equipped with the topology of weak convergence. 

As we observed in Lemma~\ref{lm:LGFconformal}, the whole-space LGF is conformally invariant, and so it is natural to inquire how this invariance property is reflected in the associated GMC measures. It turns out that these measures are \emph{conformally covariant}, i.e., their laws are related under conformal mappings by a correction term that accounts for the domain's inflation caused by the mapping. 

\begin{proposition}[Coordinate change formula for $\mu_{\bh, \gamma}$]
\label{pr:coordinate_change_measure}
For $d > 2$ and $\gamma \in (0, \sqrt{\smash[b]{2d}})$, consider a conformal automorphism $\phi$ of $\R^d$ which can be expressed as a composition of translations, dilations and rotations. Then, it holds $\P$-almost surely that 
\begin{equation}
\label{eq:conCovLQG}
\mu_{\bh, \gamma}(\phi(A)) = \mu_{\bh \circ \phi + Q \log |\phi'|, \gamma}(A) \;, \qquad \forall \, A \in \CB(\R^d) \;,
\end{equation}
where $Q$ is defined as in \eqref{eq:defQd}.
\end{proposition}
\begin{proof}
Let $c > 0$ and consider the field $\bh^c \eqdef \bh(c \,\cdot)$. For any $\eps > 0$ and $x \in \R^d$, we have the following relation between the spherical averages: $\bh^c_{\eps}(x) = \bh_{c \eps}(c x)$. For $A \in \CB(\R^d)$, we have that 
\begin{equation*}
\mu_{\bh^c \!, \gamma}(A) 
= \lim_{\eps \to 0} \eps^{\gamma^2/2} \int_A e^{\gamma \bh^c_{\eps}(x)} dx
= \lim_{\eps \to 0} \eps^{\gamma^2/2} \int_A e^{\gamma \bh_{c \eps}(c x)} dx 
= \lim_{\delta \to 0} (\delta/c)^{\gamma^2/2} c^{-d} \int_{c A} e^{\gamma \bh_{\delta}(x)} dx \;,
\end{equation*}
where in the last equality we set $\delta = c \eps$ and we did the change of variables $x \mapsto x/c$. We note that the term on the right-hand side of the above display equals $c^{-d - \gamma^2/2} \mu_{\bh, \gamma}(c A)$, and so we have that 
\begin{equation*}
c^{d + \gamma^2/2}  \mu_{\bh^c\!, \gamma}(A)  =  \mu_{\bh, \gamma}(c A) \qquad \iff \qquad \mu_{\bh^c  + Q \log c, \gamma}(A) = \mu_{\bh, \gamma}(c A)\;.
\end{equation*}
The same argument also works if instead of just a scaling we have a composition of a scaling, rotation, and translation.
\end{proof}

\subsection{Some results on multidimensional SDEs}
\label{sub:SDEs}
In this section, we collect some standard results on SDEs that will play a key role in our analysis, especially in Section~\ref{sec:identification} for the proof of Theorem~\ref{th:identSphere}.

For $p \geq 1$ and $(a_i)_{i = 0}^{p-1} \subset \R$, we let $\mathbf{A}$ be the Frobenius companion matrix of the polynomial $p(\lambda) = \lambda^p + a_{p-1} \lambda^{p-1} + \cdots + a_1 \lambda + a_0$ as specified in Definition~\ref{def:companion}. Furthermore, for some $b \in \R$, we consider the $p$-dimensional vector $\smash{\mathbf{b} = (0, 0, \ldots, 0, b)^{\top}}$. We are interested in $p$-dimensional Langevin equations of the following form,
\begin{equation}
\label{eq:SDEpOrderMatGen}
\mathrm{d} \mathbf{X}_t = \mathbf{A} \mathbf{X}_t \mathrm{d}t + \mathbf{b} \mathrm{d} \mathbf{B}_t \;,
\end{equation}
where $(\mathbf{B}_t)_{t \in \R}$ denotes a $p$-dimensional two-sided Brownian motion and $\smash{\mathbf{X}_t = (\rmX_t, \rmX_t^{(1)}, \ldots, \rmX_t^{(p-2)}, \rmX_t^{(p-1)})^{\top}}$.

\begin{remark}
Roughly speaking, one should have in mind that \eqref{eq:SDEpOrderMatGen} is the systems of first order SDEs associated to the following $p$-th order SDE,
\begin{equation*}
\rmX^{(p)}_t + a_{p-1} \rmX^{(p-1)}_t + \cdots + a_{1} \rmX^{(1)}_t + a_0 \rmX_t = b \xi_t \;,
\end{equation*}
where $(\xi_t)_{t \in \R}$ is a white noise. 
\end{remark}

\begin{lemma}[{\cite[Theorem~5.6.7]{KSBook}}]
\label{lm:stationaryGaussianMarkov}
Consider the setting described above. If all the eigenvalues of the matrix $\mathbf{A}$ are negative real numbers, then the SDE \eqref{eq:SDEpOrderMatGen} admits a unique stationary Gaussian Markov solution $(\mathbf{X}_t)_{t \in \R}$.	
\end{lemma}

Before stating the next result, we recall that the power spectrum of a stationary Gaussian process $(\rmX_t)_{t \in \R}$ is defined as the Fourier transform of its covariance kernel.
\begin{lemma}
\label{lm:power}
Consider the setting described above and assume that all the eigenvalues of the matrix $\mathbf{A}$ are negative real numbers.
Let $\smash{(\mathbf{X}_t)_{t \in \R}}$ be the unique stationary Gaussian Markov solution to the SDE \eqref{eq:SDEpOrderMatGen}. Then, the power spectrum of the first component $\smash{(\rmX_t)_{t \in \R}}$ is given by
\begin{equation*}
\widehat{\CK}(\omega) = \frac{1}{\sqrt{2 \pi}} \frac{b^2}{[\sum_{k = 0}^{p} a_k (i \omega)^k] [\sum_{k = 0}^{p} a_k (-i \omega)^k]} \;, \qquad \forall \, \omega \in \R \;,
\end{equation*}
where $a_p = 1$.
Moreover, if we assume that $\smash{(-\lambda_k)_{k = 1}^p \subset \R^{-}}$ are the distinct eigenvalues of the matrix $\mathbf{A}$, then the process $\smash{(\rmX_t)_{t \in \R}}$ admits the following representations 
\begin{equation*}
\rmX_t = b \sum_{k=1}^{p} c_k \int_{-\infty}^{t} e^{-\lambda_k (t-s)} dB_s	\;, \qquad \forall \, t \in \R \;,
\end{equation*}
where, for each $k \in [p]$, we set $c_k \eqdef (\prod_{j = 1, \, j \neq k}^p (\lambda_j - \lambda_k))^{-1}$.	
\end{lemma}
\begin{proof}
These facts are standard and we refer to \cite[Section~10.3]{Pap} for a proof.	
\end{proof}

We finish this section with a mixing-type result that will be used in the construction of the quantum cone in Section~\ref{sec:quantum}. We consider the augmented $(p+1)$-dimensional SDE of the following form,   
\begin{equation}
\label{eq:SDEpOrderMatTot}
d (\bar \rmX_t, \mathbf{X}_t)^{\top}
= 
\begin{pmatrix}
0 & (1, 0, \ldots, 0) \\
\mathbf{0} & \mathbf{A}
\end{pmatrix} 
(\bar \rmX_t, \mathbf{X}_t)^{\top} dt + 
\begin{pmatrix}
0 \\
\mathbf{b}
\end{pmatrix} 
d\mathbf{B}_t \;, \qquad (\bar \rmX_0, \mathbf{X}_0)^{\top} = \mathbf{x}\;,
\end{equation}
where $\mathbf{A}$ and $\mathbf{b}$ are as specified above, $(\mathbf{B}_t)_{t \geq 0}$ is a $(p+1)$-dimensional Brownian motion, and $\mathbf{x} = (\bar x, x_0, x_1, \ldots, x_{p-1}) \in \R^{p+1}$. In other words, $(\mathbf{X}_t)_{t \geq 0}$ is a solution to the Langevin equation \eqref{eq:SDEpOrderMatGen} with initial condition at time $0$ given by $(x_0, x_1, \ldots, x_{p-1})$, and $(\bar \rmX_t)_{t \geq 0}$ is the integral of $(\rmX_t)_{t \geq 0}$ starting from $\bar x$.

\begin{lemma}
\label{lm:TVdistannceSpherical}
Consider the setting described above and assume that all the eigenvalues of the matrix $\mathbf{A}$ are distinct and negative real numbers.
For $\smash{\mathbf{x} = (\bar x, x_0, x_1, \ldots, x_{p-1})}$, $\smash{\mathbf{y} = (\bar y, y_0, \ldots, y_{p-1}) \in \mathbb{R}^{p + 1}}$, let $\smash{(\bar \rmX_{\mathbf{x}, t}, \mathbf{X}_{\mathbf{x}, t})_{t \geq 0}}$ and $\smash{(\bar \rmX_{\mathbf{y}, t}, \mathbf{X}_{\mathbf{y}, t})_{t \geq 0}}$ be the solutions to the SDE \eqref{eq:SDEpOrderMatTot} with initial conditions $\mathbf{x}$ and $\mathbf{y}$, respectively. Then, for any $R > 0$, it holds that 
\begin{equation*}
\lim_{t \to \infty} \sup_{\mathbf{x}, \mathbf{y} \in B(0, R)} \dTV\bigl[(\bar \rmX_{\mathbf{x}, t}, \mathbf{X}_{\mathbf{x}, t}), (\bar \rmX_{\mathbf{y}, t}, \mathbf{X}_{\mathbf{y}, t})\bigr] = 0 \;.
\end{equation*}
\end{lemma}
\begin{proof}
Fix $R > 0$ and let $\smash{\mathbf{x}, \mathbf{y} \in B(0, R)}$. To establish the result, we use Pinsker's inequality \cite[Lemma~2.5]{Pinsker}, which states that the total variation distance between two probability distributions is bounded above by the square root of their \emph{Kullback--Leibler} (KL) divergence. Specifically, we have that   
\begin{equation}
\label{eq:Pinsker}
\dTV\bigl[(\bar \rmX_{\mathbf{x}, t}, \mathbf{X}_{\mathbf{x}, t}), (\bar \rmX_{\mathbf{y}, t}, \mathbf{X}_{\mathbf{y}, t})\bigr] \leq \sqrt{\frac{1}{2} \dKL\bigl[(\bar \rmX_{\mathbf{x}, t}, \mathbf{X}_{\mathbf{x}, t}) || (\bar \rmX_{\mathbf{y}, t}, \mathbf{X}_{\mathbf{y}, t})\bigr]} \;.
\end{equation}
We observe that, for each $t \geq 0$, $(\bar \rmX_{\mathbf{x}, t}, \mathbf{X}_{\mathbf{x}, t})$ and $(\bar \rmX_{\mathbf{y}, t}, \mathbf{X}_{\mathbf{y}, t})$ are multivariate Gaussians with different means but the same covariance matrix. This follows directly from the fact that 
\begin{equation}
\label{eq:repAug}
(\bar \rmX_{\mathbf{x}, t}, \mathbf{X}_{\mathbf{x}, t}) = e^{\bar{\mathbf{A}} t} \mathbf{x} + \int_0^t e^{\bar{\mathbf{A}} (t-s)} \, \bar{\mathbf{b}} \, d B_s \;, \qquad \forall \, t \geq 0 \;,	
\end{equation} 
where $\bar{\mathbf{A}}$ and $\bar{\mathbf{b}}$ denote the augmented drift matrix and diffusion vector in the SDE \eqref{eq:SDEpOrderMatTot}. Naturally, a similar representation holds for $(\bar \rmX_{\mathbf{y}, t}, \mathbf{X}_{\mathbf{y}, t})_{t \geq 0}$. Therefore, the KL divergence between the law of $(\bar \rmX_{\mathbf{x}, t}, \mathbf{X}_{\mathbf{x}, t})_{t \geq 0}$ and the law of $(\bar \rmX_{\mathbf{y}, t}, \mathbf{X}_{\mathbf{y}, t})_{t \geq 0}$ is simply given by\footnote{We observe that the matrix $(\bar{\mathbf{b}}, \bar{\mathbf{A}} \, \bar{\mathbf{b}}, \bar{\mathbf{A}}^2 \, \bar{\mathbf{b}}, \ldots, \bar{\mathbf{A}}^p \, \bar{\mathbf{b}})$ has full rank, and so, thanks to Kalman rank condition (see e.g.\ \cite[Theorem~1.16]{Coron}), the inverse of the covariance matrix $\Sigma_t$ is well-defined for all $t > 0$.}
\begin{equation}
\label{eq:CameronMartin}
\dKL\bigl[(\bar \rmX_{\mathbf{x}, t}, \mathbf{X}_{\mathbf{x}, t}) || (\bar \rmX_{\mathbf{y}, t}, \mathbf{X}_{\mathbf{y}, t})\bigr] = \frac{1}{2} (\rmm_{\mathbf{x}, t} - \rmm_{\mathbf{y}, t}) \Sigma_t^{-1} (\rmm_{\mathbf{x}, t} - \rmm_{\mathbf{y}, t})^{\top} \;, 
\end{equation}
where here we set 
\begin{equation*}
\rmm_{\mathbf{x}, t} \eqdef \E\bigl[(\bar \rmX_{\mathbf{x}, t}, \mathbf{X}_{\mathbf{x}, t})\bigr] \;, \qquad \rmm_{\mathbf{y}, t} \eqdef \E\bigl[(\bar \rmX_{\mathbf{y}, t}, \mathbf{X}_{\mathbf{y}, t})\bigr]  \;, \qquad \Sigma_t \eqdef \Cov\bigl((\bar \rmX_{\mathbf{x}, t}, \mathbf{X}_{\mathbf{x}, t})\bigr) \;.
\end{equation*}

Thanks to \eqref{eq:Pinsker}, it suffices to prove that the right-hand side of \eqref{eq:CameronMartin} converges to zero as $t \to \infty$. Using the representation \eqref{eq:repAug} and applying It\^o's isometry, we have that 
 \begin{equation*}
\rmm_{\mathbf{x}, t} = e^{\bar{\mathbf{A}} t} \mathbf{x} \;, \qquad \rmm_{\mathbf{y}, t} = e^{\bar{\mathbf{A}} t} \mathbf{y} \;, \qquad \Sigma_t = \int_0^t (e^{\bar{\mathbf{A}} s} \bar{\mathbf{b}}) (e^{\bar{\mathbf{A}} s} \bar{\mathbf{b}})^{\top} ds \;.
\end{equation*}
For all $t \geq 0$, due to the block-upper triangular structure of the matrix $\bar{\mathbf{A}}$, we have that 
\begin{equation}
\label{eq:repExpBarA}
e^{\bar{\mathbf{A}} t} = \begin{pmatrix}
1 & \int_0^t (e^{\mathbf{A} s} e_1)^{\top} ds \\
\mathbf{0} & e^{\mathbf{A}t} 
\end{pmatrix} \;.
\end{equation}
Now, let $\mathbf{D} = \text{diag}(-\lambda_1, \ldots, -\lambda_p)$ denote the diagonal matrix of the negative eigenvalues of $\mathbf{A}$, and let $\mathbf{V}$ represent the matrix of the corresponding eigenvectors. We then have that 
\begin{equation*}
e^{\mathbf{A}t} = \mathbf{V} e^{\mathbf{D} t} \mathbf{V}^{-1} = \sum_{k = 1}^p e^{-\lambda_k t} \mathbf{v}_k \mathbf{v}_k^{-1}\;,
\end{equation*}
where $\mathbf{v}_k$ and $\mathbf{v}_k^{-1}$ denote the $k$-th column of $\mathbf{V}$ and the $k$-th row of $\mathbf{V}^{-1}$, respectively. Thus, each entry of $e^{\mathbf{A} t}$ can be expressed as a linear combination of the terms $(e^{-\lambda_k t})_{k = 1}^p$, which, for brevity, we denote by $\langle e^{-\lambda_k t}\rangle_{k = 1}^p$. In particular, this observation implies that each entry of the vector in the top-right position of \eqref{eq:repExpBarA} is given by an $\OO(1)$ term plus a term of the form $\langle e^{-\lambda_k t}\rangle_{k = 1}^p$. Furthermore, recalling that only the last entry of $\bar{\mathbf{b}}$ is non-zero, we deduce that 
\begin{equation*}
e^{\bar{\mathbf{A}} s} \bar{\mathbf{b}} = \bigl(\OO(1) + \langle e^{-\lambda_k s} \rangle_{k=1}^p, \langle e^{-\lambda_k s} \rangle_{k=1}^p, \ldots, \langle e^{-\lambda_k s} \rangle_{k=1}^p \bigr)^{\top} \;.	
\end{equation*}
Combining these considerations, we obtain the following two facts: 
\begin{enumerate}[start=1,label={{{(P\arabic*})}}]
	\item \label{it:P1} The difference in means $\rmm_{\mathbf{x}, t} - \rmm_{\mathbf{y}, t}$ is such that its first component consists of an $\OO(1)$ term (depending on $R$), while all remaining components converge to zero exponentially fast as $t \to \infty$.
	\item \label{it:P2} The covariance matrix  $\Sigma_t$ is such that its top-left entry consists of an $\OO(t)$ term plus terms that vanish as $t \to \infty$, while all the remaining entries consist of an $\OO(1)$ term plus terms that vanish as $t \to \infty$.
\end{enumerate}
One can easily check that \ref{it:P2} implies that the inverse $\Sigma_t^{-1}$ has elements in its first row and first column that decay to zero as $1/t$ as $t \to \infty$, while all other entries consist of $\OO(1)$ terms. Combining these observations with \ref{it:P1}, we conclude that the right-hand side of \eqref{eq:CameronMartin} converges to zero as $t \to \infty$, thereby proving the claim.
\end{proof}

\section{Identification of the spherical average in arbitrary even dimension}
\label{sec:identification}
In dimension $d = 2$, it is well known that for any $x \in \mathbb{R}^d$, the process $(h_{e^{-t}}(x) - h_1(x))_{t \in \mathbb{R}}$ has the distribution of a standard two-sided Brownian motion. 
The aim of this section is to establish the characterisation of the spherical average process in any arbitrary even dimension $d > 2$, as stated in Theorem~\ref{th:identSphere}. 

\begin{proof}[Proof of Theorem~\ref{th:identSphere}]
We start by verifying that the multivariate process $(\mathbf{S}_t)_{t \in \R}$ defined in \eqref{eq:defIntOU} is the stationary solution to the $\cd_d$-dimensional SDE of the form \eqref{eq:SDEpOrderMat}, with the matrix $\mathbf{A}$ and the vector $\mathbf{b}$ as specified in the theorem's statement. Since, the covariance kernel of the process $(\rmS_t)_{t \in \R}$ is explicitly known, thanks to Lemma~\ref{lm:power}, it suffices to check that the power spectrum of the first derivative process $\smash{(\rmS^{(1)}_t)_{t \in \R}}$ coincides with that of the first component of the stationary solution of the aforementioned SDE.

We begin by noting that, assuming the process $(\rmS_t)_{t \in \R}$ is differentiable, the following holds, 
\begin{equation*}
\E\bigl[\rmS_t \rmS_s\bigr] = \int_0^t \int_0^s \E\bigl[\rmS^{(1)}_u \rmS^{(1)}_v\bigr] dv du\;, \qquad \forall \, t, \, s \in \R\;.
\end{equation*}
Hence, to compute the covariance of the first derivative $\rmS^{(1)}$, it suffices to take the second mixed derivative of the covariance kernel of $\rmS$, as specified in \eqref{eq:covSphereDiag}. Proceeding in this manner, we obtain that
\begin{equation*}
\CK^{(1)}(t, s) \eqdef \E\bigl[\rmS^{(1)}_t \rmS^{(1)}_s\bigr]  = c(d) \int_0^{\pi} \frac{1 - \cos \theta \cosh\abs{t-s}}{(\cos \theta - \cosh\abs{t-s})^2} (\sin \theta)^{d-2} d \theta\;, \qquad \forall \, s, t \in \R \;.
\end{equation*}
Since, the covariance kernel $\CK^{(1)} :\R \times \R \to \R$ only depends on the distance $\abs{t-s}$, with a slight abuse of notation, we define the function $\CK^{(1)} :\R \to \R$ by letting
\begin{equation*}
\CK^{(1)}(u) \eqdef c(d) \int_0^{\pi} \frac{1 - \cos \theta \cosh u}{(\cos \theta - \cosh u)^2} (\sin \theta)^{d-2} d \theta\;, \qquad \forall \, u \in \R\;.
\end{equation*}
Our goal now is to compute the Fourier transform of $\CK^{(1)}$. To this end, for $\theta \in [0, \pi]$, we define the function $\psi_{\theta}: \R \to \R$ by letting 
\begin{equation*}
\psi_{\theta}(u) \eqdef \frac{1 - \cos \theta \cosh u}{(\cos \theta - \cosh u)^2} \;, \qquad \forall \, u \in \R \;.
\end{equation*}
By applying the change of variable $u \mapsto e^{u}$ and performing some elementary integral computations, we find that  
\begin{equation*}
\widehat\psi_{\theta}(\omega) 
= \frac{1}{\sqrt{2 \pi}} \int_{-\infty}^{\infty} \frac{1 - \cos \theta \cosh u}{(\cos \theta - \cosh u)^2} e^{-i u \omega} du 
= \frac{1}{\sqrt{2 \pi}} \frac{2 \pi \omega}{e^{\pi \omega} - e^{-\pi \omega}} \bigl(e^{(\theta - \pi) \omega} + e^{-(\theta - \pi) \omega}\bigr)\;,
\end{equation*}
and so it holds that 
\begin{align*}
\widehat{\CK^{(1)}}(\omega) = c(d) \int_{0}^{\pi} \widehat\psi_{\theta}(\omega) (\sin \theta)^{d-2}  d\theta = \frac{1}{\sqrt{2 \pi}} \frac{c(d) 2 \pi \omega}{e^{\pi \omega} - e^{-\pi \omega}} \int_{0}^{\pi} \bigl(e^{(\theta - \pi) \omega} + e^{- (\theta - \pi) \omega}\bigr)(\sin \theta)^{d-2}  d \theta \;.
\end{align*}
Furthermore, we have the following identity 
\begin{equation}
\label{eq:firstIdCompSph}
\int_0^{\pi} (e^{(\theta - \pi) \omega} + e^{- (\theta -\pi) \omega})(\sin \theta)^{d-2} d \theta = \frac{(d-2) (d-3)}{\omega^2 + (d-2)^2} \int_0^{\pi} (e^{(\theta - \pi) \omega} + e^{- (\theta - \pi) \omega})(\sin \theta)^{d-4} d \theta \;,
\end{equation}
and also
\begin{equation}
\label{eq:secondIdCompSph}
	\int_0^{\pi} (e^{(\theta - \pi) \omega} + e^{-(\theta - \pi) \omega})(\sin \theta)^2 d \theta = \frac{2 (e^{\pi \omega} - e^{-\pi \omega})}{\omega(\omega^2+4)} 	\;.
\end{equation}
The identity \eqref{eq:firstIdCompSph} is derived using integration by parts, whereas \eqref{eq:secondIdCompSph} follows from some elementary integral computations. Therefore, using \eqref{eq:firstIdCompSph} and \eqref{eq:secondIdCompSph}, we find that 
\begin{equation*}
\widehat{\CK^{(1)}}(\omega) = \frac{1}{\sqrt{2 \pi}} \frac{2^{d-2}\Gamma(d/2)^2}{\prod_{k=1}^{\frac{d-2}{2}} \bigl(\omega^2 + (d-2 k)^2\bigr)} \;, \qquad \forall \, \omega \in \R \;,
\end{equation*}
where we used the identity $c(d) 2 \pi (d-2)! = 2^{d-2}\Gamma(d/2)^2$. Now, noting that we can write  
\begin{equation*}
\prod_{k=1}^{\frac{d-2}{2}} \bigl(\omega^2 + (d-2 k)^2\bigr) = \Biggl[\prod_{k=1}^{\frac{d-2}{2}} \bigl((d-2 k) + i \omega\bigr)\Biggr]\Biggl[\prod_{k=1}^{\frac{d-2}{2}} \bigl((d-2 k) - i\omega\bigr)\Biggr] \;,
\end{equation*}
thanks to Lemma~\ref{lm:power}, the multivariate process $(\mathbf{S}_t)_{t \in \R}$ is a solution to the $\cd_d$-dimensional SDE of the form \eqref{eq:SDEpOrderMat} where the coefficients $\smash{(a_k)_{k = 0}^{\cd_d}}$ of the matrix $\mathbf{A}$ are chosen in such a way that its eigenvalues are $\smash{(-\lambda_k)_{k=1}^{\cd_d}}$ where 
\begin{equation*}
	\lambda_k \eqdef d-2k \;, \qquad \forall \, k \in [\cd_d] \;,
\end{equation*}
and the coefficient $b = c(d) (d-2)! \pi$. We introduce the function $g:\R \to \R$ given by
\begin{equation*}
	g(s) \eqdef (d-2)e^{2s}(1-e^{2s})^{\frac{d-4}{2}}\;, \qquad \forall \, s \in \R \;.
\end{equation*}
Then, thanks again to Lemma~\ref{lm:power}, the process $(\rmS^{(1)}_t)_{t \in \R}$ admits the following representation,
\begin{equation}
\label{eq:eqDerSphere}
\rmS^{(1)}_t  = 2^{d/2-1} \Gamma(d/2) \sum_{k=1}^{\cd_d} c_k \int_{-\infty}^{t} e^{-(d-2k) (t-s)} dB_s = \int_{-\infty}^{t} g(s-t) d B_s \;, \qquad \forall \, t \in \R \;,
\end{equation}
where the second equality follows by recalling that $c_k = (\prod_{j = 1, , j \neq k}^{\cd_d} (\lambda_j - \lambda_k))^{-1}$ and a straightforward computation. For $t \geq 0$, by taking the integral on both sides of \eqref{eq:eqDerSphere}, we get that
\begin{align*}
\rmS_t & = \int_0^t\int_{-\infty}^{u} g(s-u)d B_s \, du
= \int_{-\infty}^t\int_{\max(0, s)}^t g(s-u) du \, d B_s 
= \int_{-\infty}^0 g(s) (B_{t+s} - B_s) ds
\end{align*}
where the first equality follows from stochastic Fubini's theorem, and the second equality follows from stochastic integration by parts and rearrangements. The same expression can be derived for $t < 0$, and so the claim follows.
\end{proof}

\section{The LBM in arbitrary dimension}
\label{sec:LMB}
In this section, we focus on the construction of the LBM and study its main properties. First, in Section~\ref{sub:LBMFixed}, we provide additional details on constructing the LBM starting from a fixed point, and we prove that if $\alpha \in (0, 2)$ is chosen according to \eqref{eq:relGammaAlpha}, then the LBM is compatible with the coordinate change rule described in Proposition~\ref{pr:coordinate_change_measure}.
Next, in Section~\ref{sub:LBMMarkov}, we construct the $d$-dimensional LBM for all $Q > 2$ as a Markov process. To achieve this, we closely follow the procedure established in \cite{GRV_LBM}, which itself builds upon \cite{Brosamler}. In particular, a key part of the construction involves potential analysis of the GMC measure $\mu_{\bh, \alpha}$. We then focus on the subcritical phase, i.e., $Q > \sqrt{\smash[b]{2d}}$, where we identify the additive functional used in the definition of the LBM and prove that its transition semigroup admits a density with respect to $\mu_{\bh, \alpha}$.
Finally, in Section~\ref{sub:spectral}, we prove Theorem~\ref{th_specDim}, i.e., we compute the spectral dimension of the LBM.

We recall that $(\BO, \BF, (B_t)_{t \geq 0}, (\BF_{t})_{t \geq 0}, (\BP_x)_{x \in \R^d})$ is a $d$-dimensional Brownian motion, independent of the whole-space LGF $h$ defined on the probability space $(\Omega, \FF, \P)$. Furthermore, for $t \geq 0$, we introduce the following occupation measures
\begin{equation}
\label{eq:defSigma}
\sigma(A) = \int_0^{\infty} \mathbbm{1}_{\{B_s \in A\}} ds\;, \qquad \sigma_t(A) = \int_0^{t} \mathbbm{1}_{\{B_s \in A\}} ds \;, \qquad \forall \, A \in \CB(\R^d), \; \forall \, t \geq 0 \;.
\end{equation}
In this section, we restrict our attention to the case $d > 2$, since the proofs of the corresponding two-dimensional results can be found in \cite{GRV_LBM, Ber_LBM, SpecDim}.
Moreover, we assume throughout this section that the parameters $Q > 2$ and $\alpha \in (0, 2)$ satisfy the relation specified in \eqref{eq:relGammaAlpha}. 

\subsection{The LBM from a fixed point}
\label{sub:LBMFixed}
As noted in the introduction, for a fixed point $x \in \R^d$ and $Q > 2$, the LBM $(\rmB_{\bh, \alpha, t})_{t \geq 0}$, as defined in Definition~\ref{def:LBM}, is constructed by time-changing the standard $d$-dimensional Brownian motion $(B_t)_{t \geq 0}$ started from $x$. In particular, constructing the LBM from a fixed point is a straightforward task that falls under the framework of the theory of GMC measures. 

Specifically, for all $x \in \R^d$, we observe that the measure $\sigma_t$ defined in \eqref{eq:defSigma} has $\BP_x$-almost surely \emph{effective dimension} $2$ (regardless of the ambient dimension $d \geq 2$), i.e., $2$ is the largest non-negative real number $\frks$ such that for all $R > 0$ and all $\eps>0$, 
\begin{equation*}
\int_{B(0, R) \times B(0, R)} \frac{1}{\abs{x-y}^{\frks - \eps}} \sigma_t(dx) \sigma_t(dy) < \infty \;,
\end{equation*}
We refer to \cite[Lemma~5.1]{BerMulti} for a proof of this fact. 
Therefore, for $\alpha \in (0, 2)$, if we define the GMC measure associated with with respect to $\sigma$ as
\begin{equation}
\label{eq:defNuAlpha}
\nu_{\bh, \alpha}(dx) \eqdef \lim_{\eps \to 0} \eps^{\frac{\alpha^2}{2}} e^{\alpha \bh_{\eps}(x)} \sigma(dx) \;.
\end{equation}
it is well-known (see \cite{Berestycki_Elementary, BP24, GRV_LBM}) that for all $x \in \R^d$, the limit in the above expression converges $\BP_x$-almost surely in $\P$-probability with respect to the topology of weak convergence in the space of non-negative, locally finite measures on $\R^d$.
Hence, to rigorously construct the LBM starting from $x$, it suffices to note that the measure $\nu_{\bh, \alpha}$ is related to the clock process $\rmF_{\bh, \alpha}$, as defined in \eqref{eq:clockLBM}, through the following relationship
\begin{equation*}
	\rmF_{\bh, \alpha}(t) = \nu_{\bh, \alpha}\bigl([B]_{t}\bigr)\;, \qquad \forall \, t \geq 0 \;,
\end{equation*}
where $[B]_t$ represents the trajectory of the Brownian motion up to time $t$.

In order to prove the conformal invariance of the LBM, we need the following auxiliary result.
\begin{proposition}[Coordinate change formula for $\nu_{\bh, \alpha}$]
\label{pr:coordinate_change_clock}
Consider the same setting as in Proposition~\ref{pr:coordinate_change_measure}. For $Q > 2$, let $\alpha \in (0, 2)$ be as specified in \eqref{eq:relGammaAlpha}. Then, for all $x \in \R^d$, it holds $\BP_x \otimes \P$-almost surely that
\begin{equation}
\label{eq:coordinate_change_clock}
\nu_{\bh, \alpha}(\phi(A)) = \nu_{\bh \circ \phi + Q \log |\phi'|, \alpha}(A) \;, \qquad \forall \, A \in \CB(\R^d) \;.
\end{equation}
\end{proposition}
\begin{proof}
This proof is completely analogous to the proof of Proposition~\ref{pr:coordinate_change_measure}. However, we provide the details in order to highlight where the relation \eqref{eq:relGammaAlpha} between $Q$ and $\alpha$ originates from.  

We start by observing that, for every $d \geq 2$, the measure $\sigma$ defined in \eqref{eq:defSigma} has two-dimensional scaling, i.e., if $\phi: \R^d \to \R^d$ is a conformal automorphism which can be expressed as a composition of translations, dilations and rotations of $\R^d$, then for all Borel subsets $A \subseteq \R^d$,
\begin{equation*}
\sigma(\phi(A)) \eqlaw |\phi'|^2 \sigma(A) \;.
\end{equation*}

Now, let $c > 0$ and consider the field $\bh^c \eqdef \bh(c \,\cdot)$. For any $\eps > 0$ and $x \in \R^d$, we have the following relation between the spherical averages: $\bh^c_{\eps}(x) = \bh_{c \eps}(c x)$. For $A \in \CB(\R^d)$, we have that 
\begin{equation*}
\nu_{\bh^c, \alpha}(A) 
= \lim_{\eps \to 0} \eps^{\alpha^2/2} \int_A e^{\alpha \bh^c_{\eps}(x)} \sigma(dx)
= \lim_{\eps \to 0} \eps^{\alpha^2/2} \int_A e^{\alpha \bh_{c \eps}(c x)} \sigma(dx) 
= \lim_{\delta \to 0} (\delta/c)^{\alpha^2/2} c^{-2} \int_{c A} e^{\alpha \bh_{\delta}(x)} \sigma(dx) \;,
\end{equation*}
where in the last equality we set $\delta = c \eps$, we did the change of variables $x \mapsto x/c$, and we used  use the two-dimensional scaling of the occupation measure of the Brownian motion. We note that the term on the right-hand side of the above display equals $c^{-2 - \alpha^2/2} \nu_{\bh, \alpha}(c A)$, and so we have that 
\begin{equation*}
c^{2 + \alpha^2/2}  \nu_{\bh^c, \alpha}(A)  =  \nu_{\bh, \alpha}(cA) \qquad \iff \qquad \nu_{\bh^c  + Q \log c, \alpha, t}(A) = \nu_{\bh, \alpha}(c A)\;.
\end{equation*}
where here we used the fact that, for $\alpha = \alpha(Q)$ as in \eqref{eq:relGammaAlpha}, it holds that $\alpha Q = 2 + \alpha^2/2$. Finally, as for the proof of Proposition~\ref{pr:coordinate_change_measure}, we observe that the same argument also works if instead of just a scaling we have a composition of a scaling, rotation, and translation.
\end{proof}

\subsubsection{Conformal invariance of the LBM}
\label{sub:conformalLBM}
We now have all the necessary elements to prove that if $\alpha \in (0, 2)$ is chosen as specified in \eqref{eq:relGammaAlpha}, then the LBM is compatible with the coordinate change rule.

\begin{proposition}
\label{pr:conformalLBM}
Consider the same setting as in Proposition~\ref{pr:coordinate_change_measure}. For $Q > 2$, let $\alpha \in (0, 2)$ be as specified in \eqref{eq:relGammaAlpha}. Then, for all $x \in \R^d$, it holds $\BP_x \otimes \P$-almost surely that
\begin{equation*}
\bigl(\phi^{-1}(\rmB_{\bh, \alpha, t})\bigr)_{t \geq 0} = \bigl(\bar{\rmB}_{\bh \circ \phi + Q \log \abs{\phi'}, \alpha, t})_{t \geq 0} \quad \text{ where } \quad \bar{\rmB}_{\bh \circ \phi + Q \log \abs{\phi'}, \alpha, t}	\eqdef \bar{B}_{\bar{\rmF}^{-1}_{\bh \circ \phi + Q \log \abs{\phi'}, \alpha}(t)} \;,
\end{equation*}
with $(\bar{B})_{t})_{t \geq 0}$ a Brownian motion on $\R^d$ and $\bar{\rmF}_{\bh \circ \phi + Q \log \abs{\phi}, \alpha}$ the associated clock process with respect to the field $\bh \circ \phi + Q \log \abs{\phi'}$. 
\end{proposition}
\begin{proof}
For $c > 0$, define $(\bar{B}_t)_{t \geq 0} \eqdef (c^{-1} B_{c^{2} t})_{t \geq 0}$ and consider the dilation $\phi: \R^d \to \R^d$ give by $\phi(x) = c x$. Then, for all $t \geq 0$, it holds that 
\begin{equation*}
\phi^{-1}\bigl(\rmB_{\bh, \alpha, t}\bigr) = c^{-1} B_{\rmF^{-1}_{\bh, \alpha}(t)} = \bar{B}_{c^{-2} \rmF^{-1}_{\bh, \alpha}(t)} = \bar{B}_{\rmF_{\bh, \alpha, c}^{-1}(t)}\;,
\end{equation*}
where we set $\rmF_{\bh, \alpha, c}(t) \eqdef \rmF_{\bh, \alpha}(c^{2} t)$. Furthermore, for all $t \geq 0$, thanks to Proposition~\ref{pr:coordinate_change_clock}, we observe that
\begin{equation*}
\rmF_{\bh, \alpha, c}(t) =  \nu_{\bh}([B]_{c^2 t}) \overset{\eqref{eq:coordinate_change_clock}}{=} \nu_{\bh\circ \phi + Q \log\abs{\phi'}, \alpha}(c^{-1} [B]_{c^{2} t}) = \bar{\nu}_{\bh\circ \phi + Q \log\abs{\phi'}, \alpha}([\bar{B}]_t) = \bar{\rmF}_{\bh\circ \phi + Q \log\abs{\phi'}, \alpha}(t) \;,
\end{equation*}
where $\bar{\nu}_{\bh \circ \phi + Q \log\abs{\phi'}, \alpha}$ denotes the GMC with respect to the field $\bh \circ \phi + Q \log\abs{\phi'}$ along the trajectory of the Brownian motion $(\bar{B}_t)_{t \geq 0}$.
Finally, as for the proof of Proposition~\ref{pr:coordinate_change_clock}, we observe that the same argument also works if instead of just a scaling we have a composition of a scaling, rotation, and translation.
\end{proof}

\begin{remark}
The reason why, in Proposition~\ref{pr:coordinate_change_clock}, we only considered conformal automorphism of $\R^d$ which can be expressed as a composition of translations, dilations and rotations, instead of a general conformal maps, is that the Brownian motion in dimension three and higher is not conformally invariant. Indeed, consider the inversion map $\phi(z) =  z/\abs{z}^2$ and let $(B_t)_{t \geq 0}$ be a Brownian motion on $\R^d$ for $d > 3$. Then of course $(\phi(B_t))_{t \geq 0}$ cannot be expressed as a time-change of a Brownian motion. 
\end{remark}

\subsection{The LBM as a Markov process}
\label{sub:LBMMarkov}
In this section we construct the $d$-dimensional LBM as a diffusion process for all $Q > 2$. We begin by recalling some definitions and results that will be useful in what follows.
\begin{definition}
A \emph{positive continuous additive functional} (PCAF) $(\rmA_t)_{t \geq 0}$ is a $\BF_t$-adapted continuous functional with values in $[0, \infty]$ such that there exists a set $\Lambda \in \BF_{\infty}$ and a polar set $\rmN \subseteq \R^d$ for the $d$-dimensional Brownian motion for which $\BP_x(\Lambda) = 1$, for all $x \in \R^d \setminus \rmN$, and 
\begin{equation*}
\rmA_{t+s}(\omega) = \rmA_t(\omega) + \rmA_s(\omega) \circ \theta_{t}(\omega) \;, \qquad \forall \, \omega \in \Lambda \;, \; \forall \, s, t\geq 0\;,
\end{equation*}
where $(\theta)_{t \geq 0}$ denotes the family of shift mappings on $\Omega$, i.e., $B_{t+s} = B_t \circ \theta_s$, for all $s$, $t \geq 0$. Furthermore, if $\rmN = \emptyset$, then $(\rmA_t)_{t \geq 0}$ is called a PCAF in the \emph{strict sense}.
\end{definition}

\begin{definition}
A Borel measure $\rho_{\rmA}$ on $\R^d$ is called the\footnote{The fact that the Revuz measure of a PCAF is unique follows from the general theory of symmetric Markov processes, see e.g.\ \cite[Theorem~A.3.5]{Fukushima_Symmetric}.} \emph{Revuz measure} of a PCAF $(\rmA_t)_{t \geq 0}$ if for any non-negative Borel function $f : \R^d \to [0, \infty]$ it holds that 
\begin{equation*}
\int_{\R^d} f(x) \rho_{\rmA}(dx) = \lim_{t \to 0} \frac{1}{t} \int_{\R^d} \BE_x\Biggl[\int_0^t f(B_s) d \rmA_s\Biggr] dx \;.
\end{equation*}
\end{definition}

\subsubsection{Potential theory of the GMC measure in arbitrary dimension $d > 2$}
For each $R > 0$, we let $G_R : B(0, R) \times B(0, R) \to \R$ be the Green's function of the Laplacian on the ball $B(0, R)$ with zero boundary conditions, i.e., 
\begin{equation}
\label{eq:defGreenBall}
- \Delta G_R(x, \cdot) = \delta(x - \cdot)\;, \qquad G_R(x, \cdot)_{|\partial B(0, R)} = 0 \;.
\end{equation}

\begin{definition}
The $R$-potential of a positive Borel measure $\rho$ on $\R^d$ is the function $\mathrm{g}_R(\rho) : B(0, R)\to \R \cup \{\infty\}$ defined as follows 
\begin{equation}
\label{eq:potential}
\mathrm{g}_R[\rho](x) \eqdef \int_{B(0, R)} G_R(x, y) \rho(dy) \;, \qquad \forall \, x \in B(0, R) \;.
\end{equation}
Furthermore, the proper domain of $\mathrm{g}_R(\rho)$ is the set $\rmG_R[\rho] \subseteq B(0, R)$ given by
\begin{equation*}
\rmG_R[\rho] \eqdef \bigl\{x \in B(0, R) \, : \, \mathrm{g}_R[\rho](x) < \infty \bigr\} \;.
\end{equation*}
\end{definition}
It is a standard fact in potential theory that $B(0, R) \setminus \rmG_R[\rho]$ is a polar set for the $d$-dimensional Brownian motion (see e.g.\ \cite[page~247]{Brosamler}). 

In order to associate a PCAF to the measure $\mu_{\bh, \alpha}$, we begin with the fact such a measure does not charge polar sets. In what follows, for $R > 0$, we denote by $\mu_{\bh, \alpha}|_{B(0, R)}$ the restriction of the measure $\mu_{\bh, \alpha}$ to the ball $B(0, R)$.
\begin{lemma}[{\cite[Theorem~5.2]{Sturm1}}]
\label{lm:noChargePolar}
Let $d > 2$, $Q > 2$, and $R > 0$. Then, $\P$-almost surely, the measure $\mu_{\bh, \alpha}|_{B(0, R)}$ does not charge polar sets of the $d$-dimensional Brownian motion.
\end{lemma}

Letting $\tau_R$ denote the first time the Brownian motion $(B_t)_{t \geq 0}$ exits the ball $B(0, R)$, it follows from Lemma~\ref{lm:noChargePolar} and \cite[Proposition~3.2]{Brosamler} that, for all $Q > 2$, we can associate to the measure $\mu_{\bh, \alpha}|_{B(0, R)}$ a unique PCAF $(\rmF_{\bh, \alpha, R}(t))_{t \geq 0}$ such that the process
\begin{equation}
\label{eq:introPCAF}
\bigl(\mathrm{g}_R[\mu_{\bh,\alpha}](B_{t \wedge \tau_R}) - \mathrm{g}_R[\mu_{\bh,\alpha}](B_{0}) + \rmF_{\bh, \alpha, R}(t)\bigr)_{t \geq 0}
\end{equation}
is a mean-zero martingale under $\BP_x$ for all $x \in \rmG_R[\mu_{\bh, \alpha}]$.

\begin{lemma}
\label{lm:boundPCAF}
Let $d > 2$, $Q > \sqrt{\smash[b]{2d}}$, and $R > 0$. Then, it holds $\P$-almost surely that,
\begin{equation}
\label{eq:boundgR}
\sup_{x \in B(0, R)} \mathrm{g}_R[\mu_{\bh, \alpha}](x) < \infty \;,
\end{equation}
or in other words $\rmG_R[\mu_{\bh, \alpha}] = B(0, R)$. Therefore, for $Q > \sqrt{2d}$, $\rmF_{\bh, \alpha, R}$ is a PCAF in the strict sense.
\end{lemma}
\begin{proof}
Fix $d > 2$, $Q > \sqrt{\smash[b]{2d}}$, and $R > 0$. By recalling that for any $x$, $y \in B(0, R)$, the quantity $G_R(x, y)$ is bounded by a multiple of $\abs{x-y}^{2-d}$, it suffices to show that 
\begin{equation*}
\sup_{x \in B(0, R)} \int_{B(0, R)} \abs{x-y}^{2-d} \mu_{\bh,\alpha}(dy) < \infty \;.
\end{equation*}
By using Lemma~\ref{lm:modContLQG0}, for any $x \in B(0, R)$ and $\eps > 0$, it holds that 
\begin{align*}
\int_{B(0, R)} \abs{x-y}^{2-d} \mu_{\bh,\alpha}(dy) 
& = \sum_{n \in \N_0} \int_{\{y \in B(0, R) \, : \, \abs{x - y} \in (e^{-(n+1)}2R, e^{-n}2R]\}} \abs{x-y}^{2-d} \mu_{\bh,\alpha}(dy)  \\
& \lesssim \sum_{n \in \N_0} e^{-n(2-d)} \mu_{\bh,\alpha}\bigl(B(x, e^{-(n-1)}2R)\bigr) \\
& \lesssim \sum_{n \in \N_0} e^{-n(2 + \alpha^2/2 - \sqrt{\smash[b]{2d}}\alpha - \eps)} \;,
\end{align*}
where the random implicit constants depend only on $R$. 
The desired result follows since $2 + \alpha^2/2 - \sqrt{\smash[b]{2d}}\alpha > 0$ for all $Q > \sqrt{\smash[b]{2d}}$ and by taking $\eps > 0$ small enough. 	
\end{proof}

\begin{remark}
\label{rm:PCAFnotStrict}
We expect that if $Q \in (2, \sqrt{\smash[b]{2d}}]$, then the set $\smash{B(0, R) \setminus \rmG_R[\mu_{\bh, \alpha}]}$ is $\P$-almost surely \emph{not} empty. Indeed, for $\beta \in [Q, \sqrt{\smash[b]{2d}}]$, the set $\TT_{\beta} \cap B(0, R)$ is $\P$-almost surely non-empty, and if $x \in \TT_{\beta} \cap B(0, R)$, then it should hold that $\mathrm{g}_R[\mu_{\bh, \alpha}](x) = \infty$. 
To see this, we recall that the field $\bh$ around a typical point $x \in \TT_{\beta}$ locally looks like $\bh - \beta\log\abs{\cdot - x}$.
In particular, using Lemma~\ref{lm:scalingThick} and following the same reasoning as in the proof of Lemma~\ref{lm:boundPCAF}, for any $\eps > 0$, we obtain that
\begin{equation*}
\int_{B(0, R)} \abs{x-y}^{2-d} \mu_{\bh - \beta\log\abs{\cdot - x},\alpha}(dy) \gtrsim \sum_{n \in \N_0} e^{-n(2 + \alpha^2/2 - \beta \alpha + \eps)}\;,
\end{equation*} 
and since by assumption $2 + \alpha^2/2 - \beta \alpha \leq 0$, the sum on the right-hand side of the above display diverges. We also emphasise that for any $\beta \in (2, \sqrt{\smash[b]{2d}}]$, the set of $\beta$-thick points $\TT_{\beta}$ has Hausdorff dimension equals to $d - \beta^2/2$, which is strictly less than $d - 2$. Consequently, by \cite[Theorem~8.20]{Peres}, $\TT_{\beta}$ is a polar set for the standard $d$-dimensional Brownian motion. Combining these observations, we expect that for $\smash{Q \in (2, \sqrt{\smash[b]{2d}}]}$, $\smash{\rmF_{\bh, \alpha, R}}$ is a PCAF \emph{not} in the strict sense.
\end{remark}

\subsubsection{Identification of the PCAF in the subcritical regime}
We now limit ourself to the subcritical regime, i.e., $Q > \sqrt{\smash[b]{2d}}$, and we identify the PCAF $\rmF_{\bh, \alpha, R}$ introduced in \eqref{eq:introPCAF}. We emphasise that in \cite{Sturm1}, the authors provide an identification of the PCAF for almost every choice of the starting point for all $Q > 2$.

To do so, we first need to introduce some notation. In what follows, recalling \eqref{eq:defLGFSeries}, we define $\smash{h_{(n)} \eqdef \sum_{i=1}^n \alpha_i f_i}$ and we let $\smash{\bh_{(n)} \eqdef h_{(n)} - h_{(n), 1}(0)}$ where $h_{(n), 1}(0)$ denotes the spherical average of $h_{(n)}$ over the unit sphere centred at the origin.
We define the function $\rmC:\R^d \to \R$ by letting
\begin{equation}
\label{eq:CR}
\log \rmC(x) \eqdef \lim_{\eps \to 0} \E\bigl[\bh_{\eps}(x)^2\bigr] + \log\eps \;, \qquad \forall \, x \in \R^d \;,
\end{equation}
and, for $\alpha \in (0, 2)$, we introduce the measure $\mu_{\bh, \alpha, (n)}$ by letting
\begin{equation}
\label{eq:defmuN}
\mu_{\bh, \alpha, (n)}(dx) = e^{\alpha^2 \bh_{(n)}(x) - \frac{1}{2} \alpha^2 \E[\bh_{(n)}(x)^2] + \frac{1}{2}\gamma^2 \log \rmC(x)} dx \;.
\end{equation}
It is a standard fact (see e.g.\ \cite[Lemma~2.9]{BP24}) that the measure $\mu_{\bh, \alpha}$ introduced in \eqref{eq:defLQGSphere} is $\P$-almost surely the weak limit for $n \to \infty$ of the measure $\mu_{\bh, \alpha, (n)}$. In what follows, with a slight abuse of notation, we denote by $\mu_{\bh, \alpha, (\infty)}$ the limiting measure $\mu_{\bh, \alpha}$.

For $R > 0$, $Q > \sqrt{\smash[b]{2d}}$, and $n \in \N$, we observe that the mapping $B(0, R) \ni x \mapsto \mathrm{g}_R[\mu_{\bh,\alpha, (n)}](x)$ is $\P$-almost surely bounded and continuous. This follows easily since, for each $n \in \N$, the measure $\mu_{\bh,\alpha, (n)}$ is absolutely continuous with respect to the Lebesgue measure.  
It is well-known (see e.g.\ \cite[page~467]{Bass}), that we can associate to the measure $\mu_{\bh, \alpha, (n)}$ a unique PCAF $\rmF_{\bh,\alpha, (n), R}$ such that the process 
\begin{equation*}
\bigl(\mathrm{g}_R[\mu_{\bh,\alpha, (n)}](B_{t \wedge \tau_R}) - \mathrm{g}_R[\mu_{\bh,\alpha, (n)}](B_{0}) + \rmF_{\bh, \alpha, (n), R}(t)\bigr)_{t \geq 0}
\end{equation*}
is a mean-zero martingale under $\BP_x$ for all $x \in B(0, R)$. We also emphasise that, for $n \in \N$, we have the following explicit representation of the PCAF $\smash{\rmF_{\bh, \alpha, (n), R}}$,
\begin{equation*}
\rmF_{\bh, \alpha, (n), R}(t) = \int_0^{t \wedge \tau_R} e^{\alpha^2 \bh_{(n)}(B_s) - \frac{1}{2} \alpha^2 \E[\bh_{(n)}(B_s)^2] + \frac{1}{2}\alpha^2 \log \rmC(B_s)} ds \;, \qquad \forall \, t \geq 0\;.
\end{equation*}

For any $z \in \R^d$, we denote by $\mu_{\bh, \alpha}^{z}$ the shifted measure $\mu_{\bh, \alpha}^{z}(\cdot) = \mu_{\bh, \alpha}(z + \cdot)$. We have the following result which is the higher dimensional analogue of \cite[Proposition~2.3]{GRV_LBM}. 
\begin{proposition}
\label{pr:contPCAF}
Let $d > 2$, $Q > \sqrt{\smash[b]{2d}}$, and $R > 0$. Then,  it holds $\P$-almost surely that:
\begin{enumerate}[start=1,label={{{(G\arabic*})}}]
\item \label{it:G1} As $n \to \infty$, it holds that $\sup_{x \in B(0, R)}\abs{\mathrm{g}_R[\mu_{\bh,\alpha, (n)}](x) - \mathrm{g}_R[\mu_{\bh, \alpha}](x)}\to 0$. 
\item \label{it:G2} The map $B(0, R) \ni x \mapsto \mathrm{g}_R[\mu_{\bh,\alpha}](x)$ is continuous.
\item \label{it:G3} For any $z_0 \in \R^d$, as $z \to z_0$ it holds $\sup_{x \in B(0, R)}\abs{\mathrm{g}_R[\mu^{z}_{\bh,\alpha}](x) - \mathrm{g}_R[\mu^{z_0}_{\bh,\alpha}](x)} \to 0$. 
\end{enumerate}
\end{proposition}
\begin{proof}
We begin by proving \ref{it:G1}. The proof closely follows the argument presented in \cite[Proposition~2.3]{GRV_LBM}. It is included here for completeness and to show that the techniques of \cite{GRV_LBM} do not heavily depend on the ambient dimension $d \geq 2$. With a slight abuse of notation, we denote the limiting measure $\mu_{\bh, \alpha}$ as $\mu_{\bh, \alpha, (\infty)}$. We consider a smooth function $\theta : \R^d \to \R$ such that $\theta(x) = 1$ for $\abs{x} \leq 1$ and $\theta(x) = 0$ for $\abs{x} \geq 2$. For $\delta > 0$, we then set $\smash{\theta_{\delta}(x) \eqdef \theta(x/\delta)}$ and $\smash{\bar{\theta}_{\delta}(x) \eqdef 1 - \theta_{\delta}(x)}$. Then, for any $n \in \N \cup \{\infty\}$, it holds that 
\begin{equation*}
\mathrm{g}_R[\mu_{\bh,\alpha, (n)}](x) = \underbrace{\int_{B(0, R)} G_R(x, y) \theta_{\delta}(x-y) \mu_{\bh,\alpha, (n)}(dy)}_{\rmI_{(n), \delta}(x)} + \underbrace{\int_{B(0, R)} G_R(x, y) \bar{\theta}_{\delta}(x-y) \mu_{\bh,\alpha, (n)}(dy)}_{\bar{\rmI}_{(n), \delta}(x)}\;.
\end{equation*}
We start by showing that the family of maps $B(0, R) \ni x \mapsto \rmI_{(n), \delta}(x)$ converges uniformly to $0$ as $\delta \to 0$, i.e., we need to verify that 
\begin{equation*}
\lim_{\delta \to 0 } \sup_{n \in \N \cup \{\infty\}}\sup_{x \in (0,R)} \int_{B(0, R)} \abs{x-y}^{2-d}\theta_{\delta}(x-y) \mu_{\bh,\alpha, (n)}(dy) = 0 \;. 
\end{equation*}
This holds since, by using Lemma~\ref{lm:modContLQG0}, for $n \in \N \cup \{\infty\}$, $x \in B(0,R)$, and any $\eps > 0$, it holds that 
\begin{align*}
\int_{B(0, R)} \abs{x-y}^{2-d}\theta_{\delta}(x-y) \mu_{\bh,\alpha, (n)}(dy)
& \leq \sum_{k \geq \lfloor \log(R/(2\delta)) \rfloor} \int_{\{y \in B(0, R) \, : \, \abs{x - y} \in (e^{-(k+1)}R, e^{-k}R]\}} \abs{x-y}^{2-d} \mu_{\bh,\alpha, (n)}(dy)  \\
& \lesssim \sum_{k \geq \lfloor \log(R/(2\delta)) \rfloor} e^{-k(2-d)} \mu_{\bh,\alpha, (n)}\bigl(B(x, e^{-(k-1)}R)\bigr) \\
& \lesssim \sum_{k \geq \lfloor \log(R/(2\delta)) \rfloor} e^{-k(2 + \alpha^2/2 - \sqrt{\smash[b]{2d}} \alpha - \eps)}  \;,
\end{align*}
where the (random) implicit constants depend only on $R$. For all $Q > \sqrt{\smash[b]{2d}}$ the latter series converges to $0$ as $\delta \to 0$ by taking $\eps > 0$ small enough.
Therefore, to conclude, it suffices to show that for each fixed $\delta > 0$, the family $\bar{\rmI}_{(n), \delta}$ converges uniformly on $B(0, R)$ to $\smash{\bar{\rmI}_{(\infty), \delta}}$ as $n \to \infty$. The fact that this convergence holds pointwise is a direct consequence of the fact that the measure $\mu_{\bh, \alpha, (n)}$ converges weakly to $\smash{\mu_{\bh, \alpha}}$ as $n \to \infty$. Hence, it suffices to show that the family $\smash{(\bar\rmI_{(n), \delta})_{n \in \N}}$ is relatively compact for the topology of uniform convergence in $B(0, R)$. This follows since for each $\delta > 0$, the mapping $B(0, R) \times B(0, R) \ni (x, y) \mapsto G_{R}(x, y) \bar\theta_d(x-y)$ is uniformly continuous in $B(0, R) \times B(0, R)$, the fact that $\sup_{n \in \N} \mu_{\bh, \alpha, (n)}(B(0, R)) <  \infty$, and Arzel\'a--Ascoli theorem.

Finally, the proofs of \ref{it:G2} and \ref{it:G3} rely on similar ideas and are entirely analogous to the proofs of the corresponding properties established in \cite[Proposition~2.3]{GRV_LBM}. Therefore, their proofs are omitted.
\end{proof}
 
\begin{lemma}
\label{lm:convClockR}
Let $d > 2$, $Q > \sqrt{\smash[b]{2d}}$, and $R > 0$. Then, it holds $\P$-almost surely for all $x \in B(0, R)$ and $\eta > 0$ that
\begin{equation*}
\lim_{n \to \infty} \BP_x\Biggl(\sup_{t \geq 0}\abs{\rmF_{\bh, \alpha, (n), R}(t) - \rmF_{\bh, \alpha, R}(t)} \geq \eta\Biggr) = 0\;.
\end{equation*}
\end{lemma}
\begin{proof}
This result is a consequence of \ref{it:G1} in Proposition~\ref{pr:contPCAF} and \cite[Proposition~2.1]{Bass}.
\end{proof}

We are finally ready to identify the PCAF of the LBM on the whole space and to state its main properties. 
\begin{theorem}
\label{th:conLBM}
Let $d > 2$ and $Q > \sqrt{\smash[b]{2d}}$. Then, $\P$-almost surely, there exists a unique PCAF in the strict sense $(\rmF_{\bh, \alpha}(t))_{t \geq 0}$ such that, for all $R > 0$, 
\begin{equation*}
\rmF_{\bh, \alpha}(t) = \rmF_{\bh, \alpha, R}(t) \;, \qquad \forall \, t < \tau_R \;.
\end{equation*}
Furthermore, the following properties hold $\P$-almost surely:
\begin{enumerate}[start=1,label={{{(B\arabic*})}}]
\item \label{it:B1}  The Revuz measure of $\rmF_{\bh, \alpha}$ is $\mu_{\bh, \alpha}$.
\item \label{it:B2}  For all $x \in \R^d$, $T> 0$, and $\eta > 0$, the following limit holds
\begin{equation*}
\lim_{n \to \infty} \BP_x\Biggl(\sup_{t \leq T}\abs{\rmF_{\bh, \alpha, (n)}(t) - \rmF_{\bh, \alpha}(t)} \geq \eta\Biggr) = 0\;.
\end{equation*}
\item \label{it:B3}  For all $x \in \R^d$, it holds $\BP_x$-almost surely that the map $[0, \infty) \ni t \mapsto \rmF_{\bh, \alpha}(t)$ is strictly increasing. 
\item \label{it:B4}  For all $x \in \R^d$, it holds $\BP_x$-almost surely that $\lim_{t \to \infty} \rmF_{\bh,\alpha}(t) = \infty$. 
\item \label{it:B5}  For each $T > 0$, and for any bounded continuous function $\mathscr{F}: \CC([0,T], \R^{+}_0) \to \R$, the mapping $\R^d \ni x \to \BE_{x}[\mathscr{F}(\rmF_{\bh, \alpha}(\cdot))]$ is continous.
\end{enumerate}
\end{theorem}

The proof of Theorem~\ref{th:conLBM} relies, among other things, on the following standard coupling result for $d$-dimensional Brownian motions which is a consequence of the strong-Markov property. Consider a Brownian motion $W$ independent of $B$ and defined on the same probability space. We define the collection of stopping times $(\tau_i)_{i \in \{1, \ldots, d\}}$ inductively by letting
\begin{equation*}
	\tau_1 \eqdef \inf\bigl\{t \geq 0 \, : \, B^{(1)}_t = W^{(1)}_t\bigr\}\;, \qquad  \tau_i \eqdef \inf\bigl\{t \geq \tau_{i-1} \, : \, B^{(i)}_t = W^{(i)}_t\bigr\}\;, \quad \forall \, i \in \{2, \ldots, d\}\;,
\end{equation*}
where $B^{(i)}$ and $W^{(i)}$ denote the $i$-th coordinate of $B$ and $W$, respectively. Under, $\BP_{x, y}$ the law of $(B, W)$ is that of two independent $d$-dimensional Brownian motions started from $(x, y)$.  We have the following generalisation of \cite[Lemma~2.9]{GRV_LBM}.
\begin{lemma}
\label{lm:couplingBB}
Consider the situation described above. Under $\BP_{x, y}$, the process $\bar B$ given by
\begin{equation*}
\bar B_t \eqdef \begin{cases}
(W^{(1)}_t, \ldots, W^{(d)}_t)\;, \qquad & \text{ for } t \leq \tau_1 \;, \\		
(B^{(1)}_t, W^{(2)}_t, \ldots, W^{(d)}_t)\;, \qquad & \text{ for } t \in (\tau_1, \tau_2] \;, \\		
\hfill \vdots\\
(B^{(1)}_t, \ldots, B^{(d-1)}_t, W^{(d)}_t)\;, \qquad & \text{ for } t \in (\tau_{d-1}, \tau_{d}] \;, \\
(B^{(1)}_t, \ldots, B^{(d)}_t)\;, \qquad & \text{ for } t > \tau_d \;, 
\end{cases}
\end{equation*}
is a $d$-dimensional Brownian motion started from $y$. Furthermore, for all $\eta > 0$, it holds that
\begin{equation*}
	\lim_{\delta \to 0} \sup_{x, y \in \R^d, \, \abs{x-y}\leq \delta}\BP_{x, y}(\tau_d > \eta) = 0 \;, \qquad \BP_{x, y}(\tau_d < \infty) = 1\;.
\end{equation*}
\end{lemma} 

\begin{proof}[Proof of Theorem~\ref{th:conLBM}]
The proof of this result is similar to the proof of \cite[Theorem~2.7]{GRV_LBM}. Existence and uniqueness of the PCAF are immediate consequences of the previous discussions. For \ref{it:B1}, we refer to the proof of \cite[Theorem~2.7]{GRV_LBM} and references therein. For \ref{it:B2}, this is an immediate consequence of Lemma~\ref{lm:convClockR} and the fact that we can take $R$ large enough to make $\P(\tau_R < T)$ arbitrarily small. Items \ref{it:B3} and \ref{it:B5} can be proved in the exact same way as in \cite[Theorem~2.7]{GRV_LBM} by using Lemma~\ref{lm:couplingBB}.

It remains to prove \ref{it:B4}. In this case, the proof differs from \cite{GRV_LBM} since in that paper, the authors leverage on the exponential decay of correlation of the massive GFF. We note that it suffices to check that $\P$-almost surely for a fixed $x \in \R^d$, it holds that $\BP_x(\lim_{t \to \infty} \rmF_{\bh,\alpha}(t) = \infty) = 1$. To get the desired result, it is then sufficient to apply Lemma~\ref{lm:couplingBB}. Therefore, we take $x = 0$ and our goal is to show that $\lim_{t \to \infty} \rmF_{\bh,\alpha}(t) = \infty$, $\P \otimes \BP_0$-almost surely. To this end, it suffices to show that $\lim_{n \to \infty} \rmF_{\bh,\alpha}(\tau_n) = \infty$, $\P \otimes \BP_0$-almost surely. For all $n \in \N$, by the scale invariance of the law of $\bh$ modulo an additive constant, we have that $\smash{\bh^n \eqdef \bh(n \cdot) - \bh_{n}(0) \eqlaw \bh}$. Consequently, by Proposition~\ref{pr:coordinate_change_clock}, it follows that 
\begin{equation*}
\rmF_{\bh,\alpha}(\tau_n) = \nu_{\bh, \alpha}\bigl([B]_{\tau_n}\bigr) \eqlaw n^{\alpha Q} e^{\alpha \bh_{n}(0)} \nu_{\bh^n, \alpha}\bigl(n^{-1}[B]_{\tau_n}\bigr)\;, 
\end{equation*}
where we recall that the measure $\nu_{\bh, \alpha}$ is defined in \eqref{eq:defNuAlpha}.
For any $c > 0$, using the fact that $\bh_{n}(0)$ is a centred Gaussian with variance of order $\log n$ (see Lemma~\ref{lm:varianceSpherical}), we have that 
\begin{equation*}
\sum_{n \in \N} \P\bigl(n^{\alpha Q} e^{\alpha \bh_{n}(0)} \leq c\bigr) < \infty \;,
\end{equation*}
and so, thanks to the Borel--Cantelli lemma, we have that 
\begin{equation}
\label{eq:pointOneExplode}
\P\Bigl(\lim_{n \to \infty} n^{\alpha Q} e^{\alpha \bh_{n}(0)} = \infty\Bigr) = 1 \;.	
\end{equation}
Moreover, we note that for each $n \in \N$, the law of $\nu_{\bh^n, \alpha}(n^{-1}[B]_{\tau_n})$ does not depend on $n$, and also, thanks to \ref{it:B3}, for any $\eps > 0$, there exists $\delta > 0$ such that $\P\otimes\BP_0(\nu_{\bh^n, \alpha}(n^{-1}[B]_{\tau_n}) \geq \delta) \geq 1 - \eps$, for all $n \in \N$. In particular, thanks to the downward continuity of measures, we have that 
\begin{equation}
\label{eq:pointTwoExplode}
\P\otimes\BP_0 \Biggl(\limsup_{n \to \infty} \bigl(\nu_{\bh^n, \alpha}(n^{-1}[B]_{\tau_n} \bigr) \geq \delta  \Biggr) \geq \limsup_{n \to \infty} \P\otimes\BP_0\bigl(\nu_{\bh^n, \alpha}(n^{-1}[B]_{\tau_n}) \geq \delta \bigr) \geq 1 -\eps \;.
\end{equation}
Hence, thanks to \eqref{eq:pointOneExplode} and \eqref{eq:pointTwoExplode} and by arbitrariness of $\eps > 0$, we have that $\P\otimes\BP_0(\lim_{n \to \infty} \rmF_{\bh,\alpha}(\tau_n) = \infty) = 1$, from which the desired result follows. 
\end{proof}

We note that the continuity of $\rmF_{\bh, \alpha}$ ensures that the LBM does not become trapped in certain regions of $\R^d$. Additionally, the strict monotonicity of $\rmF_{\bh, \alpha}$ guarantees that the LBM has no discontinuities, and since $\rmF_{\bh, \alpha}$ approaches infinity as $t \to \infty$, the LBM has an infinite lifetime, i.e., it will never ``reach infinity'' in a finite amount of time.

In what follows, we will also require the following strengthened version of the fact that the measure $\mu_{\bh, \alpha}$ is the Revuz measure associated with the PCAF $\rmF_{\bh, \alpha}$.
\begin{lemma}
\label{lm:StrenghRevuz}
Let $d > 2$ and $Q > \sqrt{\smash[b]{2d}}$. Then, $\P$-almost surely, for all $x \in \R^d$ and all Borel measurable functions $\eta: \R_0^{+} \to \R_0^{+}$  and $f:\R^d \to \R_0^{+}$ the following equality holds 
\begin{equation}
\label{eq:StrenghRevuz}
\BE_x \Biggl[\int_0^{\infty} \eta(t) f(B_t) d \rmF_{\bh, \alpha}(t) \Biggr] = \int_0^{\infty} \int_{\R^d} \eta(t) f(y) p_{t}(x, y) \mu_{\bh, \alpha}(dy) dt \;,
\end{equation}
where $p_t(\cdot, \cdot)$ denotes the heat kernel of the $d$-dimensional Brownian motion (see \eqref{eq:EuclideanHeat} below).
\end{lemma}
\begin{proof}
The result is proved for the two-dimensional case in \cite[Proposition~2.5]{AndresKajino}. Given the bound on the potential of the measure $\mu_{\bh, \alpha}$ provided in Lemma~\ref{lm:boundPCAF}, the proof generalises verbatim to any arbitrary dimensions $d > 2$. 
\end{proof}

We conclude this section with the following result, which is analogous to Theorem~\ref{th:conLBM}, but with the Brownian motion replaced by a Brownian bridge, and which will be useful for the study of the spectral dimension. We recall that we use the convention that for $x$, $y \in \R^d$ and $t > 0$, under $\BP_{x, y, t}$, the process $(B_s)_{s \in [0, t]}$ is a $d$-dimensional Brownian bridge of length $t$ from $x$ to $y$.

\begin{theorem}
\label{th:conLBMBB}
Let $d > 2$ and $Q > \sqrt{\smash[b]{2d}}$. Then, $\P$-almost surely, for all $x$, $y \in \R^d$ and $t \geq 0$, the following properties hold:
\begin{enumerate}[start=1,label={{{(D\arabic*})}}]
\item \label{it:D1} The following limit holds
\begin{equation*}
\lim_{n \to \infty} \BP_{x, y, t}\Biggl(\sup_{s \leq t}\abs{\rmF_{\bh, \alpha, (n)}(s) - \rmF_{\bh, \alpha}(s)} \geq \eta\Biggr) = 0\;.
\end{equation*}
\item \label{it:D2} It holds $\BP_{x, y, t}$-almost surely that the map $[0, t] \ni s \mapsto \rmF_{\bh, \alpha}(s)$ is strictly increasing and continuous.  
\item \label{it:D3} For each bounded continuous function $\mathscr{F}: \CC([0,t]) \to \R$, the mapping $\R^d \times \R^d \ni (x, y) \to \E_{x, y, t}[\mathscr{F}(\rmF_{\bh, \alpha}(\cdot))]$ is continous. 
\end{enumerate}
\end{theorem}
\begin{proof}
As observed in \cite{SpecDim}, the proof consists in reproducing the arguments of the proof of Theorem~\ref{th:conLBM} with minor modifications.
\end{proof} 

\subsubsection{The Liouville Heat Kernel}
From the construction of the LBM and from the general theory of time changes of Markov processes, we can deduce some of its properties. First of all, for $d > 2$ and $Q > \sqrt{\smash[b]{2d}}$, the LBM $\rmB_{\bh,\alpha}$ is a diffusion process, i.e., a continuous strong Markov process, and it is transient (cf.\ \cite[Theorems~A.2.12~and~6.2.3]{Fukushima_Symmetric}). Furthermore, thanks to \cite[Theorem~6.2.1~(i)]{Fukushima_Symmetric}, the LBM is $\mu_{\bh,\alpha}$-symmetric, i.e., its transition semigroup $(\rmP_{\bh,\alpha, t})_{t \geq 0}$ given by
\begin{equation*}
\rmP_{\bh,\alpha, t}[f](x) \eqdef \BE_x\bigl[f(\rmB_{\bh,\alpha, t})\bigr] \;, \qquad \forall \, f \in \CB(\R^d, \R), \; \forall \, x \in \R^d 
\end{equation*}
satisfies the following equality 
\begin{equation}
\label{eq:invarianceTransition}
\int_{\R^d} \rmP_{\bh,\alpha, t}[f](x) g(x) \mu_{\bh,\alpha}(dx) = \int_{\R^d} \rmP_{\bh,\alpha, t}[g](x) f(x) \mu_{\bh,\alpha}(dx) \;, \qquad \forall \, f, g \in \CB(\R^d, \R) \;. 
\end{equation}
We are now ready to prove Proposition~\ref{pr:existenceLHK} on the absolute continuity of the semigroup $\rmP_{\bh,\alpha}$ with respect to the measure $\mu_{\bh,\alpha}$. 
\begin{proof}[Proof of Proposition~\ref{pr:existenceLHK}]
We will not provide a detailed proof of this result since it is identical to the proof of \cite[Theorem~2.5]{GRV_heat}. We observe that, since the semigroup $\rmP_{\bh,\alpha}$ is symmetric with respect to $\mu_{\bh, \alpha}$, by \cite[Theorem~4.2.4]{Fukushima_Symmetric}, it suffices to prove that the resolvent of the LBM $\rmB_{\bh,\alpha}$ is absolutely continuous with respect to $\mu_{\bh, \alpha}$. We recall that the resolvent family $(\rmR_{\bh, \alpha, \lambda})_{\lambda > 0}$ associated to the LBM is defined as follows
\begin{equation*}
\rmR_{\bh, \alpha, \lambda}[f](x) = \int_0^{\infty} e^{-\lambda t} \rmP_{\bh,\alpha, t}[f](x) dt\;, \qquad \forall \, \lambda > 0\;, \quad \forall \, f \in \CB(\R^d, \R)\;.
\end{equation*}
Therefore, it suffices to prove that $\P$-almost surely, for all $A \in \CB(\R^d)$ such that $\mu_{\bh, \alpha}(A) = 0$, it holds that $\rmR_{\bh, \alpha, \lambda}[\mathbbm{1}_{A}](x) = 0$ for all $x \in \R^d$ and $\lambda > 0$. To this end, thanks to \eqref{eq:invarianceTransition}, we have that 
\begin{equation*}
\lambda \int_{\R^d} \rmR_{\bh, \alpha, \lambda}[\mathbbm{1}_{A}](x) \mu_{\bh, \alpha}(dx) = \mu_{\bh, \alpha}(A)\;.
\end{equation*}
Hence, if $\mu_{\bh, \alpha}(A) = 0$, it follows that $\rmR_{\bh, \alpha, \lambda}[\mathbbm{1}_{A}](x) = 0$ for $\mu_{\bh, \alpha}$-almost every $x \in \R^d$. Since the measure $\mu_{\bh, \alpha}$ has full support, this implies the existence of a dense subset of $\R^d$ on which the function $\rmR_{\bh, \alpha, \lambda}[\mathbbm{1}_{A}]$ is identically zero. Therefore, the conclusion holds if the mapping $\R^d \ni x \mapsto \rmR_{\bh, \alpha, \lambda}[\mathbbm{1}_{A}](x)$ is continuous. By following the proof of \cite[Theorem~2.4]{GRV_heat}, one can show that the resolvent operator $\rmR_{\bh, \alpha, \lambda}$ maps the set of measurable bounded functions into the set of continuous bounded functions. 
More precisely, we observe that the proof of \cite[Theorem2.4]{GRV_heat} relies on Lemma~\ref{lm:couplingBB} and the fact that, for any fixed $R > 0$, it holds $\P$-almost surely that
\begin{equation*}
	\lim_{t \to 0} \sup_{x \in B(0, R)} \BE_x\bigl[\rmF_{\bh, \alpha}(t)\bigr] = 0 \;.
\end{equation*}
It is easy to see that the uniform limit above holds in any dimension $d > 2$. Indeed, fix $R > 0$, take any $x \in B(0, R)$, and let $\delta \in (0, 1/2)$. Then, by Lemma~\ref{lm:StrenghRevuz}, we can write
\begin{equation}
\label{eq:HeatExistsEqu}
\BE_x\bigl[\rmF_{\bh, \alpha}(t)\bigr] = \int_{B(x, t^{1/2-\delta})} \int_{0}^t p_{s}(x, y) ds \, \mu_{\bh, \alpha}(dy) + \int_{B(x, t^{1/2-\delta})^c} \int_{0}^t p_{s}(x, y) ds \, \mu_{\bh, \alpha}(dy) \;.
\end{equation}
We now proceed to show that both terms on the right-hand side of the above expression converge to $0$ as $t \to 0$, uniformly for all $x \in B(0, R)$. Regarding the first term, by Lemma~\ref{lm:boundftxy} and following the same approach as in the proof of Proposition~\ref{pr:contPCAF}, we obtain $\P$-almost surely that for any $\eps > 0$,\begin{equation*}
\int_{B(x, t^{1/2-\delta})} \int_{0}^t p_{s}(x, y) ds \mu_{\bh, \alpha}(dy) \lesssim \int_{B(x, t^{1/2-\delta})} \abs{x-y}^{2-d} \mu_{\bh, \alpha}(dy) \lesssim \sum_{k \geq \lfloor \abs{\log t} (1/2-\delta) \rfloor} e^{-k(2 + \alpha^2/2 - \sqrt{\smash[b]{2d}} \alpha - \eps)} \;,
\end{equation*} 
and we note that, for all $Q > \sqrt{\smash[b]{2d}}$, by choosing $\eps > 0$ sufficiently small, the term on the right-hand side of the above expression converges to $0$ as $t \to 0$.
Hence, it remains to control the second term on the right-hand side of \eqref{eq:HeatExistsEqu}. To this end, recalling the explicit expression of the heat kernel for the standard $d$-dimensional Brownian motion and applying a change of variables, we obtain that for $\abs{x - y} \geq t^{1/2 - \delta}$, there exists a constant $c > 0$ such that\begin{equation}
\label{eq:HeatExistsEqu1}
\int_{0}^t p_{s}(x, y) ds \lesssim \abs{x-y}^{2-d} \int_{\frac{\abs{x-y}^2}{t}}^{\infty} s^{\frac{d}{2}-2} e^{-\frac{s}{2}} ds \lesssim t^{(1/2-\delta)(2-d)} e^{-c \frac{\abs{x-y}^2}{t}}\;.
\end{equation} 
Moreover, we have that  
\begin{align}
\int_{B(x, t^{1/2-\delta})^c} e^{-c \frac{\abs{x-y}^2}{t}} \mu_{\bh, \alpha}(dy)
& \leq e^{-c t^{-2 \delta}} \mu_{\bh, \alpha}\bigl(B(x, 2)\bigr) + \int_{B(x, 2)^c} e^{-c \frac{\abs{x-y}^2}{t}} \mu_{\bh, \alpha}(dy) \nonumber \\
& \lesssim e^{-c t^{-2 \delta}} \mu_{\bh, \alpha}\bigl(B(x, 2)\bigr) + t^{d} \int_{\R^d} e^{-c \abs{x-y}^2} \mu_{\bh, \alpha}(dy)  \;. \label{eq:HeatExistsEqu2}
\end{align}
Therefore, combining \eqref{eq:HeatExistsEqu1} and \eqref{eq:HeatExistsEqu2}, we have that 
\begin{equation*}
\int_{B(x, t^{1/2-\delta})^c} \int_{0}^t p_{s}(x, y) ds \mu_{\bh, \alpha}(dy) \lesssim t^{(1/2-\delta)(2-d)}\Biggl(e^{-c t^{-2 \delta}} \mu_{\bh, \alpha}\bigl(B(x, 2)\bigr) + t^{d} \int_{\R^d} e^{-c \abs{x-y}^2} \mu_{\bh, \alpha}(dy)\Biggr) \;,
\end{equation*}
from which the claim follows since the expression on the right-hand side of the display converges to $0$ as $t \to 0$, $\P$-almost surely.
\end{proof}

\subsection{Spectral dimension of the LBM}
\label{sub:spectral}
In this section, we prove Theorem~\ref{th_specDim}, which concerns the spectral dimension of the LBM. In particular, we focus only on the case $d > 2$ and refer to \cite{SpecDim} for the proof of Theorem~\ref{th_specDim} in the case $d = 2$.

We begin by noting that for $Q > \sqrt{\smash[b]{2d}}$ and $\beta \in (-\infty, Q)$, the results derived in the previous section remain valid when $\bh$ is replaced by the field $\bh - \beta \log\abs{\cdot}$. For instance, one can easily check that the conclusion of Lemma~\ref{lm:boundPCAF} remains valid when $\bh$ is replaced by $\bh - \beta \log\abs{\cdot}$, i.e., that for all $R >0$ it holds $\P$-almost surely that $\smash{\sup_{x \in B(0, R)} \mathrm{g}_R[\mu_{\bh - \beta \log\abs{\cdot}, \alpha}](x) < \infty}$. Indeed, in the proof of Lemma~\ref{lm:boundPCAF}, it suffices to use Lemma~\ref{lm:modContLQG} in place of Lemma~\ref{lm:modContLQG0} to obtain the desired result. Verifying that the remaining results still hold with $\bh$ replaced by $\bh - \beta \log\abs{\cdot}$ reduces to similar considerations.

To simplify the notation, for a fixed $\beta \in (-\infty, Q)$, we introduce the following field,
\begin{equation*}
\bar \bh \eqdef \bh - \beta \log\abs{\cdot} \;. 
\end{equation*}

\subsubsection{Some preliminary results and a formal scaling argument}
We denote by $p_t(\cdot, \cdot)$ the heat kernel of the $d$-dimensional Brownian motion given by
\begin{equation}
\label{eq:EuclideanHeat}
p_t(x, y) \eqdef (2 \pi t)^{-\frac{d}{2}} e^{-\frac{\abs{x-y}^2}{2t}}\;, \qquad \forall \, t \geq 0, \; \forall \, x, y \in \R^d\;.
\end{equation} 
In what follows, we recall that for $x$, $y \in \R^d$ and $t > 0$, under $\BP_{x, y, t}$, the process $(B_s)_{s \in [0, t]}$ is a $d$-dimensional Brownian bridge of length $t$ from $x$ to $y$.
We begin with the following key lemma.
\begin{lemma}
\label{lm:keyBB}
Let $d > 2$, $Q > \sqrt{\smash[b]{2d}}$, and $\beta \in (-\infty, Q)$. Then, $\P$-almost surely, for all continuous functions $\eta: \R_0^{+} \to \R_0^{+}$, for all $x \in \R^d$, and for $\mu_{\bh, \alpha}$-almost every $y \in \R^d$, the following equality holds
\begin{equation*}
\int_0^{\infty} \eta(t) \frkp_{\bar\bh, \alpha, t}(x, y) dt = \int_0^{\infty} \BE_{x, y, t}\bigl[\eta(\rmF_{\bar\bh, \alpha}(t))\bigr] p_t(x, y) dt \;.
\end{equation*}
\end{lemma}
\begin{proof}
The proof is similar to the proof of \cite[Theorem~3.4]{SpecDim} and is included here for completeness. Let $\eta: \R^{+} \to \R^{+}$ be a continuous function, and let $x \in \R^d$. Then, thanks to Lemma~\ref{lm:StrenghRevuz} with $\bh$ replaced by the field $\bar\bh$, for all continuous functions $f: \R^d \to \R^{+}_0$, it holds that 
\begin{align*}
\int_{\R^d} \Biggl(\int_{0}^{\infty} \eta(t) \frkp_{\bar\bh, \alpha, t}(x, y) dt \Biggr) f(y) \mu_{\bar\bh, \alpha} (dy) 
& = \int_{0}^{\infty} \eta(t) \Biggr(\int_{\R^d} f(y) \frkp_{\bar\bh, \alpha, t}(x, y)  \mu_{\bar\bh, \alpha} (dy)\Biggr) dt \\
& = \BE_{x} \Biggl[\int_{0}^{\infty} \eta(t) f(\rmB_{\bar\bh, \alpha, t}) dt\Biggr] \\
& = \BE_{x} \Biggl[\int_{0}^{\infty} \eta\bigl(\rmF_{\bar \bh, \alpha}(t)\bigr) f(B_t) d \rmF_{\bar \bh, \alpha}(t)\Biggr] \\
& = \BE_{x} \Biggl[\int_{0}^{\infty} \BE_{x}\bigl[\eta\bigl(\rmF_{\bar \bh, \alpha}(t)\bigr) \, | \, \sigma(B_t)\bigr] f(B_t) d \rmF_{\bar \bh, \alpha}(t)\Biggr] \\
& \overset{\eqref{eq:StrenghRevuz}}{=} \int_{\R^d} \Biggl(\int_{0}^{\infty} \BE_{x, y, t}\bigl[\eta\bigl(\rmF_{\bar \bh, \alpha}(t)\bigr)\bigr] p_{t}(x, y) dt \Biggr) f(y) \mu_{\bar \bh, \alpha}(dy) \;,
\end{align*}  
and so the claim follows. 
\end{proof}

We emphasise that, for $\chi \geq 0$, by choosing $\R^{+} \ni t \mapsto \eta(t) \eqdef t^{\chi} e^{-t}$ in Lemma~\ref{lm:keyBB}, we obtain precisely the relation \eqref{eq:keyRelSpecMain}.
We are now in a position to state the main result of this section, from which Theorem~\ref{th_specDim} follows directly.
\begin{proposition}
\label{pr:MainSpec}
Let $d > 2$, $Q > \sqrt{\smash[b]{2d}}$, $\beta \in (-\infty, Q)$, and define
\begin{equation} 
\label{eq:barChi} 
\bar{\chi} = \bar \chi(\beta) \eqdef \frac{d-2}{2 + \alpha^2/2 - \alpha\beta} \;.
\end{equation}
Then, it holds $\P$-almost surely that 
\begin{equation*}
\frkm_{\bar\bh, \alpha, \chi}(0) = \infty \;, \qquad \forall \, \chi \in [0,  \bar{\chi}) \qquad \text{ and } \qquad 
\frkm_{\bar \bh, \alpha, \chi}(0)  < \infty \;, \qquad \forall \, \chi > \bar{\chi} \;,
\end{equation*}
where we recall that $\smash{\frkm_{\bar\bh, \alpha, \chi}}$ is defined in \eqref{eq:defMellinTrue}.
\end{proposition}

Before proceeding with the rigorous proof of Proposition~\ref{pr:MainSpec}, we provide a formal scaling argument to explain where the exponent $\bar \chi$ comes from. Recalling the definition \eqref{eq:defMellinTrue} of $\smash{\frkm_{\bar\bh, \alpha, \chi}(0)}$, we note that, for all $\chi \geq 0$, the mapping $[0, \infty) \ni u \mapsto u^{\chi} e^{-u}$ is uniformly bounded by a constant $c = c(\chi)$. In particular, since the heat kernel $p_t(0, 0)$ is equal to a constant times $t^{-d/2}$, which is integrable over the unbounded interval $(1, \infty)$, the part of the integral in $\smash{\frkm_{\bar\bh, \alpha, \chi}(0)}$ that determines the spectral dimension is the behaviour of the following quantity,
\begin{equation} 
\label{eq:heurSpec1}
\int_0^{1} \BE_{0, 0, t}\bigl[(\rmF_{\bar\bh,\alpha}(t))^{\chi}\bigr] p_t(0, 0) dt \;.
\end{equation} 
By the scale invariance of the law of $\bh$ modulo an additive constant, we have that $\bh^{\sqrt{t}} \eqdef \bh(\sqrt{t} \cdot) - \bh_{\sqrt{t}}(0) \eqlaw \bh$. In particular, using the coordinate change formula for the measure $\nu_{\bh, \alpha}$ (Proposition~\ref{pr:coordinate_change_clock}), we obtain that 
\begin{equation}
\label{eq:heurSpec2}
\rmF_{\bar\bh,\alpha}(t) = \nu_{\bar\bh, \alpha}\bigl([B]_{t}\bigr)  = t^{\alpha (Q -\beta)/2 + o_t(1)} \nu_{\bh^{\sqrt{t}} - \beta \log\abs{\cdot}, \alpha} \bigl(t^{-1/2} [B]_{t}\bigr) \;,
\end{equation}
Therefore, plugging \eqref{eq:heurSpec2} into \eqref{eq:heurSpec1} and using the fact that the law of $t^{-1/2}[B]_{t}$ does not depend on $t$, we obtian that 
\begin{equation*}
\int_0^{1} \BE_{0, 0, t}\bigl[(\rmF_{\bar\bh,\alpha}(t))^{\chi}\bigr] p_t(0, 0) dt \approx \int_0^1 t^{\chi \alpha(Q - \beta)/2 - d/2 + o_t(1)} dt \;.
\end{equation*}
Since $\alpha(Q - \beta) = 2 + \alpha^2/2 - \alpha\beta$, the integral on the right-hand side of the above expression converges (resp.\ diverges) if $\chi > \bar \chi$ (resp.\ $\chi \in [0, \bar \chi)$). 
However, there are several reasons why this scaling argument is not rigorous. For instance, while it is true that the law of the field $\bh^{\sqrt{t}}$ does not depend on $t$, this is insufficient for our purposes since we need statements that hold $\mathbb{P}$-almost surely, rather than just in distribution.

\subsubsection{Computation of the spectral dimension}
We now proceed with the rigorous proof of Proposition~\ref{pr:MainSpec}. The proof is divided into two parts: in Lemma~\ref{lm:SpecDimLower}, we address the case $\chi \in [0,  \bar{\chi})$, and in Lemma~\ref{lm:SpecDimUpper}, we handle the case $\chi > \bar{\chi}$.
\begin{lemma}
\label{lm:SpecDimLower}
Let $d > 2$, $Q > \sqrt{\smash[b]{2d}}$, $\beta \in (-\infty, Q)$, and let $\bar{\chi}$ be as defined in \eqref{eq:barChi}. Then, it holds $\P$-almost surely that  
\begin{equation*}
\frkm_{\bar\bh, \alpha, \chi}(0) = \infty \;, \qquad \forall \, 	\chi \in [0,  \bar{\chi}) \;.
\end{equation*}
\end{lemma}
\begin{proof}
Let $d > 2$, $Q > \sqrt{\smash[b]{2d}}$, $\beta \in (-\infty, Q)$, and consider $\chi \in [0,  \bar{\chi})$. We begin by observing that   
\begin{equation*}
\int_0^{\infty} \BE_{0,0,t}\bigl[(\rmF_{\bar\bh, \alpha}(t))^{\chi}e^{-\rmF_{\bar\bh,\alpha}(t)}\bigr] p_t(0, 0) dt 
\gtrsim \int_0^{\infty} \BE_{0, 0, t}\bigl[(\rmF_{\bar\bh, \alpha}(t))^{\chi}  \mathbbm{1}_{\{\rmF_{\bar\bh, \alpha}(t) \leq 1\}}\bigr] p_t(0, 0) dt\;.
\end{equation*}
Next, for any $t > 0$, we introduce the fields 
\begin{equation*}
\bh^{\sqrt{t}} \eqdef \bh(\sqrt{t} \cdot) -  \bh_{\sqrt{t}}(0) \;, \qquad {\bar\bh}^{\sqrt{t}} \eqdef \bh^{\sqrt{t}} - \beta \log\abs{\cdot} \;.	
\end{equation*}
Recall that $\bh_{\sqrt{t}}(0)$ is centred Gaussian with variance of order $\abs{\log \sqrt{t}}$ (Lemma~\ref{lm:varianceSpherical}), so it holds $\P$-almost surely that $\bh_{\sqrt{t}}(0)/\abs{\log \sqrt{t}} = o_t(1)$ as $t \to 0$. By applying the coordinate change formula for $\rmF_{\bh, \alpha}$ (Proposition~\ref{pr:coordinate_change_clock}), it follows $\P$-almost surely that, for any sufficiently small $\bar t \in (0,1)$ and for all $t \in (0,\bar t)$, 
\begin{align*}
\BE_{0, 0, t}\bigl[(\rmF_{\bar\bh, \alpha}(t))^{\chi} \mathbbm{1}_{\{\rmF_{\bar\bh, \alpha}(t) \leq 1\}}\bigr] 
& = t^{\alpha \chi (Q - \beta) /2} e^{\alpha \chi \bh_{\sqrt t}(0)} \BE_{0, 0, 1}\bigl[(\rmF_{{\bar\bh}^{\sqrt{t}}, \alpha}(1))^{\chi} \mathbbm{1}_{\{\rmF_{{\bar\bh}^{\sqrt{t}}, \alpha}(1) t^{\alpha (Q - \beta) /2} e^{\alpha \bh_{\sqrt t}(0)} \leq 1\}}\bigr] \\
& \geq t^{\chi (2 + \alpha^2/2 - \alpha\beta)/2 + o_t(1)} \underbrace{\BE_{0, 0, 1}\bigl[(\rmF_{{\bar\bh}^{\sqrt{t}} , \alpha}(1))^{\chi} \mathbbm{1}_{\{\rmF_{{\bar\bh}^{\sqrt{t}}, \alpha}(1) \leq 1\}} \bigr]}_{\rmZ_t}\;,
\end{align*}
where to get the inequality in the second line, we used that $\beta < Q$ and the fact that $\alpha(Q - \beta) = 2 + \alpha^2/2 - \alpha\beta$. Hence, it holds $\P$-almost surely that
\begin{equation}
\label{eq:SpecCase11}
\int_0^{\infty} \BE_{0,0,t}\bigl[(\rmF_{\bar\bh, \alpha}(t))^{\chi}e^{-\rmF_{\bar\bh, \alpha}(t)}\bigr] p_t(0, 0) dt  
\gtrsim \int_0^{\bar t} t^{\chi (2 + \alpha^2/2 - \alpha\beta)/2- d/2 + o_t(1)} \rmZ_t  dt \;.
\end{equation}

We next observe that by the invariance of the law of $\bh$ modulo an additive constant, the law of $\rmZ_t$ does not depend on $t > 0$. Moreover, for any $\eps \in (0, 1)$, there exists $s_{\eps} > 0$ such that $\P(\rmZ_t \geq s_{\eps}) \geq 1 - \eps^2$. Then, by Markov's inequality, for all $r \in (0, 1)$, 
\begin{equation*}
	\P\Biggl(\int_0^r\mathbbm{1}_{\{\rmZ_t < s_{\eps}\}} dt > \eps  r\Biggr) \leq \frac{1}{\eps r} \int_0^r \P(\rmZ_t < s_{\eps}) dt < \eps\;.
\end{equation*}
By downward continuity of measures, it follows that 
\begin{equation*}
\P\Biggl(\limsup_{n \to \infty} \Biggl(\int_0^{\frac{1}{n}}\mathbbm{1}_{\{\rmZ_t \geq s_{\eps}\}} dt \geq \frac{1-\eps}{n}\Biggr)\Biggr) \geq \limsup_{n \to \infty} \P\Biggl(\int_0^{\frac{1}{n}}\mathbbm{1}_{\{\rmZ_t \geq s_{\eps}\}} dt \geq \frac{1-\eps}{n}\Biggr) \geq 1-\eps \;.
\end{equation*}
Hence, with $\P$-probability at least $1- \eps$, there are infinitely many $n \in \N$ such that 
\begin{equation*}
\int_0^{\frac{1}{n}}\mathbbm{1}_{\{\rmZ_t \geq s_{\eps}\}} dt \geq \frac{1-\eps}{n} \;.	
\end{equation*} 
On that event, by choosing $n \in \N$ large enough so that $1/n < \bar t$, the right-hand side of \eqref{eq:SpecCase11} is bounded below by
\begin{equation*}
\int_0^{\frac{1}{n}} t^{\chi(2 + \alpha^2/2 - \alpha\beta)/2 - d/2 + o_t(1)} \rmZ_t dt \geq s_{\eps} \int_{\frac{\eps}{n}}^{\frac{1}{n}} t^{\chi(2 + \alpha^2/2 - \alpha\beta)/2 - d/2 + o_t(1)} dt \to \infty \;, \qquad \text{ as } n \to \infty \;,
\end{equation*}
where we used the fact that $-\chi(2 + \alpha^2/2 - \alpha\beta)/2 + d/2 - 1 > 0$ for all $\chi \in [0, \bar \chi)$. 
\end{proof}

Next, we address the case $\chi > \bar{\chi}$. A crucial element in the proof of this case is the following result, which is proved in Appendix~\ref{ap:keySpec}.
\begin{lemma}
\label{lm:keySpec}
Let $d > 2$, $Q > \sqrt{\smash[b]{2d}}$, and $\beta \in (-\infty, Q)$. Then, it holds $\P$-almost surely that for any $n \in \N$ and for all $t \in (0, 1]$,
\begin{equation}
\label{eq:keySpec1}
\BE_{0}\bigl[(\rmF_{\bar\bh, \alpha}(t))^n \mathbbm{1}_{\{[B]_{t} \subseteq B(0, t^{1/2- \delta})\}}\bigr] \lesssim t^{n[(1/2-\delta)(2+\alpha^2/2 - \alpha\beta) - \eps]} \;, \qquad \forall \, \eps, \; \delta > 0 \;,
\end{equation}
for some random implicit constant. 
\end{lemma}

\begin{lemma}
\label{lm:SpecDimUpper}
Let $d > 2$, $Q > \sqrt{\smash[b]{2d}}$, $\beta \in (-\infty, Q)$, and let $\bar{\chi}$ be as defined in \eqref{eq:barChi}. Then, it holds $\P$-almost surely that 
\begin{equation*}
\frkm_{\bar\bh, \alpha, \chi}(0) < \infty \;, \qquad \forall \, 	\chi > \bar{\chi} \;.
\end{equation*}
\end{lemma}
\begin{proof}
Let $d > 2$, $Q > \sqrt{\smash[b]{2d}}$, $\beta \in (-\infty, Q)$, and consider $\chi > \bar{\chi}$. Since the function $[0, \infty) \ni u \mapsto u^{\chi} e^{-u}$ is uniformly bounded by a constant depending only on $\chi$, it holds that
\begin{equation*}
\int_1^{\infty} \BE_{0, 0, t}\bigl[(\rmF_{\bar\bh, \alpha}(t))^{\chi}e^{-\rmF_{\bar\bh, \alpha}(t)}\bigr] p_t(0, 0) dt \lesssim \int_1^{\infty} t^{-d/2} dt < \infty \;. 
\end{equation*}
Therefore, it holds $\P$-almost surely that 
\begin{equation*}
\int_1^{\infty} \BE_{0, 0, t}\bigl[(\rmF_{\bar\bh, \alpha}(t))^{\chi}e^{-\rmF_{\bar\bh, \alpha}(t)}\bigr] p_t(0, 0) dt < \infty \;.
\end{equation*}
Hence, from now on, we can restrict our attention to the following quantity 
\begin{equation}
\label{eq:restrProofUpper}
\int_0^{1} \BE_{0,0,t}\bigl[(\rmF_{\bar\bh, \alpha}(t))^{\chi}e^{-\rmF_{\bar\bh, \alpha}(t)}\bigr] p_t(0, 0) dt \;.
\end{equation}
By the time-reversal symmetry of the Brownian bridge, the fact that for all $\chi > 0$ and $a$, $b > 0$ it holds that $(a+b)^{\chi} \lesssim a^{\chi} + b^{\chi}$ for some implicit constant depending only on $\chi$, and the monotonicity of the mapping $\R_0^{+} \ni t \mapsto \rmF_{\bar\bh, \alpha}(t)$, one obtains $\P$-almost surely that the term in \eqref{eq:restrProofUpper} can be bounded above by a constant multiple of the following sum,
\begin{equation*}
\int_0^1 \BE_{0, 0, t}\bigl[(\rmF_{\bar\bh, \alpha}(t/2))^{\chi} e^{-\rmF_{\bar\bh, \alpha}(t/2)} \bigr]  p_t(0, 0) dt + \int_0^1 \BE_{0, 0, t}\bigl[(\rmF_{\bar\bh, \alpha}(t/2))^{\chi} e^{-\rmF_{\bar\bh, \alpha}(t/2)} \bigr] p_t(0, 0) dt \;,
\end{equation*}
where we note that the two terms in the above display are equal.
Next, we recall that for any $x$, $y \in \R^d$ and $s \in [0, t)$, it holds that   
\begin{equation*}
\frac{d \BP_{x, y, t}}{d \BP_x}|_{\BF_{s}} = \Biggl(\frac{t}{t-s}\Biggr)^{d/2} \exp\Biggl(\frac{\abs{x-y}^2}{2t} - \frac{\abs{B_s- y}^2}{2(t-s)}\Biggr) \;,
\end{equation*}
Therefore, invoking this absolute continuity result, we obtain that  
\begin{equation}
\label{eq:usingI1andI2}
\int_0^1 \BE_{0, 0, t}\bigl[(\rmF_{\bar\bh, \alpha}(t/2))^{\chi} e^{-\rmF_{\bar\bh, \alpha}(t/2)} \bigr]  p_t(0, 0) dt \lesssim \int_0^1 \BE_{0}\bigl[(\rmF_{\bar\bh, \alpha}(t/2))^{\chi} e^{-\rmF_{\bar\bh, \alpha}(t/2)}\bigr]t^{-d/2} dt \;.
\end{equation}
For any $\delta > 0$, using again the fact that for each $\chi > 0$, the function $\R^{+}_{0} \ni u \mapsto u^{\chi} e^{-u}$ is uniformly bounded by a constant depending only on $\chi$, we observe that the term on the right-hand side of \eqref{eq:usingI1andI2} can be estimated from above by a constant multiple of
\begin{equation}
\label{eq:splitBall}
\int_0^1 \BE_{0}\bigl[(\rmF_{\bar \bh,\alpha}(t/2))^{\chi} \mathbbm{1}_{\{[B]_{t/2} \subseteq B(0, (t/2)^{1/2- \delta})\}}\bigr] t^{-d/2} dt + \int_0^1 \P_{0}\bigl([B]_{t/2} \not\subseteq  B(0, (t/2)^{1/2- \delta})\bigr) t^{-d/2} dt \;.
\end{equation} 
The second term in \eqref{eq:splitBall} is easily controlled since, by the reflection principle, for each $\delta > 0$ and some $c > 0$, 
\begin{equation*}
\P_{0}\bigl([B]_{t/2} \not \subseteq B(0, (t/2)^{1/2-\delta})\bigr) \lesssim e^{-c t^{-2 \delta}} \;.
\end{equation*}
Hence, plugging this estimate into the second term in \eqref{eq:splitBall}, we readily get that it is finite for any $\delta > 0$.
For the first term in \eqref{eq:splitBall}, let $n \in \N$ be such that $\chi \in [n-1 , n)$. By concavity of the map $\R^{+}_{0} \ni u \mapsto u^{\chi/n}$ and Jensen's inequality,
\begin{equation}
\label{eq:boundChiMoment0}
\BE_{0}\bigl[(\rmF_{\bar\bh, \alpha}(t/2))^{\chi} \mathbbm{1}_{\{[B]_{t/2} \subseteq B(0, (t/2)^{1/2- \delta})\}}\bigr]  
\leq \BE_{0}\bigl[(\rmF_{\bar\bh, \alpha}(t/2))^n \mathbbm{1}_{\{[B]_{t/2} \subseteq B(0, (t/2)^{1/2- \delta})\}}\bigr] ^{\chi/n} \;.
\end{equation}
Invoking Lemma~\ref{lm:keySpec}, for any $\eps > 0$, we find $\P$-almost surely that
\begin{equation*}
\BE_{0}\bigl[(\rmF_{\bar\bh, \alpha}(t/2))^{\chi} \mathbbm{1}_{\{[B]_{t/2} \subseteq B(0, (t/2)^{1/2- \delta})\}}\bigr] \lesssim t^{\chi[(1/2-\delta)(2+\alpha^2/2- \alpha \beta) - \eps]} \;.
\end{equation*}
In particular, for any $\chi > \bar{\chi}$ one has $\chi(2 + \alpha^2 - \alpha \beta)/2 - d/2 > -1$. Hence, by choosing $\eps$, $\delta > 0$ small enough (depending on $\chi$, $\alpha$, $\beta$), we obtain that also the first term in \eqref{eq:splitBall} is finite. Therefore, this concludes the proof.
\end{proof}

\begin{proof}[Proof of Proposition~\ref{pr:MainSpec}.]
The proof follows directly by combining Lemma~\ref{lm:SpecDimLower} with Lemma~\ref{lm:SpecDimUpper}. 
\end{proof}

\section{Quantum cones in arbitrary even dimension}
\label{sec:quantum}
In this section, we prove the results related to the quantum cone. The proof of Theorem~\ref{th:convRecentred} is divided between Sections~\ref{sub:tight} and~\ref{sub:uniq}. In Section~\ref{sub:tight}, we establish the tightness of the collection of processes $(\rmS_{b, \cdot})_{b > 0}$ defined in \eqref{eq:defRecProc}. Then, in Section~\ref{sub:uniq}, we prove the convergence in total variation distance.
In Section~\ref{sub:quantumLocal}, we prove Theorem~\ref{th:convQuantum}, and finally, Section~\ref{sub:LBMQuantum} is dedicated to the proof of Theorem~\ref{th:invarianceShift}.

\subsection{Tightness of the recentred process}
\label{sub:tight}
Fix $d > 2$ even, $\gamma \in (0, \sqrt{\smash[b]{2d}})$, and $\beta \in (-\infty, Q)$, where we recall that $Q$ is defined as in \eqref{eq:defQd} and satisfies $Q > \sqrt{\smash[b]{2d}}$. Recalling that the the collection of processes $\smash{(\rmS_{b, \cdot})_{b > 0}}$ is defined in \eqref{eq:defRecProc}, the main goal of this section is to prove the following tightness result.
\begin{proposition}
\label{pr:tightQC}
For $d > 2$ even, the collection of processes $(\rmS_{b, \cdot})_{b > 0}$ is tight in $\CC^{\cd_d}_{\loc}(\R)$ equipped with the local uniform metric.
\end{proposition}

\begin{remark}
\label{rm:QCd2}
In dimension $d = 2$, the circle average process of the whole-plane GFF is a two-sided Brownian motion. In this case, it follows from \cite[Proposition~4.13]{DMS21} that the collection of processes $\smash{(\rmS_{b, \cdot})_{b > 0}}$ converges weakly as $b \to \infty$ to a process $\smash{(\tilde{B}_t)_{t \in \R}}$, defined as follows. Let $\smash{(B^1_t)_{t \geq 0}}$ be a standard Brownian motion, and let $\smash{(B^2_t)_{t \geq 0}}$ be a Brownian motion conditioned to stay above the line $t \mapsto -(Q - \beta)t$ for all $t > 0$. Then, we let
\begin{equation}
\label{eq:defBMQC}
\tilde B_t \eqdef
\begin{cases}	
B^1_t \;, \qquad & \forall t \geq 0 \;, \\
B^2_{-t} \;, \qquad & \forall t < 0 \;.
\end{cases}
\end{equation}
\end{remark}

The proof of Proposition~\ref{pr:tightQC} relies on the observation that the spherical average process can be viewed, in a loose sense, as a smoothed version of a two-sided Brownian motion (see Theorem~\ref{th:identSphere} and Figure~\ref{fig:sph}). In particular, to prove Proposition~\ref{pr:tightQC}, we first introduce a family of processes $\smash{(\tilde \rmS_{b, \cdot})_{b > 0}}$ for which we can prove tightness by leveraging a suitable invariance in law. We then show that the tightness of $\smash{(\tilde \rmS_{b, \cdot})_{b > 0}}$ implies the tightness of $\smash{(\rmS_{b, \cdot})_{b > 0}}$. 

We introduce here some notation that will be used throughout this section. Recalling the representation \eqref{eq:defIntOU} of the spherical average process, we introduce the function $g: \R \to \R$ defined as
\begin{equation*}
g(s) \eqdef (d-2) e^{2s}(1-e^{2s})^{\frac{d-4}{2}} \;, \qquad \forall \, s \in \R \;.
\end{equation*}
We define the process $\smash{(\tilde \rmS_t)_{t \in \R}}$ by letting
\begin{equation}
\label{eq:defrmStilde}
\tilde \rmS_t \eqdef \int_{-\infty}^{0} g(s) \bigl(\tilde B_{s+t} - \tilde B_s\bigr) ds  \;, \qquad \forall \, t \in \R \;,
\end{equation}
where $(\tilde B_t)_{t \in \R}$ is defined as in \eqref{eq:defBMQC}. Furthermore, for $i \in [\cd_d]$, let $\smash{(\tilde{\rmS}_t^{(i)})_{t \in \R}}$ denote the $i$-th derivative of $\smash{(\tilde{\rmS}_t)_{t \in \R}}$. A straightforward computation yields that
\begin{equation}
\label{eq:defrmStildeDervs}
\begin{alignedat}{2}
\tilde \rmS^{(i)}_t & = (-1)^{i} \int_{-\infty}^{0} g^{(i)}(s) \tilde B_{s+t} ds \;, \qquad && \forall \, t \in \R \; , \; \forall \, i \in [\cd_d - 1] \;, \\
\tilde \rmS^{(\cd_d)}_t & = (-1)^{\cd_d + 1} g^{(\cd_d - 1)}(0) \tilde{B}_t + (-1)^{\cd_d} \int_{-\infty}^{0} g^{(\cd_d)}(s) \tilde B_{s+t} ds \;, \qquad && \forall \, t \in \R \;.
\end{alignedat}
\end{equation}
For $b > 0$, we introduce the stopping times $\tilde \tau_b$ and $\tilde \sigma_b$ by letting
\begin{equation}
\label{eq:defTildeTb}
\tilde \tau_b \eqdef \inf\bigl\{t \in \R \, : \,  \tilde B_t - (Q-\beta) t = -b\bigr\} \;, \qquad\qquad \tilde \sigma_b \eqdef \inf\bigl\{t \in \R \, : \,  \tilde \rmS_t - (Q - \beta) t = -b\bigr\} \;.
\end{equation}
For each $b > 0$ and $i \in [\cd_d]$, we also define recentred process $(\tilde\rmS_{b, t})_{t \in \R}$ and its $i$-th derivative $\smash{(\tilde\rmS^{(i)}_{b, t})_{t \in \R}}$ by letting 
\begin{equation}
\label{eq:defTildeSb}
\tilde{\rmS}_{b, t} \eqdef \tilde{\rmS}_{\tilde \sigma_b + t}  - \tilde{\rmS}_{\tilde \sigma_b}\;, \qquad \qquad \tilde{\rmS}^{(i)}_{b, t} \eqdef \tilde{\rmS}^{(i)}_{\tilde \sigma_b + t} \;, \qquad  \forall \, t \in \R \;. 
\end{equation}
The main reason why it is convenient to work with the process $(\tilde B_t)_{t \in \R}$ is that it is invariant under the following random recentring 
\begin{equation}
\label{eq:IVBM}
\bigl(\tilde B_{\tilde \tau_b + t} - \tilde B_{\tilde \tau_b}\bigr)_{t \in \R} \eqlaw \bigl(\tilde B_{t}\bigr)_{t \in \R} \;.
\end{equation}
The invariance in law \eqref{eq:IVBM} follows directly from the definition of the process $\smash{(\tilde{B}_t)_{t \in \R}}$ or from \cite[Proposition~4.13]{DMS21}.
Moreover, from the definition \eqref{eq:defrmStilde} of $\smash{(\tilde \rmS_t)_{t \in \R}}$ and from \eqref{eq:defrmStildeDervs}, it is readily seen that \eqref{eq:IVBM} implies that 
\begin{equation}
\label{eq:IVSPH}
\bigl(\tilde \rmS^{(i)}_{\tilde \tau_b + t} - \tilde \rmS^{(i)}_{\tilde \tau_b}\bigr)_{t \in \R} \eqlaw \bigl(\tilde \rmS^{(i)}_{t} - \tilde \rmS^{(i)}_{0}\bigr)_{t \in \R} \;, \qquad \forall \, i \in [\cd_d]_0 \;,
\end{equation}
where here and in what follows, with a slight abuse of notation, we use the superscript $(0)$ to denote the process itself.

For $b$, $c > 0$, we define the moduli of continuity $\smash{\omega_{b, c}}$, $\smash{\tilde \omega_{b, c}:[0,1]\to \R^{+}_0}$ of $\smash{(\rmS_{b, t})_{t \in [-c, c]}}$ and $\smash{(\tilde \rmS_{b, t})_{t \in [-c, c]}}$, respectively, by letting 
\begin{equation}
\label{eq:modContRec}
\omega_{b, c}(\delta) \eqdef \sum_{i = 0}^{\cd_d} \sup_{s, t \in [-c, c], \, \abs{t-s} \leq \delta} \abs{\rmS^{(i)}_{b, t} - \rmS^{(i)}_{b, s}} \;, \qquad \tilde \omega_{b, c}(\delta) \eqdef \sum_{i = 0}^{\cd_d} \sup_{s, t \in [-c, c], \, \abs{t-s} \leq \delta} \abs{\tilde \rmS^{(i)}_{b, t} - \tilde \rmS^{(i)}_{b, s}} \;.
\end{equation}
Similarly, for $c > 0$, we define the modulus of continuity $\tilde \omega_{c}:[0,1]\to \R^{+}_0$ of $(\tilde \rmS_t)_{t \in [-c, c]}$ by letting 
\begin{equation}
\label{eq:modContPlain}
\tilde \omega_{c}(\delta) \eqdef \sum_{i = 0}^{\cd_d} \sup_{s, t \in [-c, c], \, \abs{t-s} \leq \delta} \abs{\tilde \rmS^{(i)}_{t} - \tilde \rmS^{(i)}_{s}} \;.
\end{equation}

\subsubsection{The collection $(\tilde \rmS_{b, \cdot})_{b > 0}$ is tight}
We now proceed to show that the collection of processes $(\tilde \rmS_{b, \cdot})_{b > 0}$ is tight.
\begin{lemma}
\label{lm:tightQCTilde}
For $d > 2$ even, the collection of processes $(\tilde \rmS_{b, \cdot})_{b > 0}$ is tight in $\CC^{\cd_d}_{\loc}(\R)$ equipped with the local uniform metric.
\end{lemma}
The proof of Lemma~\ref{lm:tightQCTilde} relies on the fact that, with high probability and uniformly for $b > 0$, $\tilde \tau_b$ and $\tilde \sigma_b$ are close to each other. 
\begin{lemma}
\label{lm:StoppingClose}
For $b > 0$ let $\tilde \sigma_b$ and $\tilde \tau_b$ as defined in \eqref{eq:defTildeTb}. Then, it holds that  
\begin{equation*}
\lim_{\lambda \to \infty} \sup_{b > 0} \P\bigl(\abs{\tilde \sigma_b - \tilde \tau_b} \geq \lambda \bigr) = 0 \;.
\end{equation*}
\end{lemma}
\begin{proof}
For $\lambda \geq 0$, we have the following trivial inclusion
\begin{equation*}
\bigl\{\abs{\tilde \sigma_b - \tilde \tau_b} \geq \lambda \bigr\} \subseteq \bigl\{\tilde \sigma_b - \tilde \tau_b \geq \lambda \bigr\} \cup \bigl\{\tilde \tau_b - \tilde \sigma_b \geq \lambda \bigr\} \;,
\end{equation*}
and so we can focus on finding suitable upper bounds for the probabilities of the two events on the right-hand side of the above display. 

\textbf{Step 1:} We start by bounding the probability of the event $\{\tilde \sigma_b - \tilde \tau_b \geq \lambda\}$. To this end, for any $c > 0$, we note that  
\begin{equation}
\label{eq:Step1Comp1}
\bigl\{\tilde \sigma_b - \tilde \tau_b \geq \lambda\bigr\} \subseteq \bigl\{\abs{\tilde \rmS_{\tilde \tau_{b+c}} - \tilde B_{\tilde \tau_{b+c}}} \geq c\bigr\} \cup \bigl\{\tilde \tau_{b+c} - \tilde \tau_b \geq \lambda\bigr\} \;.
\end{equation}
To see why this is true, it is more convenient to consider the complement of the given relation and prove that the intersection of the complements of the two events on the right-hand side is contained within the complement of the event on the left-hand side. Indeed, noting that
\begin{equation*}
\abs{\tilde \rmS_{\tilde \tau_{b+c}} - \tilde B_{\tilde \tau_{b+c}}} < c \quad \implies \quad \tilde \sigma_b < \tilde \tau_{b+c}	\;,
\end{equation*}
and combining this fact with the complement of the second event on the right-hand side of \eqref{eq:Step1Comp1}, we readily obtain that it must hold that $\tilde \sigma_b - \tilde \tau_b < \lambda$.

Hence, our task now is to bound the two probabilities appearing on the right-hand side of \eqref{eq:Step1Comp1}. By recalling the definition \eqref{eq:defrmStilde} of $(\tilde\rmS_t)_{t \in \R}$, we can write
\begin{equation*}
\abs{\tilde \rmS_{\tilde \tau_{b+c}} - \tilde{B}_{\tilde \tau_{b+c}}} =   \Biggl\lvert\int_{-\infty}^{0} g(s) (\tilde B_{s+\tilde \tau_{b+c}} - \tilde B_{\tilde \tau_{b+c}}) ds - \int_{-\infty}^{0} g(s) \tilde B_s ds\Biggr\rvert  \;.
\end{equation*}
Therefore, thanks to the triangle inequality, the union bound, and the invariance in law \eqref{eq:IVBM}, we have that 
\begin{equation}
\label{eq:TightStep1Eq1}
\P\bigl(\abs{\tilde \rmS_{\tau_{b+c}} -   \tilde{B}_{\tau_{b+c}}} \geq c \bigr) \lesssim \P\Biggl(\Biggl\lvert \int_{-\infty}^0 g(s) \tilde{B}_s ds\Biggr\rvert \geq \frac{c}{2  }\Biggr) \;.
\end{equation}
Moreover, using the strong Markov property of the Brownian motion, we have that 
\begin{equation}
\label{eq:TightStep1Eq2}
\P\bigl(\tilde \tau_{b+c} - \tilde \tau_b \geq \lambda \bigr) = \P\Biggl(\inf_{s \in [0, \lambda]} \bigl(\tilde B_s - (Q-\beta) s\bigr) \geq -c\Biggr) \leq \P\bigl(\tilde B_{\lambda} \geq (Q-\beta) \lambda - c \bigr) \lesssim e^{- a ((Q-\beta)\lambda -c)^2/\lambda} \;,
\end{equation}
for some constant $a > 0$.
Therefore, choosing $c = c(\lambda) = \sqrt{\lambda}$, using the union bound, and estimates \eqref{eq:TightStep1Eq1},~\eqref{eq:TightStep1Eq2}, we readily get that 
\begin{equation*}
\lim_{\lambda \to \infty} \sup_{b > 0} \P\bigl(\tilde\sigma_b - \tilde \tau_b \geq \lambda\bigr) = 0 \;.
\end{equation*}

\textbf{Step 2:} We now have to bound the probability of the event $\{\tilde \tau_b - \tilde \sigma_b \geq \lambda\}$. To this end, we note that, for any $c > 0$, it holds that 
\begin{equation}
\label{eq:Step2Comp1}
\bigl\{\tilde \tau_b - \tilde \sigma_b \geq \lambda\bigr\}	\subseteq \Bigl\{\inf_{t \leq \tilde \tau_{b-c}} \tilde \rmS_t - (Q - \beta) t \leq -b\Bigr\} \cup \bigl\{\tilde\tau_b - \tilde \tau_{b-c} \geq \lambda\bigr\} \;.
\end{equation}
Also here, to see why this is true, it is more convenient to consider the complement of the given relation. Indeed, noting that
\begin{equation*}
\inf_{t \leq \tilde \tau_{b-c}} \tilde \rmS_t - (Q - \beta) t > -b \quad \implies \quad \tilde \sigma_b > \tilde \tau_{b+c}	\;,
\end{equation*}
and combining this fact with the complement of the second event on the right-hand side of \eqref{eq:Step2Comp1}, we readily obtain that it must hold that $\tilde \tau_b - \tilde \sigma_b < \lambda$.

Now, writing $\tilde \tau$ as a shorthand for $\tilde \tau_{b-c}$, we note  that the first event on the right-hand side of \eqref{eq:Step2Comp1} is contained in the following union of events
\begin{equation*}
\underbrace{\bigcup_{n \in \N} \Biggl\{\sup_{t \in [\tilde\tau-n, \tilde\tau - (n-1)]} \abs{\tilde \rmS_t - \tilde{B}_t} \geq (\log n)^2 + c/2\Biggr\}}_{\rmE_1} \cup \underbrace{\bigcup_{n \in \N}\Biggl\{\inf_{t \in [\tilde\tau-n, \tilde\tau-(n-1)] } (\tilde{B}_t - (Q - \beta) t) < (\log n)^2 -(b-c/2)\Biggr\}}_{\rmE_2}\;. 
\end{equation*}
We now proceed to bound the probabilities of the events $\rmE_1$ and $\rmE_2$. We start by noting that, for any $n \in \N$, it holds that  
\begin{equation*}
\sup_{t \in [\tilde \tau - n, \tilde \tau - (n-1)]} \abs{\tilde \rmS_t - \tilde{B}_t} = \sup_{t \in [- n, - (n-1)]}   \abs*{\int_{-\infty}^{0} g(s) (\tilde{B}_{s + t + \tilde \tau} - \tilde{B}_{\tilde \tau + t}) ds - \int_{-\infty}^0 g(s) \tilde{B}_s ds} \;.
\end{equation*}
By using the invariance in law \eqref{eq:IVBM}, it holds that 
\begin{equation*}
\int_{-\infty}^{0} g(s) (\tilde{B}_{s + t + \tilde \tau} - \tilde{B}_{\tilde \tau + t}) ds \eqlaw	\int_{-\infty}^{0} g(s) (\tilde{B}_{s+t} - \tilde{B}_{t}) ds \;, \qquad \forall \, t \in \R \;.
\end{equation*}
Therefore, applying the triangle inequality and the union bound, we obtain that the probability of the event $\rmE_1$ is bounded from above, up to a finite multiplicative constant, by
\begin{equation}
\label{eq:giantUnionStep21}
\sum_{n \in \N} \P\Biggl(\sup_{t \in [- n, - (n-1)]} \abs*{\int_{-\infty}^{0} g(s) (\tilde{B}_{s+t} - \tilde{B}_{t}) ds} \geq (\log n)^2/2 + c/4\Biggr) \lesssim \sum_{n \in \N} e^{- a ((\log n)^2+c)} \;,
\end{equation}
for some constant $a > 0$ and where the implicit constant is independent of everything else. The upper bound on the right-hand side of \eqref{eq:giantUnionStep21} follows from the Gaussian tail bound.
 
On the other hand, using once again the invariance in law \eqref{eq:IVBM} and recalling that $\tilde \tau$ is a shorthand for $\tilde \tau_{b-c}$, for any $n \in \N$, we have that 
\begin{equation*}
\inf_{t \in [\tilde\tau-n, \tilde\tau-(n-1)] } (\tilde{B}_t - (Q- \beta) t) \leq (\log n)^2 -(b-c/2) \eqlaw \inf_{t \in [-n, -(n-1)] } (\tilde{B}_t - (Q-\beta) t) \leq (\log n)^2 -c/2 \;.
\end{equation*}
In particular, using the union bound, we obtain that the probability of the event $\rmE_2$ is bounded from above, up to a finite multiplicative constant, by
\begin{equation}
\label{eq:giantUnionStep22}
\sum_{n \in \N} \P\Biggl(\inf_{t \in [-n, -(n-1)] } \tilde{B}_t \leq -  (Q-\beta) (n-1) + (\log n)^2 -c/2\Biggr) \lesssim \sum_{n \in \N} e^{-a(n + c)} \;,
\end{equation}
for some constant $a > 0$ and where the implicit constant is independent of everything else. Also here, the upper bound on the right-hand side of \eqref{eq:giantUnionStep22} follows from the Gaussian tail bound.

Furthermore, proceeding exactly as in \eqref{eq:TightStep1Eq2}, we have that 
\begin{equation}
\label{eq:giantUnionStep23}
\P\bigl(\tilde\tau_b - \tilde \tau_{b-c} \geq \lambda\bigr) \lesssim e^{- a ((Q-\beta)\lambda -c)^2/\lambda} \;,
\end{equation}
for some constant $a > 0$.
Finally, the conclusion follows thanks to estimates \eqref{eq:giantUnionStep21}, \eqref{eq:giantUnionStep22}, and \eqref{eq:giantUnionStep23} by choosing $c = c(\lambda) = \sqrt{\lambda}$.
\end{proof}

We are now ready to prove Lemma~\ref{lm:tightQCTilde}.
\begin{proof}[Proof of Lemma~\ref{lm:tightQCTilde}]
By recalling the definitions in \eqref{eq:modContRec} and \eqref{eq:modContPlain}, and applying the Arzel\`a--Ascoli theorem, to prove the tightness of $(\tilde \rmS_{b, \cdot})_{b > 0}$ in $\CC^{\cd_d}_{\loc}(\R)$ equipped with the local uniform metric, it suffices to verify that for all $c > 0$ and any $\eps > 0$, it holds that  
\begin{equation*}
\lim_{\delta \to 0} \limsup_{b \to \infty} \P\bigl(\tilde \omega_{b, c}(\delta) > \eps \bigr) = 0 \;.
\end{equation*}
To this end, we note that for $c > 0$ and $\tilde c > c$, on the event $\smash{\{[\tilde \sigma_b - c, \tilde \sigma_b + c] \subseteq [\tilde \tau_b - \tilde c, \tilde \tau_b + \tilde c]\}}$, it holds that 
\begin{align*}
\tilde \omega_{b, c}(\delta)  
& = \sum_{i = 0}^{\cd_d} \sup_{s, t \in [\tilde \sigma_b - c, \tilde \sigma_b + c], \, \abs{t-s} \leq \delta} \abs{\tilde \rmS^{(i)}_{t} - \tilde \rmS^{(i)}_{s}}
\leq \sum_{i = 0}^{\cd_d} \sup_{s, t \in [\tilde \tau_b - \tilde c, \tilde \tau_b + \tilde c], \, \abs{t-s} \leq \delta} \abs{\tilde \rmS^{(i)}_{t} - \tilde \rmS^{(i)}_{s}} \\
& = \sum_{i = 0}^{\cd_d} \sup_{s, t \in [- \tilde c, \tilde c], \, \abs{t-s} \leq \delta} \abs{\tilde \rmS^{(i)}_{\tilde \tau_b + t} - \tilde \rmS^{(i)}_{\tilde \tau_b + s}} 
 \eqlaw \sum_{i = 0}^{\cd_d} \sup_{s, t \in [- \tilde c, \tilde c], \, \abs{t-s} \leq \delta} \abs{\tilde \rmS^{(i)}_{t} - \tilde \rmS^{(i)}_{s}} \;,
\end{align*}
where, in order to get the last equality, we used the invariance in law \eqref{eq:IVSPH}. 
Therefore, this fact implies that, for any $\eps > 0$, 
\begin{equation}
\label{eq:ArAsStep2}
\P\bigl(\tilde \omega_{b, c}(\delta) > \eps \bigr) \leq  \P\bigl(\tilde \omega_{\tilde c}(\delta) > \eps \bigr) + \P\bigl([\tilde \sigma_b - c, \tilde \sigma_b + c] \not\subseteq [\tilde \tau_b - \tilde c, \tilde \tau_b + \tilde c]\bigr) \;.
\end{equation}
The conclusion follows by first taking $\limsup_{b \to \infty}$, then $\lim_{\delta \to 0}$, and finally $\limsup_{\tilde{c} \to \infty}$ in the above expression. Indeed, since for any fixed $\tilde c > 0$ it holds that $\lim_{\delta \to 0} \P(\tilde \omega_{\tilde c}(\delta) > \eps) = 0$, it suffices to apply Lemma~\ref{lm:StoppingClose} to obtain the desired result.
\end{proof}

\subsubsection{The collection $(\rmS_{b, \cdot})_{b > 0}$ is tight}
We now prove that tightness of $(\tilde \rmS_{b, \cdot})_{b > 0}$ implies tightness of $(\rmS_{b, \cdot})_{b > 0}$, from which Proposition~\ref{pr:tightQC} follows directly.  
We start by proving that the stopping times $\sigma_b$ and $\tilde \sigma_b$ are close to each other. 
\begin{lemma}
\label{lm:StoppingCloseFin}
For $b > 0$ let $\sigma_b$ and $\tilde \sigma_b$ as defined in \eqref{eq:defTb} and in \eqref{eq:defTildeTb}, respectively. Then, it holds that  
\begin{equation*}
\lim_{\lambda \to \infty} \limsup_{b \to \infty} \P\bigl(\abs{\sigma_b - \tilde \sigma_b} \geq \lambda \bigr) = 0 \;.
\end{equation*}
\end{lemma}
\begin{proof}
For $\lambda \geq 0$, we have the following trivial inclusion
\begin{equation*}
\bigl\{\abs{\sigma_b -  \tilde \sigma_b} \geq \lambda \bigr\} \subseteq \bigl\{\sigma_b - \tilde \sigma_b \geq \lambda \bigr\} \cup \bigl\{\tilde \sigma_b - \sigma_b \geq \lambda \bigr\} \;,
\end{equation*}
and so we can focus on finding suitable upper bounds for the probabilities of the two events on the right-hand side of the above display. Throughout the proof we can assume that that the two-sided Brownian motion $(B_t)_{t \in \R}$ used in the construction of the spherical average process $(\rmS_t)_{t \in \R}$ and the process $(\tilde B_t)_{t \in \R}$ introduced in \eqref{eq:defBMQC} are coupled in such a way that $B_t = \tilde B_t$ for all $t \geq 0$.

We start by observing that, for all $s \in \R$, one has that 
\begin{align*}
\rmS_s - \tilde \rmS_s 
= \int_{-\infty}^{0} g(r-s)\bigl(B_{r} - \tilde B_{r}\bigr) dr - \underbrace{\int_{-\infty}^0 g(r)\bigl(B_{r} - \tilde B_r\bigr) dr}_{\rmD} \;.
\end{align*}
In particular, by using the exponential decay of $g(s)$ as $s \to -\infty$, it holds that
\begin{equation}
\label{eq:CloseTildeandNot}
\lim_{t \to \infty} \P\Biggl(\sup_{s \geq t} \abs{\rmS_s - \tilde \rmS_s + \rmD} > 1\Biggr) = 0 \;.
\end{equation}
Furthermore, thanks to the tightness of $(\tilde \rmS_{b, \cdot})_{b > 0}$ proved in Proposition~\ref{lm:tightQCTilde}, for any $c > 0$, it holds that 
\begin{equation}
\label{eq:tightStopTilde}
\lim_{\lambda \to \infty} \limsup_{b \to \infty} \P\bigl(\tilde \sigma_{b+c} - \tilde \sigma_b > \lambda\bigr) = 0 \qquad \text{ and } \qquad \lim_{\lambda \to \infty} \limsup_{b \to \infty} \P\bigl(\tilde \sigma_{b} - \tilde \sigma_{b-c} > \lambda\bigr) = 0 \;.
\end{equation}

\textbf{Step 1:} We start by bounding the probability of the event $\{\sigma_b - \tilde \sigma_b \geq \lambda\}$. To this end, by letting  $c = \abs{\rmD} + 1$, we note that
\begin{equation}
\label{eq:Step1Comp2}
\{\sigma_b - \tilde \sigma_b \geq \lambda\} \subseteq \Biggl\{\sup_{s \geq \tilde \sigma_b} \abs{\rmS_s - \tilde \rmS_s + \rmD} > 1\Biggr\} \cup \bigl\{\tilde \sigma_{b + c} - \tilde \sigma_b \geq \lambda\bigr\} \;.	
\end{equation}
To see why this is true, it is more convenient to consider the complement of the given relation. Indeed, we note that
\begin{equation*}
\sup_{s \geq \tilde \sigma_b} \abs{\rmS_s - \tilde \rmS_s + \rmD} \leq 1 \quad \implies \quad \rmS_{\tilde \sigma_{b + c}} - \tilde \rmS_{\tilde \sigma_{b + c}} \leq \abs{\rmD} + 1 \;,
\end{equation*}
from which we can deduce that 
\begin{equation*}
\tilde \rmS_{\tilde \sigma_{b+c}} - (Q-\beta) \tilde \sigma_{b+c} = -(b+c)  \implies \rmS_{\tilde \sigma_{b+c}} - (Q-\beta) \tilde \sigma_{b+c} \leq -(b+c) + \abs{\rmD} + 1 \implies \rmS_{\tilde \sigma_{b+c}} - (Q-\beta) \tilde \sigma_{b+c} \leq - b \;,
\end{equation*}
or, in other words, that $\sigma_b \leq \tilde\sigma_{b+c}$. Combining this fact with the fact that we are on the complement of the second event on the right-hand side of \eqref{eq:Step1Comp2}, we readily obtain that $\sigma_b - \tilde \sigma_b < \lambda$. Therefore, the desired result follows by observing that, thanks to \eqref{eq:CloseTildeandNot} and \eqref{eq:tightStopTilde}, it holds that 
\begin{equation*}
\limsup_{b \to \infty} \P\Biggl(\sup_{s \geq \tilde \sigma_b} \abs{\rmS_s - \tilde \rmS_s + \rmD} > 1\Biggr) = 0 \qquad \text{ and } \qquad \lim_{\lambda \to \infty} \limsup_{b \to \infty} \P\bigl(\tilde \sigma_{b + c} - \tilde \sigma_b \geq \lambda\bigr) = 0 \;,
\end{equation*}
where here we implicitly used the fact that $\tilde \sigma_b \to \infty$ in probability as $b \to \infty$. 

\textbf{Step 2:} We now focus on bounding the probability of the event $\{\tilde \sigma_b - \sigma_b \geq \lambda\}$. To this end, by letting  $c = \abs{\rmD} + 1$ and by fixing $T > 0$, we note that
\begin{equation}
\label{eq:Step2Comp2}
\{\tilde \sigma_b - \sigma_b \geq \lambda\} \subseteq \Biggl\{\inf_{s \in [0,T]} \bigl(\rmS_s - (Q - \beta) s\bigr) \leq - b \Biggr\} \cup \Biggl\{\sup_{s \geq T} \abs{\rmS_s - \tilde \rmS_s + \rmD} > 1\Biggr\} \cup \bigl\{\tilde \sigma_{b} - \tilde \sigma_{b-c} \geq \lambda\bigr\} \;.	
\end{equation}
Also here, to see why this is true, it is more convenient to consider the complement of the given relation. Indeed, we note that by proceeding analogously to the previous step, one can easily verify that on the intersection of the complement of first two events on the right-hand side of \eqref{eq:Step2Comp2}, it holds that 
\begin{equation*}
\inf_{t \in [0, \tilde \sigma_{b-c}]} \bigl(\rmS_s - (Q - \beta) s\bigr) \geq -b \qquad \implies \qquad \sigma_b \geq \tilde \sigma_{b-c} \;.
\end{equation*}
Combining this fact with the fact that we are on the complement of the third event on the right-hand side of \eqref{eq:Step2Comp2}, we readily obtain that $\tilde \sigma_b - \sigma_b < \lambda$. Therefore, the desired result follows by observing that, thanks to \eqref{eq:CloseTildeandNot} and \eqref{eq:tightStopTilde}, it holds that
\begin{gather*}
\lim_{T \to \infty} \P\Biggl(\sup_{s \geq T} \abs{\rmS_s - \tilde \rmS_s + \rmD} > 1\Biggr) = 0 \qquad \text{ and } \qquad
\lim_{\lambda \to \infty} \limsup_{b \to \infty} \P\bigl(\tilde \sigma_{b} - \tilde \sigma_{b-c} \geq \lambda\bigr) = 0 \;,
\end{gather*}
and from the fact that for any fixed $T> 0$,
\begin{equation*}
\limsup_{b \to \infty} \P\Biggl(\inf_{s \in [0,T]} \bigl(\rmS_s - (Q - \beta) s\bigr) \leq - b\Biggr) = 0 \;.
\end{equation*}
\end{proof}

We are finally ready to prove tightness of of the collections of processes $(\rmS_{b, \cdot})_{b > 0}$, i.e., to prove Proposition~\ref{pr:tightQC}. 

\begin{proof}[Proof of Proposition~\ref{pr:tightQC}]
The proof relies on similar ideas to those used in the proof of Lemma~\ref{lm:tightQCTilde}. Indeed, by recalling the definitions in \eqref{eq:modContRec} and \eqref{eq:modContPlain}, and applying the Arzel\`a--Ascoli theorem, to prove the tightness of $(\rmS_{b, \cdot})_{b > 0}$ in $\CC^{\cd_d}_{\loc}(\R)$ equipped with the local uniform metric, it suffices to verify that for all $c > 0$ and any $\eps > 0$, it holds that  
\begin{equation*}
\lim_{\delta \to 0} \limsup_{b \to \infty} \P\bigl(\omega_{b, c}(\delta) > \eps \bigr) = 0 \;.
\end{equation*}
For $i \in [\cd_d]_0$, we introduce the process $\smash{(\rmD^{(i)}_t)_{t \in \R}}$ defined as 
\begin{equation*}
\rmD^{(i)}_t \eqdef \int_{-\infty}^{0} g^{(i)}(s) \bigl(B_{s+t} - \tilde B_{s+t}\bigr) ds\;, \qquad \forall \, t \in \R\;.
\end{equation*}
Proceeding similarly to the proof of Lemma~\ref{lm:StoppingCloseFin}, and by using the exponential decay of $g^{(i)}(s)$ as $s \to \infty$ for all $i \in [\cd_d]_0$, it is straightforward to check that, for any $\eps > 0$,
\begin{equation}
\label{eq:tight1Aux}
\lim_{t \to \infty} \P\Biggl(\sum_{i = 0}^{\cd_d}\sup_{s \geq t} \abs{\rmD^{(i)}_s} > \eps \Biggr) = 0 \;.
\end{equation}
Now, for $b > 0$ and $\tilde c > c$, on the following event 
\begin{equation*}
\Biggl\{\sum_{i = 0}^{\cd_d} \sup_{s \geq \sigma_b - c} \abs{\rmD^{(i)}_s} \leq \frac{\eps}{4}\Biggr\} \cap \bigl\{[\sigma_b - c, \sigma_b + c]  \subseteq [\tilde \sigma_b - \tilde c, \tilde  \sigma_b + \tilde c] \bigr\} \;,
\end{equation*}
one has that
\begin{align*}
\omega_{b, c}(\delta) 
& = \sum_{i = 0}^{\cd_d} \sup_{s, t \in [\sigma_b - c, \sigma_b + c], \, \abs{t-s} \leq \delta} \abs{\rmS^{(i)}_{t} - \rmS^{(i)}_{s}} 
 \leq \sum_{i = 0}^{\cd_d} \sup_{s, t \in [\sigma_b - c, \sigma_b + c], \, \abs{t-s} \leq \delta} \abs{\tilde \rmS^{(i)}_{t} - \tilde \rmS^{(i)}_{s}} + \frac{\eps}{2}\\
& \leq \sum_{i = 0}^{\cd_d} \sup_{s, t \in [\tilde \sigma_b - \tilde c, \tilde \sigma_b  + \tilde c], \, \abs{t-s} \leq \delta} \abs{\tilde \rmS^{(i)}_{t} - \tilde \rmS^{(i)}_{s}} + \frac{\eps}{2}
 = \sum_{i = 0}^{\cd_d} \sup_{s, t \in [- \tilde c, \tilde c], \, \abs{t-s} \leq \delta} \abs{\tilde \rmS^{(i)}_{b, t} - \tilde \rmS^{(i)}_{b, s}} + \frac{\eps}{2} \;. 
\end{align*}
In particular, this fact implies that 
\begin{align}
& \P\bigl(\omega_{b, c}(\delta) > \eps\bigr) \nonumber \\
& \hspace{10mm} \leq \P\bigl(\tilde \omega_{b, \tilde c}(\delta) > \eps/2\bigr) + \P\Biggl(\sum_{i = 0}^{\cd_d} \sup_{s \geq \sigma_b - c} \abs{\rmD^{(i)}_s} > \frac{\eps}{4}\Biggr) + \P\bigl([\tilde \sigma_b - \tilde c, \tilde  \sigma_b + \tilde c] \not\subseteq [\sigma_b - c, \sigma_b + c]\bigr) \;. \label{eq:ArAsStep1}
\end{align}
The conclusion follows by first taking $\limsup_{b \to \infty}$, then $\lim_{\delta \to 0}$, and finally $\limsup_{\tilde{c} \to \infty}$ in the above expression.
Specifically, the vanishing of the first and third probabilities on the right-hand side of \eqref{eq:ArAsStep1} is a consequence of Lemma~\ref{lm:tightQCTilde} and Lemma~\ref{lm:StoppingCloseFin}, respectively. The vanishing of the second probability on the right-hand side of \eqref{eq:ArAsStep1} follows from \eqref{eq:tight1Aux} and the fact that $\sigma_b \to \infty$ in probability as $b \to \infty$.
\end{proof}

\subsection{Convergence of the recentred process}
\label{sub:uniq}
The main goal of this section is to prove Theorem~\ref{th:convRecentred}. 
To this end, we consider a process $(\rmR_t)_{t \in \R}$ with the same law of the spherical average process $(\rmS_t)_{t \in \R}$. For each $\frkb > 0$, we introduce the following quantities
\begin{equation*}
\rho_{\frkb} \eqdef \inf\bigl\{t \in \R \, : \, \rmR_t - (Q - \beta) t = -\frkb \bigr\} \;, \qquad \rmR_{\frkb, t} \eqdef \rmR_{t+\rho_{\frkb}} - \rmR_{\rho_{\frkb}}\;, \quad \forall \, t \in \R\;.
\end{equation*}
For $\frkb$, $b > 0$, we also introduce the following stopping time 
\begin{equation}
\label{eq:defDobleStop}
\rho_{\frkb, b} \eqdef \inf\bigl\{t \in \R \, : \, \rmR_{\frkb, t} - (Q-\beta)t = - b \bigr\} \;.
\end{equation}

\begin{lemma}
\label{lm:distanceLawOriginal}
For any $\eps > 0$, there exists $u \geq 0$ large enough such that for all $\frkb > 0$, it holds that 
\begin{equation*}
	\dTV\bigl[(\rmR_{\frkb, u + t})_{t \geq 0}, (\rmS_{u + t})_{t \geq 0} \bigr] \leq \eps \;.
\end{equation*}
\end{lemma}
\begin{proof}
It suffices to show that there exists $u \geq 0$ large enough such that for all $\frkb > 0$, the conditional law of $\smash{(\rmR_{\frkb, u + t})_{t \geq 0}}$ given $\smash{(\rmR_{t})_{t \leq \rho_{\frkb}}}$ is close in total variation distance to the law of $\smash{(\rmS_{u + t})_{t \geq 0}}$. Since $\smash{\rho_{\frkb}}$ is a stopping time for $\smash{(\rmR_{t})_{t \in \R}}$, we have that the conditional law of $\smash{(\rmR_{\frkb, u + t})_{t \geq 0}}$ given $\smash{(\rmR_{t})_{t \leq \rho_{\frkb}}}$ depends only on $\smash{(\rmR_{\rho_{\frkb}}, \mathbf{R}_{\rho_{\frkb}})}$, where $(\mathbf{R}_t)_{t \in \R}$ denotes the vector of derivatives of $(\rmR_t)_{t \in \R}$ up to order $\cd_d$. Moreover, recalling that $(\rmR_t)_{t \in \R}$ has the same law of $(\rmS_t)_{t \in \R}$, then thanks to Proposition~\ref{pr:tightQC}, the collection $\smash{(\rmR_{\rho_{\frkb}}, \mathbf{R}_{\rho_{\frkb}})_{\frkb > 0}}$ is tight in $\CC^{\cd_d}_{\loc}(\R)$ equipped with the local uniform metric\footnote{We emphasise that this is the sole instance where the tightness of $(\rmS_{b, \cdot})_{b > 0}$ (and consequently of $(\rmR_{\frkb, \cdot})_{\frkb > 0}$) is used.}. The conclusion is then a consequence of Lemma~\ref{lm:TVdistannceSpherical}.
\end{proof}

We are now ready to prove the desired convergence result. 
\begin{proof}[Proof of Theorem~\ref{th:convRecentred}]
Fix $\rmT \in  \R$. Thanks to Lemma~\ref{lm:distanceLawOriginal}, for $u \geq 0$ large enough, we can find a coupling of $(\rmR_t)_{t \in \R}$ and $(\rmS_t)_{t \in \R}$ such that, for all $\frkb > 0$, with probability at least $1-\eps$, it holds that  
\begin{equation}
\label{eq:couplingTechUniq}
\rmR_{\frkb, u + t} = \rmS_{u+t} \;, \qquad \forall \, t \geq 0 \;.
\end{equation}
By choosing $b > 0$ sufficiently large (depending on $u$ but not on $\frkb$), it holds with probability at least $1-\eps$ that $\rho_{\frkb, b} = \sigma_{b}$. Indeed, this follows since, by taking $b$ large enough (depending on $u$ but not on $\frkb$), we can ensure that there is not a time $t < u$ such that $\rmS_t - (Q-\beta) t = -b$. Moreover, by the definition of $\rho_{\frkb, b}$, it holds that
\begin{equation*}
\rmS_{\rho_{\frkb, b}} - (Q - \beta) \rho_{\frkb, b} \overset{\eqref{eq:couplingTechUniq}}{=} \rmR_{\frkb, \rho_{\frkb, b}}  - (Q - \beta) \rho_{\frkb, b} =  - b \;.
\end{equation*}
In particular, this fact implies that if $b > 0$ is large enough (not depending on $\frkb$), then with probability at least $1-\eps$, it holds that 
\begin{equation*}
\rmR_{\frkb, \rho_{\frkb, b} + t} = \rmS_{\sigma_{b} + t} \;, \qquad \forall \, t \geq \rmT \;.
\end{equation*}
Therefore, by noting that $\rho_{\frkb, b} + \rho_{\frkb} = \rho_{b + \frkb}$, we can use the previous identity to write 
\begin{equation*}
\rmR_{\rho_{b + \frkb} + t} - \rmR_{\rho_{b+\frkb}} = \rmS_{\sigma_{b} + t} - \rmS_{\sigma_b}\;, \qquad \forall \, t \geq \rmT \;.
\end{equation*}
This fact implies that the collection $((\rmS_{b, t})_{t \geq \rmT})_{b > 0}$ is Cauchy in total variation distance, and so the conclusion follows.
\end{proof}

\subsection{The quantum cone as a local limit}
\label{sub:quantumLocal}
In this short section, we prove that the $\beta$-quantum cone field can be identified as a local limit of a whole-space LGF with an additional $\beta$-log singularity.

\begin{proof}[Proof of Theorem~\ref{th:convQuantum}]
Consider the setting described in the theorem's statement. For $t \in \R$, $z \in \R^d$, and $b > 0$, we have that
\begin{equation*}
\bh^{t, b}_{e^{-s}}(0) = \bh_{e^{-(s+t)}}(z) + \beta s - (Q - \beta)t  + b \;, \qquad \forall \, s \in \R \;. 	
\end{equation*}
Thus, by definition \eqref{eq:defSigmaBTh} of $\sigma_b$, and recalling the definitions in \eqref{eq:defVecDer} and \eqref{eq:defRecProc}, we obtain that
\begin{equation*}
\bh^{\sigma_b, b}_{e^{-s}}(0) = \bh_{e^{-(s+\sigma_b)}}(z) + \beta s - (Q - \beta)\sigma_b  + b = \bh_{e^{-(s+\sigma_b)}}(z) - \bh_{e^{-\sigma_b}}(z)  + \beta s \eqlaw \rmS_{b, s} + \beta s  \;, \qquad \forall \, s \in \R \;.	
\end{equation*}
Hence, thanks to Theorem~\ref{th:convRecentred}, if we take the limit as $b \to \infty$ in the previous display, we readily obtain that, for any $\rmT \in \R$, the process $(\bh^{\sigma_b, b}_{e^{-s}}(0))_{s \geq \rmT}$ converges in total variation distance to the process $\smash{(\rmS_{\infty, s} + \beta s)_{s \geq \rmT}}$ as $b \to \infty$ . The desired result follows since the shifting and rescaling procedure we performed do not affect the law of the projection of $\bh$ onto the subspace $\smash{\H_{d, \sph}}$.
\end{proof}

\subsection{Interaction between LBM and quantum cone}
\label{sub:LBMQuantum}
The main goal of this section is to prove Theorem~\ref{th:invarianceShift}, i.e., to prove that the law of the $d$-dimensional quantum cone $\bh^{\star}$ is invariant under shifting along the trajectories of the LBM. 

For $d > 2$ (resp.\ $d = 2$) even, $Q > \sqrt{2d}$, and $\beta \in (-\infty, Q)$, let $\bh^{\star}$ be the $d$-dimensional $\beta$-quantum cone field as specified in Definition~\ref{def:quantumCone} (resp.\ \cite[Definition~4.10]{DMS21}), and let $\rmB_{\bh^{\star}\!, \alpha}$ be the associated LBM as introduced in Definition~\ref{def:LBMQuantum}. Before proceeding,
we need to rule out that $\rmB_{\bh^{\star}\!, \alpha}$ does not stay ``stuck'' at zero due to the $\beta$-log singularity at the origin of $\bh^{\star}$. To confirm this fact, we note that the expected time for $\rmB_{\bh^{\star}\!, \alpha}$ to reach the boundary of the unit disk is given by $\BE_0[\nu_{\bh^{\star}\!, \alpha}(B(0, 1))]$. Hence, it suffices to verify that it holds $\P$-almost surely that $\BE_0[\nu_{\bh^{\star}\!, \alpha}(B(0, 1))] < \infty$. To this end, we need the following intermediate result related to the growth of the radial part of the quantum cone field. 

\begin{lemma}
\label{lm:boundGrowthRadialCone}
For $d \geq 2$ even, $Q > \sqrt{\smash[b]{2d}}$, and $\beta \in (-\infty , Q)$, consider the collection of recentred processes $(\rmS_{b, \cdot})_{b > 0}$ defined in \eqref{eq:defRecProc}. Then, it holds that 
\begin{equation*}
\lim_{t \to \infty} \limsup_{b \to \infty} \P\Biggl(\sup_{s \geq t} \abs{\rmS_{b, s}} s^{-3/4} > 1\Biggr) = 0 \;,
\end{equation*}
where we recall that the recentred process $(\rmS_{b, s})_{s \in \R}$ is defined in \eqref{eq:defRecProc}. 
\end{lemma}
\begin{proof}
We focus on the case $d > 2$ even, as the case $d = 2$ is trivial. For $t \geq 0$ and $b > 0$, we introduce the following events,
\begin{equation*}
\rmE_{1, b, t} \eqdef \Biggl\{\sup_{s \geq t} \abs{\rmS_{b, s} - (\tilde \rmS_{s + \sigma_{b}} -  \tilde \rmS_{\sigma_b})} < 1 \Biggr\}, \qquad \rmE_{2, b, t} \eqdef \bigl\{\abs{\sigma_b - \tilde \sigma_b} < t\bigr\}, \qquad \rmE_{3, b, t} \eqdef \bigl\{\abs{\tilde \sigma_b - \tilde \tau_b} < t\bigr\} \;,
\end{equation*}
where we recall that $\sigma_b$ is defined in \eqref{eq:defTb}, $(\tilde\rmS_s)_{s \in \R}$ is defined in \eqref{eq:defrmStilde}, and $\tilde \tau_b$ and $\tilde \sigma_b$ are defined in \eqref{eq:defTildeTb}.
We claim that it suffices to estimate the probability of the event in the statement, restricted to the event $\cap_{i=1}^3 \rmE_{i, b, t}$. Indeed, by applying Lemma~\ref{lm:StoppingClose} and Lemma~\ref{lm:StoppingCloseFin}, and following reasoning similar to the proof of \eqref{eq:CloseTildeandNot}, it follows that
\begin{equation*}
\lim_{t \to \infty} \limsup_{b \to \infty}\sum_{i = 1}^{3} \P\bigl(\rmE^c_{i, b, t}\bigr) = 0 \;.
\end{equation*} 
We start by observing that, on the event $\rmE_{1, b, t}$, it holds that
\begin{equation*}
\sup_{s \geq t} \abs{\rmS_{b, s}}  \leq 	\sup_{s \geq t} \abs{\tilde \rmS_{s + \sigma_{b}} -  \tilde \rmS_{\sigma_b}} + 1 \;,
\end{equation*}
and so, we can focus our attention on the following event, 
\begin{equation*}
\Biggl\{\sup_{s \geq t} \abs{\tilde \rmS_{s + \sigma_{b}} -  \tilde \rmS_{\sigma_b}} s^{-3/4} > 1, \; \bigcap_{i = 1}^3 \rmE_{i, b, t}\Biggr\}\;.
\end{equation*}
To this end, we begin by observing that, thanks to the triangle inequality,
\begin{equation}
\label{eq:ZeroTermTailsCQ}
\sup_{s \geq t} \abs{\tilde \rmS_{s +\sigma_{b}} -  \tilde \rmS_{\sigma_b}} \leq 
\sup_{s \geq t} \abs{\tilde \rmS_{s + \tilde \sigma_b} - \tilde \rmS_{\tilde \sigma_b}} 
+ \sup_{s \geq t} \abs{\tilde \rmS_{s + \tilde \sigma_{b}} -  \tilde \rmS_{s + \sigma_b}} 
+ \abs{\tilde \rmS_{\tilde \sigma_{b}} -  \tilde \rmS_{\sigma_b}} \;.
\end{equation}
Next, we proceed to bound each of the three terms on the right-hand side of \eqref{eq:ZeroTermTailsCQ} individually.

\textbf{First term:} 
By applying the triangle inequality, on the event $\rmE_{3, b, t}$, we observe that the first term on the right-hand side of \eqref{eq:ZeroTermTailsCQ} is bounded above by
\begin{equation}
\label{eq:firstTermTailsCQ}
\sup_{s \geq t} \abs{\tilde \rmS_{s + \tilde \tau_{b}} -  \tilde \rmS_{\tilde \tau_b}}  + \sup_{s \geq t} \sup_{u \in [-t, t]} \abs{\tilde \rmS_{s + u + \tilde \tau_{b}} -  \tilde \rmS_{s + \tilde \tau_b}} + \sup_{u \in [-t, t]} \abs{\tilde \rmS_{u + \tilde \tau_{b}} -  \tilde \rmS_{\tilde \tau_b}} \;. 
\end{equation}
We can analyse the three terms on the right-hand side of the above display separately. Beginning with the first term, leveraging the invariance in law \eqref{eq:IVSPH}, we observe that
\begin{equation*}
\P\Biggl(\sup_{s \geq t} \abs{\tilde \rmS_{s + \tilde \tau_b} - \tilde \rmS_{\tilde \tau_b}} s^{-3/4} > 1\Biggr)
\leq \sum_{n = \lfloor t \rfloor}^{\infty} \P\Biggl(\sup_{s \in [n, n+1]} \abs{\tilde \rmS_{s}} > n^{3/4}\Biggr) = \sum_{n = \lfloor t \rfloor}^{\infty} \P\Biggl(\sup_{s \in [n, n+1]} \abs{\tilde \rmS_{s}} > n^{3/4}\Biggr) \;.
\end{equation*}
By first taking $\limsup_{b \to \infty}$ and then $\lim_{t \to \infty}$ in the above inequality, we deduce that the probability on the left-hand side converges to zero. This conclusion follows directly from the Gaussian tail bound and the fact that the function $g$, used in the definition \eqref{eq:defrmStilde} of the process $(\tilde \rmS_s)_{s \in \R}$, decays to zero exponentially fast as $s \to -\infty$. The remaining two terms in \eqref{eq:firstTermTailsCQ} can be handled in a similar manner.

\textbf{Second term:} Proceeding similarly to the first term and applying the triangle inequality, on the event $\rmE_{2, b, t} \cap \rmE_{3, b, t}$, we observe that the second term on the right-hand side of \eqref{eq:ZeroTermTailsCQ} is bounded above by
\begin{equation*}
2 \sup_{s \geq t} \sup_{u \in [-t, t]} \abs{\tilde \rmS_{s + u + \tilde \tau_{b}} - \tilde \rmS_{s + \tilde \tau_{b}}} + \sup_{s \geq t} \sup_{u \in [-t, t]} \sup_{r \in [-t, t]} \abs{\tilde \rmS_{s + u + r + \tilde \tau_{b}} -  \tilde \rmS_{s + u + \tilde \tau_{b}}} \;.
\end{equation*}
Now, leveraging the invariance in law \eqref{eq:IVSPH}, we can control each term in the above display by following an approach analogous to that used for the first term.

\textbf{Third term:} Proceeding similarly to the second term and applying the triangle inequality, on the event $\rmE_{2, b, t} \cap \rmE_{3, b, t}$, we observe that the third term on the right-hand side of \eqref{eq:ZeroTermTailsCQ} is bounded above by
\begin{equation*}
2 \sup_{u \in [-t, t]} \abs{\tilde \rmS_{u + \tilde \tau_{b}} - \tilde \rmS_{\tilde \tau_{b}}} + \sup_{u \in [-t, t]} \sup_{r \in [-t, t]} \abs{\tilde \rmS_{u + r + \tilde \tau_{b}} -  \tilde \rmS_{u + \tilde \tau_{b}}} \;.
\end{equation*}
Once again, leveraging the invariance in law \eqref{eq:IVSPH}, we can control each term in the above display by proceeding in a manner analogous to the first term. Hence, the conclusion follows.
\end{proof}

We now proceed to prove the following result which guarantees that $\P$-almost surely, the LBM $\rmB_{\bh^{\star}\!, \alpha}$ does not stay stuck at the origin. 
\begin{lemma}
\label{lm:controlTimeLBMQuantum}
Let $d \geq 2$ be even, $Q > \sqrt{\smash[b]{2d}}$, $\beta \in (-\infty, Q)$, and let $\bh^{\star}$ denote the law of the $d$-dimensional $\beta$-quantum cone. Then, it holds $\P$-almost surely that, 
\begin{equation*}
\BE_0\bigl[\nu_{\bh^{\star}\!, \alpha}\bigl(B(0, 1)\bigr)\bigr] < \infty \;.
\end{equation*}
\end{lemma}
\begin{proof}
We focus on the case $d > 2$ even, as the case $d = 2$ can be treated analogously. Proceeding as in Section~\ref{sec:LMB} and recalling Defintion~\ref{def:quantumCone}, we have that 
\begin{align*}
\E\bigl[\BE_0\bigl[\nu_{\bh^{\star}\!, \alpha}\bigl(B(0, 1)\bigr) \bigr]\, \big| \, \bh^{\star}_{\rad}\bigr] 
& \lesssim \E\Biggl[\int_{B(0, 1)} \abs{x}^{2-d} \mu_{\bh^{\star}\!, \alpha}(dx) \, \Bigg| \, \bh^{\star}_{\rad} \Biggr] \\
& = \E\Biggl[\int_{B(0, 1)} \abs{x}^{2-d} e^{\alpha \bh^{\star}_{\rad}(x)} \mu_{\bh_{\sph}^{\star}, \alpha}(dx) \, \Bigg| \, \bh^{\star}_{\rad} \Biggr] \\
& = \sum_{k = 0}^{\infty} \E\Biggl[\int_{B(0, e^{-k}) \setminus B(0, e^{-(k+1)})} \abs{x}^{2-d} e^{\alpha \bh^{\star}_{\rad}(x)} \mu_{\bh_{\sph}^{\star}, \alpha}(dx) \, \Bigg| \, \bh^{\star}_{\rad} \Biggr] \\
& \leq \sum_{k = 0}^{\infty} e^{k(d-2)} e^{\alpha \sup_{t \in [k, k+1]}(\rmS_{\infty, t} + \beta t)} \E\bigl[\mu_{\bh_{\sph}^{\star}, \alpha}\bigl(B(0, e^{-k})\bigr)\bigr] \\
& \lesssim \sum_{k = 0}^{\infty} e^{-k(2+\alpha^2/2 - \alpha \beta) + \sup_{t \in [k, k+1]}\rmS_{\infty, t}} \E\bigl[\mu_{\bh_{\sph}^{\star}, \alpha}\bigl(B(0, 1)\bigr)\bigr] \;,
\end{align*}
where the last inequality follows from the coordinate change formula Proposition~\ref{pr:coordinate_change_measure}. Therefore, the conclusion follows thanks to Lemma~\ref{lm:boundGrowthRadialCone} and the fact that for all $Q > \sqrt{2d}$ and $\beta \in (-\infty, Q)$, it holds that $2+\alpha^2/2 - \alpha \beta > 0$.
\end{proof}

We can now focus on the proof of Theorem~\ref{th:invarianceShift}. We begin with the following proposition, which can be regarded as an extended version of Theorem~\ref{th:convQuantum}. 
\begin{proposition}
\label{pr:augmJointConv}
Let $d \geq 2$ be even, $Q > \sqrt{\smash[b]{2d}}$, and $\beta \in (-\infty, Q)$.  For $t \in \R$, $z \in \R^d$, and $b > 0$, consider the field $\smash{\bh^{t, b}}$ introduced in the statement of Theorem~\ref{th:convQuantum} and let $\sigma_b$ be as defined in \eqref{eq:defSigmaBTh}.
Furthermore, let $\smash{\rmB_{\bh - \alpha \log\abs{\cdot - z}, \alpha, \cdot}}$ be a two-sided LBM associated to the field $\smash{\bh - \alpha \log\abs{\cdot - z}}$ started from $z$ at time $0$, and let $\rmB_{\bh^{\star}\!, \alpha, \cdot}$ be a two-sided LBM associated to a $d$-dimensional $\alpha$-quantum cone field $\bh^{\star}$ started from the origin at time $0$. Then, as $b \to \infty$, it holds that 
\begin{equation*}
\bigl(\smash{\bh^{\sigma_b, b}}, e^{\sigma_b} (\rmB_{\bh - \alpha\log\abs{\cdot - z}, \alpha, e^{-\alpha b} \cdot}-z)\bigr) \overset{\law}{\longrightarrow} (\bh^{\star}\!, \rmB_{\bh^{\star}, \alpha, \cdot})\;,
\end{equation*}
with respect to the product topology, where the first component converges in the space of distributions, and the second component converges in $\CC(\R)$ equipped with the local uniform metric. 
\end{proposition}
\begin{proof}
To simplify the notation, we assume, without loss of generality, that $z = 0$, and define $\bar{\bh} \eqdef \bh - \alpha \log\abs{\cdot}$.
For $b > 0$, we define $\smash{(\bar B_t)_{t \in \R} \eqdef (e^{\sigma_b} B_{e^{-2 \sigma_b} t})_{t \in \R}}$ and consider the dilation $\smash{\phi_{b}: \R^d \to \R^d}$ given by $\smash{\phi_{b}(x) = e^{-\sigma_b} x}$. Then, by arguing similarly to the proof of Proposition~\ref{pr:conformalLBM}, it holds that 
\begin{equation*}
e^{\sigma_b} \rmB_{\bar \bh, \alpha, e^{-\alpha b} t} = e^{\sigma_b} B_{\rmF^{-1}_{\bar \bh, \alpha}(e^{-\alpha b} t)}  = \bar{B}_{e^{2 \sigma_b} \rmF^{-1}_{\bar \bh, \alpha}(e^{-\alpha b} t)} = \bar{B}_{\rmF^{-1}_{\bar \bh, \alpha, b}(t)}\;, \qquad \forall \, t \in \R \;,
\end{equation*}
where we set $\rmF_{\bar \bh, \alpha, b}(t) \eqdef e^{\alpha b} \rmF_{\bar \bh, \alpha}(e^{-2\sigma_b}t)$. Furthermore, thanks to Proposition~\ref{pr:coordinate_change_clock}, for all $t \in \R$, we observe that
\begin{align*}
\rmF_{\bar \bh, \alpha, b}(t) = e^{\alpha b} \nu_{\bar \bh, \alpha}\bigl([B]_{e^{-2\sigma_b}t}\bigr) \overset{\eqref{eq:coordinate_change_clock}}{=} e^{\alpha b} \nu_{\bar \bh \circ \phi_{b} + Q \log\abs{\phi_b'} , \alpha}\bigl(e^{\sigma_b} [B]_{e^{-2\sigma_b}t}\bigr) = \bar{\nu}_{\bh^{\sigma_b, b}\!, \alpha}\bigl([\bar{B}]_{t}\bigr) \eqdef \bar\rmF_{\bh^{\sigma_b, b}\!, \alpha}(t)\;,
\end{align*}
where $\smash{\bar{\nu}_{\bh^{\sigma_b, b}, \alpha}}$ denotes the GMC with respect to the field $\smash{\bh^{\sigma_b, b}}$ along the trajectory of the Brownian motion $\smash{(\bar{B}_t)_{t \in \R}}$. 
In particular, for any $b > 0$ and $t \geq 0$, we showed that  
\begin{equation*}
\bar \rmB_{\bh^{\sigma_b, b}\!,  \alpha, t} \eqdef \bar B_{\bar\rmF^{-1}_{\bh^{\sigma_b, b}\!, \alpha}(t)} \quad \implies \quad e^{\sigma_b} \rmB_{\bar \bh, \alpha, e^{-\alpha b} t} = \bar \rmB_{\bh^{\sigma_b, b}\!,  \alpha, t} \;.
\end{equation*}
Now, for any $R > 0$, consider the stopping times $\bar \tau_{R, +}$ and $\bar \tau_{R, -}$ defined as 
\begin{equation*}
\bar \tau_{R, +} \eqdef \inf\bigl\{t \geq 0 \, : \, \bar{B}_t \not\in B(0, R)\bigr\} \;, \qquad 
\bar \tau_{R, -} \eqdef \inf\bigl\{t \geq 0 \, : \, \bar{B}_{-t} \not\in B(0, R)\bigr\}
\end{equation*}
By Theorem~\ref{th:convQuantum} and Brownian scaling, for any fixed $R > 0$, the joint law of $\bh^{\sigma_b,b}|_{B(0,R)}$ and $(\bar B_{t})_{t \in [\bar \tau_{R, -}, \bar \tau_{R, +}]}$ converges in total variation distance to the joint law of $\bh^{\star}|_{B(0,R)}$ and an independent two-sided Brownian motion stopped at its exit times from $B(0,R)$. Since, the LBM associated to $\bh^{\sigma_b, b}$ (resp.\ $\bh^{\star}$), stopped at its exit time from $B(0,R)$, is a measurable function of $\bh^{\sigma_b,b}|_{B(0,R)}$ (resp.\ $\bh^*|_{B(0,R)}$) and an independent two-sided Brownian motion stopped at its exit times from $B(0,R)$, the conclusion follows.
\end{proof}

We are finally ready to prove Theorem~\ref{th:invarianceShift}.
\begin{proof}[Proof of Theorem~\ref{th:invarianceShift}]
Let $d \geq 2$ be even and $Q > \sqrt{\smash[b]{2d}}$. For $\bh$ a whole-space LGF, we consider the associated two-sided LBM $(\rmB_{\bh, \alpha, t})_{t \in \R}$ started from the origin at time 0. Moreover, let $\tau_{\bh, \alpha, +}$ (resp.\ $\tau_{\bh, \alpha, -}$) be the first positive (resp.\ negative) time that $\rmB_{\bh, \alpha, \cdot}$ exits the ball $B(0, 1)$. Let $\rmU$ be a point sampled uniformly at random from the interval $[\tau_{\bh, \alpha, -}, \tau_{\bh, \alpha, +}]$. For $b > 0$, we let
\begin{equation*}
\sigma_{b} \eqdef \inf\bigl\{s \in \R \, : \,  \bh_{e^{-s}}(\rmB_{\bh, \alpha, \rmU}) - Q s = - b\bigr\} \;,
\end{equation*}
and we consider the conformal automorphism $\smash{\phi_b(x) =  e^{- \sigma_b} x} + \rmB_{\bh, \alpha, \rmU}$. Thanks to \cite[Corollary~3.17]{BP24}, we know that if we condition on the Brownian motion modulo time change, then the conditional law of $\bh$ near $\rmB_{\bh, \alpha, \rmU}$ locally looks like a whole-space LGF plus an $\alpha$-log singularity at $\rmB_{\bh, \alpha, \rmU}$. Therefore, by Proposition~\ref{pr:augmJointConv}, we have the following joint convergence in law as $b \to \infty$, 
\begin{equation}
\label{eq:firstTra}
\bigl(\bh \circ \phi_b + Q \log\abs{\phi'_b} + b, \, e^{\sigma_b}(\rmB_{\bh, \alpha, \rmU + e^{-\alpha b}\cdot} - \rmB_{\bh, \alpha, \rmU}) \bigr) \overset{\law}{\longrightarrow} (\bh^{\star}, \rmB_{\bh^{\star}\!, \alpha}) \;.
\end{equation}
We note that for any $t \in \R$, the total variation distance between $(\bh^{\star}\!, \rmB_{\bh^{\star}\!, \alpha}, \rmU + e^{-\alpha b} t)$ and $(\bh^{\star}\!, \rmB_{\bh^{\star}\!, \alpha}, \rmU)$ tends to $0$ as $b \to \infty$. 
This is a consequence of the fact that the conditional law of $\rmU$ given $(\bh, B)$ is that of a uniform random variable. 
Hence, by defining $\sigma_{b, t}$ as follows
\begin{equation*}
\sigma_{b, t} \eqdef \inf\bigl\{s \in \R \, : \,  \bh_{e^{-s}}(\rmB_{\bh, \alpha, \rmU + e^{-\alpha b} t}) - Q s = - b\bigr\} \;,
\end{equation*}
and by considering the conformal automorphism $\smash{\phi_{b, t}(x) =  e^{- \sigma_{b, t}} x} + \rmB_{\bh, \alpha, \rmU + e^{-\alpha b} t}$, we get that 
\begin{equation}
\label{eq:secondTra}
\bigl(\bh \circ \phi_{b, t} + Q \log\abs{\phi'_{b, t}} + b, e^{\sigma_{b, t}}(\rmB_{\bh, \alpha, \rmU + e^{-\alpha b}t + e^{-\alpha b} \cdot} - \rmB_{\bh, \alpha, \rmU + e^{-\alpha b}t})\bigr) \overset{\law}{\longrightarrow} (\bh^{\star}, \rmB_{\bh^{\star}\!, \alpha}) \;,
\end{equation}
and this it sufficient to conclude. Indeed, the couple on the left side of \eqref{eq:secondTra} is obtained from the couple on the left side of \eqref{eq:firstTra} by translating time for the LBM by $t$, translating space so that the position of the LBM at time $t$ is mapped to the origin, and re-scaling to get the usual normalisation for the field. Hence, by \eqref{eq:firstTra}, the joint law of the left sides of \eqref{eq:secondTra} and \eqref{eq:firstTra} converges to the joint law of $(\bh^{\star}, \rmB_{\bh^{\star}\!, \alpha})$ and the field/process pair obtained by translating time by $t$, translating space by $\rmB_{\bh^*\!, \alpha, t}$, and re-scaling to get the usual normalisation for the field. Therefore, the conclusion follows since by \eqref{eq:secondTra}, the second field/process pair has the same law as $(\bh^{\star}, \rmB_{\bh^{\star}\!, \alpha})$. 
\end{proof}

\appendix

\section{Some basic estimates for GMC measures}
In this appendix, we collect some basic estimates for GMC measures. We recall that for each $\gamma \in (0, \sqrt{\smash[b]{2d}})$, the collection of regularised GMC measures $(\mu_{\bh, \gamma, (n)})_{n \in \N}$ is defined in \eqref{eq:defmuN}. Furthermore, with a slight abuse of notation, we denote by $\mu_{\bh, \gamma, (\infty)}$ the limiting measure $\mu_{\bh, \gamma}$.

\begin{lemma}
\label{lm:scalingThick}
Let $d \geq 2$, $\gamma \in (0, \sqrt{\smash[b]{2d}})$, $\beta \in (-\infty , Q)$, and $x \in \R^d$. It holds $\P$-almost surely for any $r \in (0, 1]$ that 
\begin{equation*}
r^{(d + \gamma^2/2 - \beta \gamma) + \eps} \lesssim \mu_{\bh - \beta \log\abs{\cdot-x}, \gamma}\bigl(B(x, r)\bigr) \lesssim r^{(d + \gamma^2/2 - \beta \gamma) - \eps} \;, \qquad \forall \, \eps > 0 \;,
\end{equation*}
for some random implicit constants. 
\end{lemma}
\begin{proof}
The proof of this result is standard and follows from a union bound over dyadic scales, and the coordinate change formula (Proposition~\ref{pr:coordinate_change_measure}).
\end{proof}

\begin{lemma}
\label{lm:modContLQG0}
Let  $d \geq 2$, $\gamma \in (0, \sqrt{\smash[b]{2d}})$, and $R > 0$. It holds $\P$-almost surely for any $n \in \N \cup \{\infty\}$, all $x \in B(0, R)$, and any $r \in (0,1]$ that 
\begin{equation*}
\mu_{\bh, \gamma, (n)}\bigl(B(x, r)\bigr)  \lesssim r^{(d + \gamma^2/2 -\sqrt{\smash[b]{2d}}  \gamma) - \eps} \;, \qquad \forall \, \eps > 0\;,
\end{equation*}
for some random implicit constant.
\end{lemma}
\begin{proof}
For $d = 2$, this fact is proved in \cite[Theorem~2.2]{GRV_LBM}, and the proof easily generalises to any dimension $d > 2$. 
\end{proof}

\begin{lemma}
\label{lm:modContLQG}
Let  $d \geq 2$, $\gamma \in (0, \sqrt{\smash[b]{2d}})$, $\beta \in (-\infty, Q)$, and $R > 0$. It holds $\P$-almost surely for all $x \in B(0, R)$, and any $r \in (0,1]$ that 
\begin{equation*}
\mu_{\bh - \beta \log\abs{\cdot}, \gamma}\bigl(B(x, r)\bigr)  \lesssim r^{(d + \gamma^2/2 - (\beta \vee \sqrt{2d}) \gamma) - \eps} \;, \qquad \forall \, \eps > 0\;,
\end{equation*}
for some random implicit constant.
\end{lemma}
\begin{proof}
For $\beta \leq 0$ this is an immediate consequence of Lemma~\ref{lm:modContLQG0}. 
We fix $\beta \in (0, Q)$ and, for simplicity, we take $R = 1$. For $n \in \N$, let $\Sigma_{n} \eqdef 2^{-n}\Z^d \cap B(0, 1)$ be the lattice $2^{-n}\Z^d$ restricted to the ball $B(0, 1)$. For $k \in \N$, we define the annulus $A_{k} \eqdef B(0, 2^{-k}) \setminus B(0, 2^{-(k+1)})$. 
By using the coordinate change formula (Proposition~\ref{pr:coordinate_change_measure}), we have that for all $n \geq 2$, $k \in [n-2]_0$, and any $x \in \Sigma_{n} \cap A_{k}$, it holds that 
\begin{align*}
\mu_{\bh - \beta \log\abs{\cdot}, \gamma}\bigl(B(x, 2^{-n})\bigr) 
& = 2^{-k \gamma (Q-\beta)} e^{\gamma \bh_{2^{-k}}(0)} \mu_{\bh^k - \beta \log\abs{\cdot}} \bigl(B(2^k x, 2^{k-n})\bigr) \\
& \lesssim 2^{-k \gamma (Q-(\beta \vee \sqrt{\smash[b]{2d}}))} e^{\gamma \bh_{2^{-k}}(0)} \mu_{\bh^k}\bigl(B(2^k x, 2^{k-n})\bigr) \;.
\end{align*}
where, by the scale invariance of the law of $\bh$ modulo additive constant, $\bh^k(\cdot) \eqdef \bh(2^{-k} \cdot) - \bh_{2^{-k}}(0) \eqlaw \bh$. Moreover, the last inequality in the previous display follows from the fact that for all $n \geq 2$, $k \in [n-2]_0$, and any $x \in \Sigma_{n} \cap A_{k}$, the ball $B(2^k x, 2^{k-n})$ is at distance at most $1/4$ from the origin, and so the $\beta$-$\log$ singularity can be bounded above by a finite constant. Then, for any $\eps > 0$, by letting $q = (\beta \vee \sqrt{\smash[b]{2d}})/\gamma $, which is strictly less than $2d/\gamma^2$, we have that 
\begin{equation}
\label{eq:fineEstMarkov0}
\begin{alignedat}{1}
& \P\Bigl(\max_{x \in \Sigma_{n} \cap A_{k}} \mu_{\bh - \beta \log\abs{\cdot}, \gamma}\bigl(B(x, 2^{-n})\bigr) \geq 2^{-n(d+\gamma^2/2 - (\beta \vee\sqrt{\smash[b]{2d}}) \gamma - \eps)}, \; e^{\gamma \bh_{2^{-k}}(0)} \leq 2^{n \eps/2}\Bigr) \\
& \hspace{20mm} \lesssim \sum_{x \in \Sigma_{n} \cap A_{k}} 2^{(n-k)[q (d+\gamma^2/2 - (\beta \vee\sqrt{\smash[b]{2d}})\gamma)]} 2^{- q n \eps/2} \E\bigl[\mu_{\bh^k}\bigl(B(2^k x, 2^{k-n})\bigr)^q\bigr] \\
& \hspace{20mm} \leq  2^{(n-k)[q(d+\gamma^2/2 - (\beta \vee\sqrt{\smash[b]{2d}}) \gamma) + d]} 2^{- q n \eps/2} \E\bigl[\mu_{\bh^k}\bigl(B(2^k x, 2^{k-n})\bigr)^q\bigr] \\
& \hspace{20mm} \lesssim 2^{(n-k)[d + q(d+\gamma^2/2 - (\beta \vee\sqrt{\smash[b]{2d}}) \gamma) - \xi_{\gamma}(q)]} 2^{-  n q \eps/2} \leq  2^{-n q \eps/2} \;,
\end{alignedat}
\end{equation}
where we used Markov's inequality, and the fact that thanks to \cite[Theorem~3.26]{BP24}, the multifractal spectrum of GMC measures $\mu_{\bh, \gamma}$ is given by $\smash{\xi_{\gamma}(q) \eqdef (d+\gamma^2/2)q - \gamma^2 q^2/2}$. Moreover, since $\bh_{2^{-k}}(0)$ is centred Gaussian with variance of order $k$ (see Lemma~\ref{lm:varianceSpherical}), the probability that $\smash{e^{\gamma \bh_{2^{-k}}(0)} > 2^{n \eps/2}}$ is summable over $n \geq 2$ and $k \in [n-2]_0$. Therefore, combining this fact with the bound in \eqref{eq:fineEstMarkov0}, we obtain that 
\begin{equation*}
\sum_{n \geq 2} \sum_{k = 0}^{n-2} \P\Bigl(\max_{x \in \Sigma_{n} \cap A_{k}} \mu_{\bh - \beta \log\abs{\cdot}, \gamma}\bigl(B(x, 2^{-n})\bigr) \geq 2^{-n(d+\gamma^2/2 - (\beta \vee\sqrt{\smash[b]{2d}}) \gamma - \eps)}\Bigr) < \infty \;.
\end{equation*}
Hence, thanks to Borel--Cantelli lemma, we get that $\P$-almost surely, for all $n \geq 2$ and $k \in [n-2]_0$, it holds that
\begin{equation}
\label{eq:combBasicEst11}	
\max_{x \in \Sigma_{n} \cap A_{k}}\mu_{\bh - \beta \log\abs{\cdot}, \gamma}\bigl(B(x, 2^{-n})\bigr) \lesssim 2^{-n(d+\gamma^2/2 - (\beta \vee\sqrt{\smash[b]{2d}}) \gamma - \eps)} \;,	
\end{equation}
for some random implicit constant. Thanks to Lemma~\ref{lm:scalingThick}, we also have that $\P$-almost surely, for all $n \in \N_0$, it holds that 
\begin{equation}
\label{eq:combBasicEst12}	
\mu_{\bh - \beta \log\abs{\cdot}, \gamma}\bigl(B(0, 2^{-n})\bigr) \lesssim 2^{-n(d+\gamma^2/2 - (\beta \vee\sqrt{\smash[b]{2d}}) \gamma - \eps)} \;,	
\end{equation}
Therefore, the conclusion follows by combining \eqref{eq:combBasicEst11} and \eqref{eq:combBasicEst12} and to the fact that for any $r \in (0, 1]$, and $x \in B(0, 1)$, there exist $n \in \N_0$ such that $2^{-(n+1)} < r \leq 2^{-n}$, and $B(x, r)$ is contained in a finite numbers of balls of the form $B(z, 2^{-n})$ with $z \in \Sigma_{n}$.
\end{proof}

\begin{lemma}
\label{lm:mainBoundTechNested}
Let $d \geq 2$ and $\gamma \in (0, \sqrt{\smash[b]{2d}})$. It holds $\P$-almost surely for any $s \in (0,1]$, any $r \in [0, s]$, and all $x \in B(0, s)$ that 
\begin{equation*}
\mu_{\bh, \gamma}\bigl(B(x, r)\bigr) \lesssim s^{\sqrt{\smash[b]{2d}}\gamma - \eps} r^{(d+\gamma^2/2 - \sqrt{\smash[b]{2d}}\gamma) - \eps}   \;, \qquad \forall \, \eps > 0 \;,
\end{equation*}
for some random implicit constant.
\end{lemma} 
\begin{proof}
For $k \leq n \in \N_0$, by using the coordinate change formula (Proposition~\ref{pr:coordinate_change_measure}), we have that for all $x \in B(0, 2^{-k})$, it holds that 
\begin{equation*}
\mu_{\bh, \gamma}\bigl(B(x, 2^{-n})\bigr) 
\eqlaw 2^{-k \gamma Q} e^{\gamma \bh_{2^{-k}}(0)} \mu_{\bh^k}\bigl(B(2^k x, 2^{k-n})\bigr) \;,
\end{equation*} 
where, by the scale invariance of the law of $\bh$ modulo additive constant, $\bh^k(\cdot) \eqdef \bh(2^{-k} \cdot) - \bh_{2^{-k}}(0) \eqlaw \bh$. We let $\smash{\Sigma_{k, n} \eqdef2^{-n}\Z^d \cap B(0, 2^{-k})}$ be the lattice $2^{-n}\Z^d$ restricted to the ball $B(0, 2^{-k})$. For any $\eps > 0$, by letting $q = \sqrt{\smash[b]{2d}}/\gamma$, we have that 
\begin{equation}
\label{eq:fineEstMarkov}
\begin{alignedat}{1}
& \P\Bigl(\max_{x \in \Sigma_{k, n}} \mu_{\bh, \gamma}\bigl(B(x, 2^{-n})\bigr) \geq 2^{-k(\sqrt{\smash[b]{2d}} \gamma - \eps)} 2^{-n(d+\gamma^2/2 - \sqrt{\smash[b]{2d}}\gamma - \eps)}, \; e^{\gamma \bh_{2^{-k}}(0)} \leq 2^{(k+n) \eps/2}\Bigr) \\
& \hspace{20mm} \leq \P\Bigl(\max_{x \in \Sigma_{k, n}}  \mu_{\bh^k, \gamma}\bigl(B(2^k x, 2^{k-n})\bigr) \geq 2^{(k-n) (d+\gamma^2/2 - \sqrt{\smash[b]{2d}}\gamma) + (k+n)\eps/2}\Bigr)	\\
& \hspace{20mm} \leq \sum_{x \in \Sigma_{k, n}} 2^{-(k+n) q \eps/2} 2^{(n-k) (d+\gamma^2/2 - \sqrt{\smash[b]{2d}}\gamma) + (k+n)\eps/2}  \E\bigl[\mu_{\bh^k, \gamma}\bigl(B(2^k x, 2^{k-n})\bigr)^q \bigr]\\
& \hspace{20mm} \leq 2^{-(k+n) q \eps/2} 2^{(n-k)[d + q (d+\gamma^2/2 - \sqrt{\smash[b]{2d}}\gamma) - \xi_{\gamma}(q)]} = 2^{-(k+n) q \eps/2} \;,
\end{alignedat}
\end{equation}
where we used Markov's inequality, and the fact that thanks to \cite[Theorem~3.26]{BP24}, the multifractal spectrum of GMC measures $\mu_{\bh, \gamma}$ is given by $\smash{\xi_{\gamma}(q) \eqdef (d+\gamma^2/2)q - \gamma^2 q^2/2}$. Moreover, since $\bh_{2^{-k}}(0)$ is centred Gaussian with variance of order $k$ (see Lemma~\ref{lm:varianceSpherical}), the probability that $\smash{e^{\gamma \bh_{2^{-k}}(0)} > 2^{(k+n) \eps/2}}$ is summable over $k \leq n \in \N_0$. Therefore, combining this fact with the bound in \eqref{eq:fineEstMarkov}, we get that 
\begin{equation*}
\sum_{k = 0}^{\infty} \sum_{n = k}^{\infty} \P\Bigl(\max_{x \in \Sigma_{k, n}} \mu_{\bh, \gamma}\bigl(B(x, 2^{-n})\bigr) \geq 2^{-k(\sqrt{\smash[b]{2d}} \gamma - \eps)} 2^{-n(d+\gamma^2/2 - \sqrt{\smash[b]{2d}}\gamma - \eps)}\Bigr) < \infty \;.
\end{equation*}
Hence, thanks to Borel--Cantelli lemma, we get that $\P$-almost surely, for all $k \leq n \in \N_0$, it holds that 
\begin{equation*}
\max_{x \in \Sigma_{k, n}} \mu_{\bh, \gamma}\bigl(B(x, 2^{-n})\bigr) \lesssim 2^{-k(\sqrt{\smash[b]{2d}} \gamma - \eps)} 2^{-n(d+\gamma^2/2 - \sqrt{\smash[b]{2d}}\gamma - \eps)} \;, 
\end{equation*}
for some random implicit constant. This is sufficient to conclude since for any $s \in (0, 1]$, $r \in (0, s]$, and $x \in B(0, s)$, there exist $k \leq n \in \N_0$ such that $2^{-(k+1)} < s \leq 2^{-k}$, $2^{-(n+1)} < r \leq 2^{-(n+1)}$, and $B(x, r)$ is contained in a finite numbers of balls of the form $B(z, 2^{-n})$ with $z \in \Sigma_{n, k}$.
\end{proof}

\section{Proof of Lemma~\ref{lm:keySpec}}
\label{ap:keySpec}
Throughout this appendix, we fix $\beta \in (-\infty, Q)$, and we define the field
\begin{equation*}
\bar \bh \eqdef \bh - \beta \log\abs{\cdot} \;. 
\end{equation*}

For $x_1, \ldots, x_n \in \R^d$ and $0 < s_1 < \ldots < s_n \in \R_{0}^{+}$, we denote by $p_{s_1, \ldots, s_n}(0, x_1, \ldots, x_n)$ the transition probability density of the $d$-dimensional Brownian motion started from $0$, i.e., 
\begin{equation*}
\BP_{0}(B_{s_1} \in dx_1, \; \ldots B_{s_n} \in dx_n) = p_{s_1, \ldots, s_n}(0, x_1, \ldots, x_n) dx_1 \cdots d x_n \;.
\end{equation*}
Furthermore, for $t \geq 0$ and $x \in \R^d$ we define the function $\GG_{t}(x, \cdot) : \R^d \to \R$ by letting 
\begin{equation}
\label{eq:defftx}
\GG_{t}(x, y) \eqdef \int_{0}^{t}  p_{s}(x, y) ds \;.
\end{equation}

\begin{lemma}
\label{lm:BasicDecoIter}
For $d > 2$, $Q > \sqrt{\smash[b]{2d}}$, $n \in \N$, $t \geq 0$, and $\delta >0$ it holds $\P$-almost surely that 
\begin{equation}
\label{eq:boundChiMoment}
\BE_{0}\bigl[(\rmF_{\bar\bh,\alpha}(t))^n \mathbbm{1}_{\{[B]_{t} \subseteq B(0, t^{1/2- \delta})\}} \bigr] \lesssim \int_{(B(0, t^{1/2 - \delta}))^n} \bigl(\GG_{t}(0, x_1) \cdots 
\GG_{t}(x_{n-1}, x_n)\bigr) \mu_{\bar\bh, \alpha}(dx_n)\cdots\mu_{\bar\bh, \alpha}(dx_1) \;,
\end{equation}
where the implicit constant only depends on $n$. 
\end{lemma}
\begin{proof}
By an immediate generalisation of Lemma~\ref{lm:StrenghRevuz}, we observe that for all $t \geq 0$, the left-hand side of \eqref{eq:boundChiMoment} is bounded above by a constant (only depending on $n$) times
\begin{equation*}
\int_{(B(0, t^{1/2 - \delta}))^n} \Biggl(\int_{0}^{t} \int_{s_1}^{t} \cdots \int_{s_{n-1}}^{t}  p_{s_1, \ldots, s_n}(0, x_1, \ldots, x_n) ds_n \cdots ds_1\Biggr)\mu_{\bar \bh,\alpha}(dx_n)\cdots\mu_{\bar\bh, \alpha}(dx_1)\;.
\end{equation*}
For $0 < s_1 < \ldots < s_n$, we have that the following decomposition holds 
\begin{equation*}
p_{s_1, \ldots, s_n}(0, x_1, \ldots, x_n) = p_{s_1}(0, x_1) p_{s_2 - s_1}(x_1, x_2) \cdots p_{s_n - s_{n-1}}(x_{n-1}, x_n) \;,
\end{equation*}
and so, by performing the change of variables $s_j \mapsto s_j - s_{j-1}$ for all $j \in [n]$, the desired result follows readily.
\end{proof}

Thanks to Lemma~\ref{lm:BasicDecoIter}, in order to prove Lemma~\ref{lm:keySpec}, is suffices to estimate the quantity on the right-hand of \eqref{eq:boundChiMoment}. We start with the following lemma.
\begin{lemma}
\label{lm:boundftxy}
For $t \geq 0$ and $x\in \R^d$, consider the function $\GG_{t}(x, \cdot) :\R^d \to \R$ defined in \eqref{eq:defftx}. Then, it holds that  
\begin{equation}
\label{eq:trivialBoundGreen}
\GG_{t}(x, y) \lesssim \abs{y-x}^{2-d} \;, \qquad \forall \, x, y \in \R^d \;.
\end{equation} 
\end{lemma}
\begin{proof}
For $t \geq 0$ and $x$, $y \in \R^d$, we have that 
\begin{equation*}
\GG_{t}(x, y) \leq \GG_{\infty}(x, y) = \frac{\Gamma(d/2 -1)}{2 \pi^{d/2}}\abs{y-x}^{2-d} \;,
\end{equation*}
where the last equality in the above display follows from a simple computation (see e.g.\ \cite[Theorem~3.33]{Peres}).
\end{proof}

Now, combining \eqref{eq:boundChiMoment} with \eqref{eq:trivialBoundGreen}, it suffices to estimate and get the ``right asymptotic'' as $t \to 0$ of the following quantity
\begin{equation*}
\int_{B(0, t^{1/2-\delta})} \abs{x_1}^{2-d} \cdots  \int_{B(0, t^{1/2-\delta})} \abs{x_n - x_{n-1}}^{2-d} \mu_{\bar\bh, \alpha}(dx_n)\cdots\mu_{\bar\bh, \alpha}(dx_1)\;.
\end{equation*}
To this end, we start with the following lemma.
\begin{lemma}
\label{lm:boundSupMass}
Let $d > 2$, $Q > \sqrt{\smash[b]{2d}}$, and $\beta \in (-\infty, Q)$. It holds $\P$-almost surely for any $r \in (0, 1]$ and all $x \in B(0, r)$ that  
\begin{equation}
\label{eq:boundSup}
\int_{B(0, r)} \abs{x-y}^{2-d} \mu_{\bar\bh, \alpha}(dy) \lesssim r^{(2 + \alpha^2/2 - \alpha \beta) - \eps} \;, \qquad \forall \, \eps > 0 \;,
\end{equation}
for some random implicit constant. 
\end{lemma}
\begin{proof}
For $\beta < 0$ the result is a straightforward consequence of Lemma~\ref{lm:scalingThick}. Hence, we can focus on the case $\beta \in (0, Q)$.
In what follows, for $x \in \R^d$ and $r_2 > r_1 \geq 0$, we write $A(x, r_2, r_1)$ as a shorthand for the annulus $B(x, r_2) \setminus B(x, r_1)$. 
For $r \in (0, 1]$, we fix a point $x \in B(0,r)$ and we split the ball $B(0, r)$ into four different regions as depicted in Figure~\ref{fig:balls}.
\begin{figure}[h]
\centering
\includegraphics[scale=0.6]{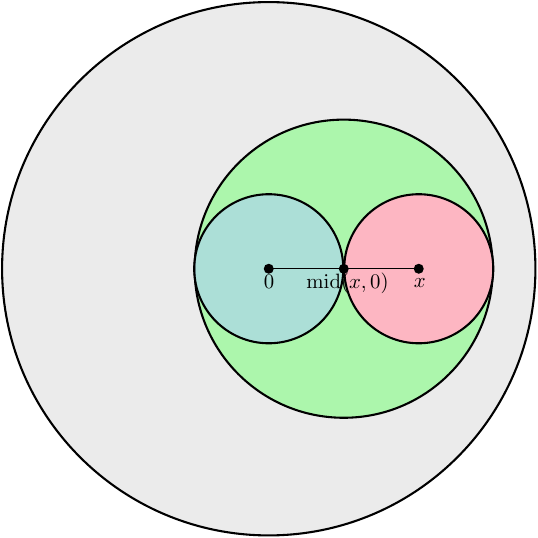}
\caption{\small Subdivision of the ball $B(0, r)$ into four different regions. The first region consists in the blue ball $B(0, \abs{x}/2)$. The second region consists in the pink ball $B(x, \abs{x}/2)$. The third region consists in the green set $B(\midpoint(0, x), 2 \abs{x}) \setminus B(0, \abs{x}/2) \setminus B(x, \abs{x}/2)$. The fourth region consists in the grey set $B(0, r) \setminus B(\midpoint(0, x), 2 \abs{x})$.}
\label{fig:balls}
\end{figure}

\textbf{Region 1:} The first region consists in the blue ball $B(0, \abs{x}/2)$. In this case, using Lemma~\ref{lm:scalingThick}, for any $\eps > 0$ small enough, we have $\P$-almost surely that  
\begin{equation*}
\int_{B(0, \abs{x}/2)} \abs{x - y}^{2-d} \mu_{\bar\bh, \alpha}(dy) 
\lesssim \abs{x}^{2-d} \mu_{\bar\bh, \alpha}\bigl(B(0, \abs{x}/2)\bigr) \lesssim \abs{x}^{2-d} \abs{x}^{(d + \alpha^2/2 - \alpha \beta) - \eps} 
\lesssim r^{(2 + \alpha^2/2 - \alpha \beta) - \eps}  \;.
\end{equation*}

\textbf{Region 2:} The second region consists in the pink ball $B(x, \abs{x}/2)$. In this case, using the fact that all points in this region are at distance at least $\abs{x}/2$ from $0$, and using Lemma~\ref{lm:mainBoundTechNested}, for any $\eps > 0$ small enough, we obtain $\P$-almost surely that 
\begin{align*}
\int_{B(x, \abs{x}/2)} \abs{x - y}^{2-d} \mu_{\bar\bh, \alpha}(dy)	
& = \sum_{k = 0}^{\infty} \int_{A(x, e^{-k} \abs{x}/2, e^{-(k+1)} \abs{x}/2)} \abs{x - y}^{2-d} \mu_{\bar\bh, \alpha}(dy)	\\
& \lesssim \sum_{k = 0}^{\infty} (e^{-k} \abs{x})^{2-d} \mu_{\bar\bh, \alpha}\bigl(B(x, e^{-k} \abs{x}/2)\bigr) \\
& \lesssim \abs{x}^{2-d-\alpha\beta}\sum_{k = 0}^{\infty} e^{-k(2-d)} \mu_{\bh, \alpha}\bigl(B(x, e^{-k} \abs{x}/2)\bigr) \\
& \lesssim \abs{x}^{(2 + \alpha^2/2 - \alpha \beta) - \eps} \sum_{k = 0}^{\infty} e^{-k[(d+ \alpha^2/2 - \sqrt{\smash[b]{2d}}\alpha)-\eps]}\\
& \lesssim \abs{x}^{(2+\alpha^2/2 - \alpha \beta) - \eps} \leq r^{(2+\alpha^2/2 - \alpha \beta) - \eps} \;,
\end{align*} 
where here we used the fact that $d+\alpha^2/2-\sqrt{\smash[b]{2d}}\alpha > 0$, for all $\alpha \in (0, 2)$.

\textbf{Region 3:} The third region consists in the green set $B(\midpoint(0, x), 2 \abs{x}) \setminus B(0, \abs{x}/2) \setminus B(x, \abs{x}/2)$. In this case, using the fact that all points in this region are at distance at least $\abs{x}/2$ from $x$, and using Lemma~\ref{lm:scalingThick}, for any $\eps > 0$ small enough, we have $\P$-almost surely that 
\begin{align*}
\int_{B(\midpoint(0, x), 2\abs{x}) \setminus B(0, \abs{x}/2) \setminus B(x, \abs{x}/2)} \abs{x - y}^{2-d} \mu_{\bar\bh, \alpha}(dy) 
\lesssim \abs{x}^{2-d} \mu_{\bar\bh, \alpha}\bigl(B(0, 2 \abs{x})\bigr)  
& \lesssim \abs{x}^{2-d} \abs{x}^{(d+\alpha^2/2 - \alpha\beta) - \eps} \\
& \leq r^{(2+\alpha^2/2 - \alpha \beta) - \eps} \;.
\end{align*}

\textbf{Region 4:} The fourth region consists in the grey set $B(0, r) \setminus B(\midpoint(0, x), 2 \abs{x})$. In this case, using Lemma~\ref{lm:scalingThick}, we get that, for any $\eps > 0$ small enough, it holds $\P$-almost surely that 
\begin{align*}
\int_{B(0, r) \setminus B(\midpoint(0, x), 2 \abs{x})} \abs{x - y}^{2-d} \mu_{\bar\bh, \alpha}(dy)	
& \lesssim  \sum_{k=0}^{-\log(\abs{x}/(2r)) - 1} \int_{A(x, 2 e^{-k} r, 2 e^{-(k+1)} r)} \abs{x - y}^{2-d} \mu_{\bar\bh, \alpha}(dy) \\
& \leq \sum_{k=0}^{-\log(\abs{x}/(2r)) - 1} (e^{-k} r)^{2-d} \mu_{\bar\bh, \alpha}\bigl(B(0, \abs{x} + 2 e^{-k} r)\bigr)  \\
& \lesssim \sum_{k=0}^{-\log(\abs{x}/(2r)) - 1} (e^{-k} r)^{2-d} (\abs{x} + 2 e^{-k} r)^{(d+\alpha^2/2 - \alpha \beta) - \eps}   \\
& \lesssim \sum_{k=0}^{-\log(\abs{x}/(2r)) - 1} (e^{-k} r)^{2-d} (e^{-k} r)^{(d+\alpha^2/2 - \alpha \beta) - \eps} \\
& \lesssim r^{(d+\alpha^2/2-\alpha\beta) - \eps} \;,  
\end{align*}
where to get the penultimate inequality, we used the fact that $\abs{x} \lesssim e^{-k} r$ for all $k \leq -\log(\abs{x}/(2r)) - 1$, and in the last inequality we used the fact that $2 + \alpha^2/2-\alpha\beta > 0$, for all $\beta \in (-\infty, Q)$.

Therefore, the conclusion follows by the arbitrariness of $x \in B(0, r)$ and by combining the four bounds in the different regions.
\end{proof}

We are now ready to prove Lemma~\ref{lm:keySpec}.
\begin{proof}[Proof of Lemma~\ref{lm:keySpec}]
Let $d \geq 2$, $Q > \sqrt{\smash[b]{2d}}$, and $\beta \in (-\infty, Q)$. For all $n \in \N$, $\eps > 0$, $\delta > 0$, and $t \in (0, 1]$, we have $\P$-almost surely that 
\begin{align}
\label{eq:proofMainTechSpecLast}	
& \BE_{0}\bigl[(\rmF_{\bar \bh, \alpha}(t))^n \mathbbm{1}_{\{[B]_{t} \subseteq B(0, t^{1/2- \delta})\}}\bigr] \nonumber\\
& \hspace{20mm}  \overset{\eqref{eq:boundChiMoment}}{\lesssim} \int_{(B(0, t^{1/2 - \delta}))^n} \bigl(\GG_{t}(0, x_1) \cdots 
\GG_{t}(x_{n-1}, x_n)\bigr) \mu_{\bar\bh, \alpha}(dx_n)\cdots\mu_{\bar\bh, \alpha}(dx_1) \nonumber \\
& \hspace{20mm} \overset{\eqref{eq:trivialBoundGreen}}{\lesssim} \int_{B(0, t^{1/2-\delta})} \abs{x_1}^{2-d} \cdots  \int_{B(0, t^{1/2-\delta})} \abs{x_n - x_{n-1}}^{2-d} \mu_{\bar \bh, \alpha}(dx_n)\cdots\mu_{\bar\bh, \alpha}(dx_1) \;.
\end{align}
Thanks to Lemma~\ref{lm:boundSupMass}, since $x_{n-1}$ is sampled inside the ball $B(0, t^{1/2-\delta})$, for any $\eps > 0$, the inner most integral in \eqref{eq:proofMainTechSpecLast} can be bounded above $\P$-almost surely by
\begin{equation*}
\sup_{x \in B(0, t^{1/2-\delta})} \int_{B(0, t^{1/2-\delta})} \abs{x_n - x}^{2-d} \mu_{\bar \bh, \alpha}(dx_n) \lesssim t^{(1/2 - \delta)(2 + \alpha^2/2 -\alpha \beta) - \eps} \;,
\end{equation*}
for some  random implicit constant. By applying the same techniques $n-1$ more times to bound the remaining integrals in \eqref{eq:proofMainTechSpecLast}, the desired conclusion follows. 
\end{proof}

\small
\bibliography{ref}{}
\bibliographystyle{Martin}
\end{document}